\documentclass{article}

\usepackage[english]{babel}

\usepackage{amsmath,amssymb,amsthm}

\usepackage{xcolor}

\usepackage{bbm}

\usepackage{tikz}
\usetikzlibrary{arrows.meta,calc}

\usepackage[font=small]{caption}

\usepackage[
a4paper,
textwidth=17cm,
textheight=23.25cm,
centering
]{geometry}

\usepackage{hyperref}

\newcommand{\id}{\mathbbm{1}}

\DeclareMathOperator{\tr}{tr}
\DeclareMathOperator{\sym}{sym}
\DeclareMathOperator{\Cof}{Cof}

\newcommand{\dd}{\displaystyle}

\newtheorem{theorem}{Theorem}[section]
\newtheorem{lemma}[theorem]{Lemma}
\newtheorem{proposition}[theorem]{Proposition}

\newtheorem{remark}[theorem]{Remark}

\numberwithin{equation}{section}

\allowdisplaybreaks[1]	

\title{Nonlinear Kirchhoff--Love shell models derived from the Ciarlet--Geymonat energy: modelling and {existence of minimizers}}

\begin{document}
	
	\author{Ionel-Dumitrel Ghiba\thanks{Corresponding author: 
			Ionel-Dumitrel Ghiba, \ \  Department of Mathematics, Alexandru Ioan Cuza University of Ia\c si,  Blvd.
			Carol I, no. 11, 700506 Ia\c si,
			Romania; Octav Mayer Institute of Mathematics of the
			Romanian Academy, Ia\c si Branch,  700505 Ia\c si, email:  dumitrel.ghiba@uaic.ro} \quad and \quad  Trung Hieu Giang
		\thanks{Trung Hieu Giang,  \ \ Department of Mathematical Analysis, Faculty of Mathematics and Physics, Charles University, Praha, Czech Republic;  Institute of Mathematics, Vietnam Academy of Science and Technology, 18 Hoang Quoc Viet, Hanoi, Vietnam, email: trung-hieu.giang@matfyz.cuni.cz}\quad and \quad  C\u at\u alina Ureche\thanks{ C\u at\u alina Ureche, \ \ Department of Mathematics, Alexandru Ioan Cuza University of Ia\c si,  Blvd.
			Carol I, no. 11, 700506 Ia\c si,
			Romania, email:  catalinaureche997@gmail.com }
	}

	\maketitle
	\begin{abstract}
		Starting from a three-dimensional model based on the Ciarlet--Geymonat energy, we derive nonlinear shell models within the classical elasticity theory of compressible isotropic materials. The Neo-Hookean term involving the norm of the deformation gradient leads to an energy depending on the first, the second, and the third fundamental forms of the deformed midsurface. The coefficients appearing in the resulting shell models depend on the classical Lam\'e coefficients of the three-dimensional material, on the thickness of the shell, and on the mean and Gaussian curvatures of the reference configuration. This shows that the behavior of the shell is influenced not only by the elastic coefficients but also by the initial geometry of the three-dimensional thin body. The purely volumetric Ciarlet--Geymonat contribution of the three-dimensional energy leads to two-dimensional energies depending on the mean and Gaussian curvatures of both configurations, namely the undeformed and the deformed midsurfaces.  {Since a purely asymptotic derivation may lead to nonlinear terms for which the lower semicontinuity of the resulting functionals is not clear, we combine the asymptotic reduction through the thickness with positive-weight quadrature rules for selected purely volumetric contributions}. After deriving the {models},  we establish {variational well-posedness in the sense of existence of minimizers. More precisely, we prove coercivity and weak lower semicontinuity of the reduced functionals and establish the existence of minimizers in appropriate Sobolev spaces. A key ingredient is a polyconvexity concept for shells together with compensated compactness results identifying the weak limits of the area-weighted mean and Gaussian curvatures. An important consequence of {the convexity analysis} is that Models I and II require no additional restriction on the thickness beyond the local geometric regularity condition ensuring the regularity of the shell parametrization. Only Model III, in which the quadratic volumetric contribution is treated by a Taylor expansion, requires an additional explicit thickness condition in our existence result.}
	\end{abstract}
	
	\begin{footnotesize}
		\tableofcontents
	\end{footnotesize}

	\section{Introduction}

	The aim of the present paper is to derive well-posed nonlinear shell models {(in the sense of existence of minimizers)} starting from the three-dimensional Ciarlet--Geymonat energy. All the models in this article are formulated in a variational setting. The resulting shell models are formulated in terms of the first, second, and third fundamental forms of the deformed midsurface, as well as the curvatures of the deformed midsurface. This allows us to prove the existence of minimizers for the derived models.
	
	We start with a three-dimensional variational problem that describes the equilibrium of a {\it shell-like thin domain} $\Omega_\xi\subset\mathbb{R}^3.$ A generic point of $\Omega_\xi$ will be denoted by $(\xi_1,\xi_2,\xi_3)$ on a fixed standard basis $e_1, e_2, e_3$ of $\mathbb{R}^3$. The deformation of the body occupying the domain $\Omega_\xi$ is described by a map 
	$
	\varphi_\xi:\Omega_\xi\subset\mathbb{R}^3\rightarrow\mathbb{R}^3 
	$ (\textit{called the deformation}). The deformation of the domain $\Omega_\xi$ is not considered a purely geometric transformation, but rather the deformation of an elastic body that occupies this domain and is composed of a homogeneous and isotropic elastic material. Moreover, the reference configuration $\Omega_\xi$ is assumed to be a natural (stress-free) state. We denote the current configuration by $\Omega_c:=\varphi_\xi(\Omega_\xi)\subset\mathbb{R}^3$. The gradient of the deformation will be denoted by\footnote{Here, given $z_1,z_2,z_3\in \mathbb{R}^{n\times k}$, the notation $(z_1\,|\,z_2\,|\,z_3)$ means a matrix $Z\in \mathbb{R}^{n\times 3k}$ obtained by concatenating \(z_1,z_2,z_3\) horizontally. The gradient matrix is defined by $
		\nabla_\xi\varphi_\xi:=\left(
		\nabla_\xi\,  \varphi_1\,|\,
		\nabla_\xi\, \varphi_2\,|\,
		\nabla_\xi\, \varphi_3
		\right)^T\equiv(\partial_{\xi_1}\varphi\,|\,\partial_{\xi_2}\varphi\,|\,\partial_{\xi_3}\varphi)\,{\in}\mathbb{R}^{3\times 3}$.} $F_\xi=\nabla_\xi\varphi_\xi\,\,{\in}\,\mathbb{R}^{3\times 3}$. We will assume that the deformation $\varphi_\xi$ is orientation-preserving, i.e., $F_\xi\in {\rm GL}^+(3):=\{X\in\,{\mathbb{R}^{3\times 3}}\;|\det({X})> 0\}$.
	
	There are many possible energies that have been used in nonlinear elasticity. This is somehow normal, as long as there is no energy able to capture simultaneously all the physical expectations and the mathematical properties that assure the well-posedness of the problem. 
	\,{Our shell models are constructed from a {parent} three-dimensional model based on a classical compressible polyconvex Neo-Hooke-type energy (Hadamard materials) \cite{Antman95,Ciarlet88}, i.e., the Ciarlet--Geymonat-type energy \cite{ciarlet1982quelques,ciarlet1982lois} for compressible materials given by 
		\begin{align}\label{ine}
			W_{\rm CG} (F_\xi)=\frac{\mu }{2}\left(|F_\xi|^2-3\right)+ {\mathcal{H}}_{\rm CG}(\det F_\xi) \qquad \forall F_\xi\in {\rm GL}^+(3),
	\end{align}}
	where  $|\cdot|$ is the Frobenius tensor norm\footnote{The standard Euclidean scalar product on $\mathbb{R}^{n\times n}$ is given by
		$ \bigl\langle  {X},{Y}\bigr\rangle _{\mathbb{R}^{n\times n}}={\rm tr}(X^T Y)$, and thus the
		{(squared)}
		Frobenius tensor norm is
		$|{X}|^2= \bigl\langle  {X},{X}\bigr\rangle _{\mathbb{R}^{n\times n}}$.} and \begin{align}\label{hcg}
		{\mathcal{H}}_{\rm CG}: (0,\infty)\to \mathbb{R}, \qquad \qquad {\mathcal{H}}_{\rm CG}(x)=-\mu \,\log x+\frac{\lambda}{4}(x^2-2\, \log x-1), 
	\end{align}
	with given positive constitutive parameters (Lam\'e coefficients) $\mu $ and $\lambda$ of the elastic materials (known from experiments in the three-dimensional space).\footnote{For small deformations, i.e., for
		$
		F_\xi = \id_3 + H,
		\  |H| \ll 1 
		$, the following quadratic approximation holds true $W_{\rm CG} (F_\xi):=\frac{\mu }{2}\left(|F_\xi|^2-3\right)+ {\mathcal{H}}_{\rm CG}(\det F_\xi) 
		=
		\mu\,|\mathrm{sym}\,H|^2
		+
		\frac{\lambda}{2}[\mathrm{tr}H]^2
		+
		O(|H|^3)$. Hence,  the small deformation limit of the Ciarlet--Geymonat model reduces to the classical linear elasticity energy. Moreover, the first Piola--Kirchhoff stress tensor vanishes at $F_\xi=\id_3$. Hence, the undeformed configuration carries no stress, which is a physically required property for elastic materials.} Throughout the entire article, we shall refer to this model as the Ciarlet--Geymonat three-dimensional elastic model.  {The same reduction strategy may be applied to other three-dimensional energies provided that, after dimensional reduction or positive-weight quadrature, the resulting integrand remains convex in the shell variables for which the required weak compactness is available. A dependence on $\Cof F_\xi$, for instance through a polyconvex term
		such as $|\Cof F_\xi|^r$, would produce additional reduced contributions
		and additional surface variables. Their weak continuity and the convexity
		of the corresponding two-dimensional density would have to be analyzed
		separately. }
	
	In this paper, we assume that the reference configuration is a shell, i.e., it is the range of the following diffeomorphism $\Theta:\Omega_h\subset\mathbb{R}^3\rightarrow\mathbb{R}^3$ 
	\begin{equation}\label{defTheta}
		\Theta(x_1,x_2,x_3)\,=\,{y_0}(x_1,x_2)+x_3\ n_{y_0}(x_1,x_2), \ \ \ \ \ \ \ \ \quad  n_{y_0}\,=\,\dd\frac{\partial_{x_1}{y_0}\times \partial_{x_2}{y_0}}{\lVert \partial_{x_1}{y_0}\times \partial_{x_2}{y_0}\rVert}\, ,
	\end{equation}
	where
	\begin{equation}\Omega_h=\left\{ x=(x_1,x_2,x_3) \,\Big| \, x'=(x_1,x_2)\in\omega,  \, -\dfrac{h}{2}\,< x_3<\, \dfrac{h}{2}\, \right\} = \,\dd\omega\,\times\left(-\frac{h}{2}, \frac{h}{2}\right),\end{equation}
	with $\omega\subset\mathbb{R}^2$ a bounded domain with Lipschitz boundary
	$\partial \omega$, the constant length $h>0$ is the \textit{thickness of the shell}, and
	${y_0}:\overline{\omega}\to \mathbb{R}^3$ is a regular function of class $C^2(\overline{\omega})$ describing a \textit{surface {embedded} in the
		three-dimensional space}.
	\,{For shell-like bodies, \(h\) is regarded as a small geometric
		parameter, although no limit \(h\to0\) is taken in the present analysis.
		Throughout the paper we assume that \(\Theta\) is a diffeomorphism. In
		addition, we impose the local geometric regularity condition
		\begin{align}\label{ch5in0}
			h\,\max \Big\{
			\sup_{x'\in\bar\omega}|\kappa_{1,y_0}|,
			\sup_{x'\in\bar\omega}|\kappa_{2,y_0}|
			\Big\}<2.
		\end{align}
		This condition guarantees that
		\(\det\nabla_x\Theta>0\) throughout the thickness and excludes the local
		degeneracy of the normal parametrization; it is not
		intended to guarantee the global injectivity of \(\Theta\).
		As shown later, Models I and II require no additional smallness condition
		for the existence theory, whereas Model III requires one further explicit
		thickness condition.}
	
	\begin{figure}[h!]
		
		\begin{tikzpicture}[>=Stealth, line cap=round, line join=round]
			
			\definecolor{mygreen}{RGB}{0,190,0}
			\definecolor{myred}{RGB}{220,25,25}
			\definecolor{myblue}{RGB}{25,35,220}
			\definecolor{mygreygreen}{RGB}{133,146,143}
			\definecolor{mygreyred}{RGB}{132,119,119}
			\coordinate (O) at (0.3,3.2);
			\fill (O) circle (1.2pt);
			\draw[->] (O)--++(0,0.7) node[above] {$e_3$};
			\draw[->] (O)--++(0.9,-0.03) node[right] {$e_2$};
			\draw[->] (O)--++(-0.25,-0.58) node[below] {$e_1$};
			
			\begin{scope}[shift={(1.8,0.65)}]
				
				\coordinate (A) at (0.00,0.00);
				\coordinate (B) at (3.00,0.00);
				\coordinate (C) at (3.58,0.52);
				\coordinate (D) at (0.58,0.52);
				
				\coordinate (t) at (0.10,-0.25);
				
				\coordinate (A2) at ($(A)+(t)$);
				\coordinate (B2) at ($(B)+(t)$);
				\coordinate (C2) at ($(C)+(t)$);
				\coordinate (D2) at ($(D)+(t)$);
				
				\coordinate (Am) at ($(A)!0.5!(A2)$);
				\coordinate (Bm) at ($(B)!0.5!(B2)$);
				\coordinate (Cm) at ($(C)!0.5!(C2)$);
				\coordinate (Dm) at ($(D)!0.5!(D2)$);
				
				\fill[mygreen!10](A)--(B)--(C)--(D)--cycle;
				\fill[mygreen!6](A2)--(B2)--(C2)--(D2)--cycle;
				\fill[mygreen!8](B)--(C)--(C2)--(B2)--cycle;
				\fill[mygreen!8](C)--(D)--(D2)--(C2)--cycle;
				\fill[mygreen!8](A)--(D)--(D2)--(A2)--cycle;
				
				\draw[mygreen,thick] (A)--(B)--(C)--(D)--cycle;
				\draw[mygreen,thick] (A2)--(B2)--(C2);
				\draw[mygreen,thick, dashed] (C2)--(D2)--(A2);
				\draw[mygreen,thick] (A)--(A2) (B)--(B2) (C)--(C2);
				\draw[mygreen,thick, dashed] (D)--(D2);
				\draw[mygreygreen,thick]	(Am)--(Bm)--(Cm);
				\draw[mygreygreen,thick,dashed]	(Cm)--(Dm)--(Am);
				
				\coordinate (x) at ($(Am)!0.5!(Cm)$);
				\fill[black] (x) circle (1.1pt);
				\node[mygreen,anchor=east,xshift=2pt,font=\footnotesize] at (x) {$x$};
				\draw[black, line width=0.5pt, ->, >=stealth] (x) -- ($(x)+(0.42,0)$);
				\draw[black, line width=0.5pt,->,>=stealth]	(x) -- ($(x)+(0.22,0.18)$);
				
				\draw[black,thick] ($(C)+(0.22,0.10)$) -- ($(C2)+(0.22,0.10)$);
				\node[right,font=\footnotesize] at ($(C)!0.5!(C2)+(0.24,0.10)$) {$h$};
				
				\node[mygreen] at (-0.2,-0.3)	{$\Omega_h$};
				
			\end{scope}

			\begin{scope}[shift={(6.2,3.5)}]
				
				\coordinate (FU1) at (0.00,0.32);
				\coordinate (FU2) at (3.30,0.32);
				\coordinate (BU1) at (0.55,0.92);
				\coordinate (BU2) at (3.85,0.92);
				
				\coordinate (dt) at	(0.15,-0.35);
				
				\coordinate (FL1) at ($(FU1)+(dt)$);
				\coordinate (FL2) at ($(FU2)+(dt)$);
				\coordinate (BL1) at ($(BU1)+(dt)$);
				\coordinate (BL2) at ($(BU2)+(dt)$);
				
				\coordinate (FM1) at ($(FU1)!0.5!(FL1)$);
				\coordinate (FM2) at ($(FU2)!0.5!(FL2)$);
				\coordinate (BM1) at ($(BU1)!0.5!(BL1)$);
				\coordinate (BM2) at ($(BU2)!0.5!(BL2)$);
				
				\path[fill=myred!10](FU1).. controls (1.10,1.55)  and (2.20,1.55) .. (FU2) --	(BU2) .. controls (2.75,2.15) and (1.65,2.15) .. (BU1)-- cycle;
				\path[fill=myred!5](FL1) .. controls (1.20,1.36) and (2.30,1.36) .. (FL2) -- 	(BL2) .. controls (2.85,1.96) and (1.75,1.96) ..
				(BL1) -- cycle;
				\path[fill=myred!8] (FU1)--(BU1)--(BL1)--(FL1)--cycle;
				\path[fill=myred!8] (FU2)--(BU2)--(BL2)--(FL2)--cycle;
				
				\draw[myred,thick] (FU1) .. controls (1.10,1.55) and (2.20,1.55) .. (FU2) -- (BU2) .. controls (2.75,2.15) and (1.65,2.15) .. (BU1) -- cycle;
				\draw[myred,thick] (FL1) .. controls (1.20,1.36) and (2.30,1.36) ..	(FL2) -- (BL2);
				\draw[myred,thick, dashed] (BL2) .. controls (2.85,1.96) and (1.75,1.96) .. (BL1);
				\draw[ myred,thick] (BL1)--(FL1);
				\draw[myred,thick] (FU1)--(FL1) (FU2)--(FL2) (BU2)--(BL2);
				\draw[myred,thick, dashed] (BU1)--(BL1);
				
				\draw[mygreyred,thick]
				(FM1) .. controls (1.15,1.455) and (2.25,1.455) .. (FM2);
				\draw[mygreyred,thick, dashed] (BM2) ..  controls (2.85,2.05) and (1.65,1.99) .. (BM1);
				\draw[ mygreyred,thick, dashed] (BM1)--(FM1);
				\draw[ mygreyred,thick] (BM2)--(FM2);
				
				\coordinate (Pxi) at (2.00,1.3);
				\fill[black] (Pxi) circle (1.0pt);
				
				\node[myred, font=\footnotesize, anchor=east, xshift=0.5pt, yshift=0.5pt] at (Pxi) {$\xi$};
				
				\draw[black, line width=0.5pt,->,>=stealth]	(Pxi) -- ($(Pxi)+(0.25,0.28)$);
				\draw[black, line width=0.5pt, ->, >=stealth] (Pxi) .. controls ($(Pxi)+(0.25,0.10)$) and ($(Pxi)+(0.42,0)$) .. ($(Pxi)+(0.72,-0.2)$);
				
				\draw[black, line width=0.75pt] ($(BU2)+(0.2,0.1)$) -- ($(BL2)+(0.2,0.1)$);
				\node[right=2pt, font=\footnotesize] at ($(BU2)!0.5!(BL2)+(0.18,0.02)$) {$h$};
				
				\node[myred] at (-0.2,-0.1) {$\Omega_{\xi}$};
				
			\end{scope}
			\begin{scope}[shift={(10.65,0.55)}]
				\coordinate (NL) at (0.00,0.52);   
				\coordinate (NM) at (2.02,0.02);   
				\coordinate (NR) at (4.02,0.60);   
				\coordinate (FR) at (4.48,1.12);   
				\coordinate (FM) at (2.32,0.72);   
				\coordinate (FL) at (0.62,1.20);   
				
				\coordinate (off) at (0.15,-0.30);
				\coordinate (NLb) at ($(NL)+(off)$);
				\coordinate (NMb) at ($(NM)+(off)$);
				\coordinate (NRb) at ($(NR)+(off)$);
				\coordinate (FRb) at ($(FR)+(off)$);
				\coordinate (FMb) at ($(FM)+(off)$);
				\coordinate (FLb) at ($(FL)+(off)$);
				
				\coordinate (cA1) at (0.60,0.86);
				\coordinate (cA2) at ($(NM)+(-0.78,-0.14)$);
				\coordinate (cB1) at ($(NM)+(0.78,0.14)$);
				\coordinate (cB2) at (3.44,0.84);
				\coordinate (cC1) at (4.22,0.86);
				\coordinate (cC2) at (4.40,0.98);
				\coordinate (cD1) at (3.70,1.46);
				\coordinate (cD2) at ($(FM)+(0.68,-0.08)$);
				\coordinate (cE1) at ($(FM)+(-0.68,0.08)$);
				\coordinate (cE2) at (1.10,1.44);
				\coordinate (cF1) at (0.34,0.98);
				\coordinate (cF2) at (0.08,0.76);
				
				\coordinate (cA1b) at ($(cA1)+(off)$);
				\coordinate (cA2b) at ($(cA2)+(off)$);
				
				\coordinate (cB1b) at ($(cB1)+(off)$);
				\coordinate (cB2b) at ($(cB2)+(off)$);
				
				\coordinate (cC1b) at ($(cC1)+(off)$);
				\coordinate (cC2b) at ($(cC2)+(off)$);
				
				\coordinate (cD1b) at ($(cD1)+(off)$);
				\coordinate (cD2b) at ($(cD2)+(off)$);
				
				\coordinate (cE1b) at ($(cE1)+(off)$);
				\coordinate (cE2b) at ($(cE2)+(off)$);
				
				\coordinate (cF1b) at ($(cF1)+(off)$);
				\coordinate (cF2b) at ($(cF2)+(off)$);
				
				\path[fill=myblue!9] (NL) .. controls (cA1) and (cA2) .. (NM)	.. controls (cB1) and (cB2) .. (NR)	.. controls (cC1) and (cC2) .. (FR)	.. controls (cD1) and (cD2) .. (FM)	.. controls (cE1) and (cE2) .. (FL)	.. controls (cF1) and (cF2) .. (NL)	-- cycle;
				
				\path[fill=myblue!4] (NLb)	.. controls (cA1b) and (cA2b) .. (NMb) .. controls (cB1b) and (cB2b) .. (NRb) .. controls (cC1b) and (cC2b) .. (FRb)	.. controls (cD1b) and (cD2b) .. (FMb)
				.. controls (cE1b) and (cE2b) .. (FLb) .. controls (cF1b) and (cF2b) .. (NLb)	-- cycle;
				
				\path[fill=myblue!7] (NL) .. controls (cA1) and (cA2) .. (NM)	.. controls (cB1) and (cB2) .. (NR)	-- (NRb) .. controls (cB2b) and (cB1b) .. (NMb)	.. controls (cA2b) and (cA1b) .. (NLb) -- cycle;
				
				\path[fill=myblue!7] (NR) .. controls (cC1) and (cC2) .. (FR)	-- (FRb) .. controls (cC2b) and (cC1b) .. (NRb)	-- cycle;
				
				\draw[myblue, thick] (NL) .. controls (cA1) and (cA2) .. (NM)	.. controls (cB1) and (cB2) .. (NR) .. controls (cC1) and (cC2) .. (FR) .. controls (cD1) and (cD2) .. (FM) .. controls (cE1) and (cE2) .. (FL) .. controls (cF1) and (cF2) .. (NL);
				
				\draw[myblue, thick] (NLb) .. controls (cA1b) and (cA2b) .. (NMb) .. controls (cB1b) and (cB2b) .. (NRb) .. controls (cC1b) and (cC2b) .. (FRb);
				\draw[myblue, thick, dashed] (FRb) .. controls (cD1b) and (cD2b) .. (FMb) .. controls (cE1b) and (cE2b) .. (FLb)	.. controls (cF1b) and (cF2b) .. (NLb);
				
				\draw[myblue, thick] (NL)--(NLb);
				
				\draw[myblue, thick] (NR)--(NRb);
				
				\draw[myblue, thick] (FR)--(FRb);
				
				\draw[myblue,thick, dashed] (FL)--(FLb);

				\coordinate (NLm) at ($(NL)!0.5!(NLb)$);
				\coordinate (NMm) at ($(NM)!0.5!(NMb)$);
				\coordinate (NRm) at ($(NR)!0.5!(NRb)$);
				
				\coordinate (FRm) at ($(FR)!0.5!(FRb)$);
				\coordinate (FMm) at ($(FM)!0.5!(FMb)$);
				\coordinate (FLm) at ($(FL)!0.5!(FLb)$);
				
				\coordinate (cA1m) at ($(cA1)!0.5!(cA1b)$);
				\coordinate (cA2m) at ($(cA2)!0.5!(cA2b)$);
				
				\coordinate (cB1m) at ($(cB1)!0.5!(cB1b)$);
				\coordinate (cB2m) at ($(cB2)!0.5!(cB2b)$);
				
				\coordinate (cD1m) at ($(cD1)!0.5!(cD1b)$);
				\coordinate (cD2m) at ($(cD2)!0.5!(cD2b)$);
				
				\coordinate (cE1m) at ($(cE1)!0.5!(cE1b)$);
				\coordinate (cE2m) at ($(cE2)!0.5!(cE2b)$);
				
				\coordinate (cC1m) at ($(cC1)!0.5!(cC1b)$);
				\coordinate (cC2m) at ($(cC2)!0.5!(cC2b)$);
				
				\coordinate (cF1m) at ($(cF1)!0.5!(cF1b)$);
				\coordinate (cF2m) at ($(cF2)!0.5!(cF2b)$);
				
				\draw[mygreyred, thick] (NLm) .. controls (cA1m) and (cA2m) .. (NMm) .. controls (cB1m) and (cB2m) .. (NRm);
				
				\draw[mygreyred, thick, dashed] (FLm) .. controls (cE2m) and (cE1m) .. (FMm) .. controls (cD2m) and (cD1m) .. (FRm);
				
				\draw[mygreyred, thick] (NRm) .. controls (cC1m) and (cC2m) .. (FRm);
				
				\draw[mygreyred, thick, dashed] (FLm) .. controls (cF1m) and (cF2m) .. (NLm);
				
				\coordinate (Px) at (2.4,0.25);
				\fill[black] (Px) circle (1.0pt);
				
				\node[ myblue, font=\footnotesize, anchor=north west, xshift=-22.5pt, yshift=7pt] at (Px) {$\varphi(x)$};
				
				\draw[black, line width=0.75pt]	($(FR)+(0.2,0.1)$) --	($(FRb)+(0.2,0.1)$);
				
				\node[ right=2pt ] at ($(FR)!0.5!(FRb)+(0.2,0.1)$) {$h$};
				
				\node[myblue] at (0,0) {$\Omega_c$};
				
				\draw[black, line width=0.5pt,->,>=stealth]	(Px) -- ($(Px)+(0.2,0.2)$);
				\draw[black, line width=0.5pt, ->, >=stealth] (Px) .. controls ($(Px)+(0.25,0.10)$) and ($(Px)+(0.3,0.1)$) .. ($(Px)+(0.5,0.15)$);
			\end{scope}
			\draw[->,thick,dashed]
			(3.7,1.6) .. controls (4.5,4) and (5,5) .. (6.9,5);
			\node at (4,4) {$\Theta, \nabla_x \Theta$};
			
			\draw[->,thick,dashed]
			(4.80,1.45) .. controls (6.30,2.10) and (7.95,2.20) .. (10.75,1.78);
			\node at (7.5,2.3) {$\varphi, \nabla_x\varphi $};
			
			\draw[->,thick,dashed]
			(9.95,4.80) .. controls (10.20,4.95) and (12.30,4.55) .. (12.8,1.85);
			\node at (12.75,3.70) {\ \ $\varphi_{\xi}, \nabla_{\xi} \varphi_\xi$};
			
		\end{tikzpicture}
		\caption{\,{Configurations and mappings used in the shell kinematics.
				The plate-like domain $\Omega_h$ represents the fictitious Cartesian configuration,
				$\Theta:\Omega_h\to\Omega_\xi$ describes the geometry of the initial shell configuration
				$\Omega_\xi$, and $\varphi_\xi:\Omega_\xi\to\Omega_c$ maps it onto the current
				(deformed) configuration $\Omega_c$. Accordingly,
				$\varphi=\varphi_\xi\circ\Theta:\Omega_h\to\Omega_c$ is the deformation expressed in
				the fictitious coordinates.}}
		\label{fig:configurations}
	\end{figure}
	
	The domain $\Omega_h $ can be viewed as a \textit{fictitious Cartesian configuration} of the body, {see Figure \ref{fig:configurations}.} The diffeomorphism $\Theta$ maps the midsurface $\omega$ of the fictitious Cartesian configuration parameter space $\Omega_h$ onto the midsurface $\omega_\xi={y_0}({\omega})$ of $\Omega_\xi$, and $n_{y_0}$ is the unit normal vector for $\omega_\xi$. For simplicity and where no confusion may arise, in what follows we will explicitly omit writing the arguments $(x_1,x_2, x_3)$.
	
	In this paper, we consider the following Kirchhoff--Love ansatz for the function $\varphi$, which maps the fictitious planar reference configuration $\Omega_h$ into the deformed (current) configuration $\Omega_c$
	\begin{equation}
		\varphi:=\varphi_\xi\circ \Theta:\Omega_h\rightarrow \Omega_c, \quad 	\varphi(x',x_3)\,=\,m(x')+x_3 n_m(x'), \ \quad  n_m\,=\,\dd\frac{\partial_{x_1}m\times \partial_{x_2}m}{\lVert \partial_{x_1}m\times \partial_{x_2}m\rVert}\,,
	\end{equation}
	where $m:\omega\to \mathbb{R}^3$ is a regular mapping. {Therefore, in our model, the normals remain straight (no transverse shear) and orthogonal to the midsurface after deformation, and no stretching in the thickness direction is considered. A generalized ansatz allowing an independent director or a transverse stretch, for example $\varphi(x',x_3)=m(x')+x_3d(x')$ or $\varphi(x',x_3)=m(x')+x_3\rho(x')n_m(x')$, would provide additional freedom in the transverse response and could relax some of the restrictions inherent in the pure Kirchhoff--Love reconstruction. However, the determinant structure,} {the weak compactness properties,} {and the convexity arguments would then change substantially, so the existence proof developed here does not transfer verbatim. We regard this extension as a natural subject for future work. Another limitation of our models is that the reconstructed three-dimensional field is not expected, in general, to satisfy the full traction-free conditions on the upper and lower faces. This is a kinematic limitation of the pure Kirchhoff--Love ansatz and is distinct from the variational well-posedness of the reduced two-dimensional problem. In the flat linearized setting, this enforces vanishing linearized extension in the thickness direction; relaxing the linearized extension in the thickness direction yields the usual plane-stress coefficient \(\lambda^*=2\lambda\mu/(\lambda+2\mu)\). If \(\lambda=0\), the distinction associated with this volumetric coupling disappears. Alternatively, one may interpret the Kirchhoff--Love condition as an internal kinematic constraint and introduce reactive stresses at the three-dimensional level, as discussed by Ozenda and Virga \cite{Ozenda21}, or use the already mentioned idea to consider an independent director \cite{vsilhavycurvature} or even the full Cosserat nonlinear framework \cite{GhibaNeffPartI,GhibaNeffPartII}.
	}

	\,{By combining an asymptotic reduction with suitable positive-weight
		quadrature rules, we obtain models of the general form
		\begin{equation}\label{e89}
			\mathcal J(m)
			=
			\int_{\omega}
			\Big[
			W_{\rm shell}\big({\rm I}_m,{\rm II}_m,{\rm III}_m\big)
			+
			W_{\rm vol}^{\mathcal Q}
			\big({\rm a}_m,{\rm a}_m{\rm H}_m,{\rm a}_m{\rm K}_m\big)
			\Big]
			{\rm det}(\nabla y_0|n_{y_0})\,{\rm d}x'
			-\mathcal L(m,n_m),
		\end{equation}
		where $\mathcal Q=B$ in Model~I and $\mathcal Q=S$ in Model~II,
		and $B$ and $S$ denote Boole's and Simpson's rules, respectively.
		For the mixed Taylor--Boole reduction of Model~III, we instead have
		\begin{align}\label{e89b}
			\mathcal J(m)
			=&
			\int_{\omega}
			\Big[
			W_{\rm shell}\big({\rm I}_m,{\rm II}_m,{\rm III}_m\big)
			+
			W_{\log}^{B}
			\big({\rm a}_m,{\rm a}_m{\rm H}_m,{\rm a}_m{\rm K}_m\big)
			\notag\\
			&\qquad\qquad
			+
			W_{\rm vol}^{T}
			\big(
			{\rm a}_m,
			{\rm a}_m\Delta{\rm H}_m,
			{\rm a}_m\Delta{\rm K}_m
			\big)
			\Big]
			{\rm det}(\nabla y_0|n_{y_0})\,{\rm d}x'
			-\mathcal L(m,n_m).
		\end{align}
		Here ${\rm I}_m,{\rm II}_m,{\rm III}_m\in\mathbb R^{2\times2}$
		denote the matrix representations of the three fundamental forms,
		${\rm a}_m=\sqrt{\det{\rm I}_m}$,
		${\rm H}_m$ and ${\rm K}_m$ are the mean and Gaussian curvatures,
		respectively, and
		\[
		\Delta{\rm H}_m={\rm H}_m-{\rm H}_{y_0},
		\qquad
		\Delta{\rm K}_m={\rm K}_m-{\rm K}_{y_0}.
		\]
		The precise expressions of
		$W_{\rm shell}$, $W_{\rm vol}^{\mathcal Q}$,
		$W_{\log}^{B}$, and $W_{\rm vol}^{T}$, as well as other rigorous details concerning the notation, will be presented in the next sections.}
	
	The models derived in the present paper involve, in a {unified} and fully argued way, some energy terms appearing in various other theories from classical elasticity. For instance, energy terms of the form $W_{\mathrm{shell}}\big( {\rm I}_{m}, {\rm II}_{m}, {\rm III}_{m}\big)$ have been considered by Mardare \cite{mardare2008derivation, mardare2019nonlinear}, {by Giang and Mardare \cite{giang2024existence}, and by Giang \cite{giang2025existence}}. In fact, the importance of the use of the third fundamental form is due to the Korn-type inequalities that may be proved for such a model. A nonlinear Korn inequality on a surface is any estimate of the distance, up to a proper isometry of $\mathbb{R}^3$, between two surfaces measured in some appropriate norm, in terms of the distances between their three fundamental forms measured in some appropriate norms. The first such nonlinear Korn inequality on a surface, which is due to {Ciarlet, Gratie, and Mardare \cite{ciarlet2006nonlinear}, was then refined (also by showing the role of the third fundamental form) in \cite{ciarlet2016nonlinear,ciarlet2019new,malin2018nonlinear,malin2021nonlinear}. The starting point of the proof of these results is the ``geometric rigidity lemma'' due to Friesecke, James, and M\"uller \cite{Mueller02} and the observation that the continuity of the surface depends on its first and second fundamental forms; see, for example, \cite{ciarlet2005recovery}.} The remark that the two first fundamental forms alone are not enough to model the bending and the change of curvature behavior of a shell was also formulated in the context of the constrained Cosserat shell model \cite{GhibaNeffPartIII}. In \cite{GhibaNeffPartVI} it was explained why, even if it has become commonplace for the stored energy function of any realistic shell model to align ``within first order'' with the classical Koiter membrane-bending (flexural) shell model, the differences of the {first and second} fundamental forms are not measures that describe the bending behavior or variations of the curvatures of the shell. There are numerous reasons why a modified version of the classical Koiter model should be considered. This conclusion is supported not only by Koiter himself \cite{koiter1973foundations}, but also by Sanders and Budiansky \cite{budiansky1962best,budiansky1963best}, who independently {of Koiter} developed similar theories around the same time. Recent papers published by {\v{S}}ilhav{\'y} \cite{vsilhavycurvature} and Virga \cite{virga2023pure}, together with the Cosserat shell models derived in \cite{GhibaNeffPartIII,saem2023geometrically,saem2023explicit}, have {motivated the consideration of different strain tensors} in shell models. Anicic and L{\'{e}}ger also explored this question {previously} in the context of linear models, see \cite{anicic1999formulation,anicic2002mesure} and \cite{anicic2003shell}. The issue of appropriate nonlinear shell bending strain measures was also already addressed by Acharya in \cite{acharya2000nonlinear} and more recently by Vitral and Hanna \cite{vitral2023dilation}. Some invariance properties related to bending strain measures were introduced even earlier by Antman \cite{antman1968general}, while other strain measures based on the classical Biot three-dimensional parent elastic energy were given by Atluri \cite{atluri1984alternate}. 
	
	\,{We derive three reduced models. Models~I and III are consistent through
		order $h^5$, while Model~II is consistent through order $h^3$. Models~I
		and II retain the complete volumetric energy in a positive-weight
		quadrature form, whereas Model~III uses a mixed Taylor--Boole reduction.
		This distinction is {insignificant} for the formal order of {approximations of Models~I and III} but becomes essential in {the proofs of existence results. In particular, }Models~I and II require no thickness restriction beyond the geometric regularity condition \eqref{ch5in}\,{, while} Model~III requires the additional condition \eqref{det-condition}, originating solely from the convexity
		of the Taylor-expanded quadratic volumetric contribution.}
	
	On the other hand, terms of order $h^5$ are usually considered when some energy terms depending on the third fundamental form are included in the structure of the shell models \cite{mardare2019nonlinear,giang2024existence}. In comparison with these models, in our paper, all the energy terms appear  due to the modelling  approach, and there is no part of the energy that is added ad hoc without being an approximation of some terms of the {parent} three-dimensional variational problem.    
	We note that in classical nonlinear elasticity $\Gamma$-limits for plates and higher scalings (energy scaling with $h^\beta$ {where $\beta\geq2$}) of the nonlinear elastic energy {were given \cite{Mueller02,Mueller02b, Mueller06}}, see also \cite{conti2008confining}, while $\Gamma$-limits for shells {with higher scalings were obtained for $\beta=2$ in \cite{Mueller03}} and for $\beta\geq 4$ in \cite{lewicka2009nonlinear,lewicka2011note,lewicka2010shell,lewicka2011matching}, see also \cite{Babadjian} for various mathematical contributions to the behavior of
	thin films. In the Cosserat framework, the models given in \cite{GhibaNeffPartI,GhibaNeffPartIII} are up to order ${\rm O}(h^5)$, and it has been proven that, besides the advantages concerning their existence results \cite{GhibaNeffPartII}, such models lead to good comparison with the {numerical} results \cite{nebel2023geometrically}.
	
	Another characteristic of the models introduced in this paper is that the constitutive coefficients are those from the three-dimensional formulation, while the influence of the curved initial shell configuration appears explicitly through the derivation approach in the expression of the coefficients of the energies for the reduced two-dimensional variational problem. Hence, it shows that the effective coefficients and the structural response of the reduced
	shell model depend on the three-dimensional material parameters, the
	thickness, and the geometry of the reference midsurface. Indeed, {theoretical and} experimental studies on cylindrical and spherical shells have shown that the mechanical response and stability of shell structures depend strongly on both the thickness of the shell and the curvature of the reference configuration; see, for instance, the classical {buckling analysis} of Zoelly \cite{zoelly1915}, the studies on cylindrical shell buckling by {von~Karman and Tsien  \cite{von1941buckling}}, the analysis of shell stability by Koiter \cite{koiter1969}, as well as more recent experimental investigations of elastic shells such as the indentation experiments of Vella et al.~\cite{vella2012indentation} and the work on geometry-induced rigidity by Lazarus, Florijn and Reis \cite{lazarus2012geometry}. In some works in the mathematical mechanics literature, the identification of the constitutive parameters from the model is done only a posteriori by fitting the solutions of some specific problems with the solutions obtained by considering the shell as a three-dimensional body. However, this fitting is usually done only for linear problems, {since one reason is that it is difficult to obtain existence results for nonlinear shell models, as even for the classical three-dimensional nonlinear models, the existence of minimizers remains delicate. Also,} such shell models do not involve the curvatures of the reference configurations in the construction of the model, see \cite{Birsan-Alten-2011}. The derivation we propose avoids this difficulty in the nonlinear case and provides a link between nonlinear three-dimensional elasticity and geometrically nonlinear shell theories. {This link {based on dimensional reduction} should not be confused with the plane-stress relaxation leading to the classical Koiter constitutive tensor: the
		Kirchhoff--Love ansatz adopted here suppresses transverse normal stretch.}
	
	\,{There is another particularity of our models, namely that
		the specific form of the purely volumetric part of the
		Ciarlet--Geymonat energy, together with the positive-weight quadrature
		rules used in the reduction and, in Model~III, the Taylor expansion of
		the quadratic volumetric term, leads to energies depending on the
		area-weighted nodal quantities
		\[
		{\rm a}_m\mathcal A_m(\zeta_j)
		=
		{\rm a}_m-2h\zeta_j({\rm a}_m{\rm H}_m)
		+h^2\zeta_j^2({\rm a}_m{\rm K}_m), \qquad \zeta_j\in \left[-\frac{1}{2}, \frac{1}{2}\right],
		\]
		and, in Model~III, also on
		${\rm a}_m\Delta{\rm H}_m$ and
		${\rm a}_m\Delta{\rm K}_m$.}  We mention that there are some other models that use functionals depending on the difference of the curvatures; see, e.g., the model proposed by Helfrich \cite{helfrich1973elastic} for biological membranes,  the model previously proposed by Canham \cite{canham1970minimum} to model the shape of the human red blood cell,  and a chapter by {Steigmann} \cite{steigmann2017role}. For particular geometries of the reference configuration, Helfrich's energy coincides with the Willmore functional (Blaschke bending energy) considered in geometry \cite{willmore1965note}. Before the article by Willmore, Blaschke's work considered the variational problem associated with this measure \cite{blaschke1924vorlesungen}. In fact, the origin of the functional now known as the Willmore functional goes back to the shell community, since Germain \cite{germain1831memoire} first considered the integral of the squared mean
	curvature related to the behavior of elastic plates.  As noted by Anicic \cite{anicic2018polyconvexity}, Helfrich's energy fits the polyconvex framework used to establish weak
	lower semicontinuity under the standard accompanying hypotheses but {it does not enforce orientation-preservation}. {This is one reason why, for the approximation of
		$-\log\det\nabla_\xi\varphi_\xi$, we do not use the direct asymptotic
		approach and instead employ a positive-weight quadrature rule: Boole's
		rule in the reductions consistent through order $h^5$ and Simpson's
		rule in the reduction consistent through order $h^3$.} The Taylor expansion applied to $-\log{\rm det} \nabla_\xi \varphi_\xi$ destroys the good three-dimensional mathematical structure, and {the resulting expression no longer has a structure for which weak lower semicontinuity follows from the available compactness results.} However, for the other pure volumetric part, i.e., $({\rm det} \nabla_\xi \varphi_\xi)^2$, the asymptotic approach may be used and leads to some energy terms related to Helfrich's model.  
	
	\,{Regarding the area-weighted quantities arising from the quadrature
		procedure, the endpoint variables
		${\rm a}_m{\rm A}_m^\pm$ occur for both Simpson's and Boole's rules,
		whereas Boole's rule additionally samples the two interior thickness
		nodes $\pm h/4$. More generally, every quadrature node produces an
		affine combination of
		$({\rm a}_m,{\rm a}_m{\rm H}_m,{\rm a}_m{\rm K}_m)$.} These types of energy terms were also present in some previous shell models; see, for instance, the papers by Anicic and L{\'e}ger \cite{anicic1999formulation} and then by Anicic \cite{anicic2018polyconvexity,anicic2019existence} in the context of the so-called $G^1_h$ shell models. In this paper, we consider a polyconvexity concept that is related to the approach used by Anicic in \cite{anicic2018polyconvexity,anicic2019existence} in the shell theories, and we show that the energies of the proposed models inherit from the three-dimensional {parent} energy a convex structure compatible with the shell-polyconvexity framework used in the existence analysis.  Actually, in order to show the lower semicontinuity of the functionals defining the models, we will use some other results proved by Anicic in \cite{anicic2018polyconvexity,anicic2019existence}. The models proposed by Anicic were modified by adding some additional terms of order $O(h^5)$ in \cite{giang2025existence}, terms that also depend on the curvatures of the reference configuration. However, these models were not derived from three-dimensional polyconvex energies as in our study. {Hence, we use from \cite{anicic2018polyconvexity,anicic2019existence} the choice of area-weighted curvature variables, their {compensated compactness} properties, and the corresponding {direct method} strategy for shell energies convex in these variables. The new contributions here are the derivation of the reduced energies directly from the three-dimensional Ciarlet--Geymonat model, the use of positive-weight Simpson and Boole quadrature rules to preserve the favorable volumetric structure while matching the intended order of the corresponding reduced models, the {principal-direction} convexity analysis showing that Models I and II require no thickness restriction beyond geometric regularity, and the explicit additional convexity threshold obtained for Model III.
	}
	
	Therefore, the present paper gives some shell models in the classical theory of isotropic elastic materials that incorporate energy terms previously considered (but not together) in various other shell models. Moreover, in our models, all these terms appear in the dimensional reduction process; they are not added ad hoc as a posteriori corrections in order to satisfy orientation preservation, coercivity, or convexity of the functionals. We {also} show that our shell models are well-posed {in the sense of existence of minimizers}.
	
	\section{The dimensional reduction of the three-dimensional model} 
	
	\subsection{The three-dimensional {parent} problem and geometry of the reference configuration}
	
	We consider the following \textit{three-dimensional  geometrically nonlinear minimization problem} posed on $\Omega_\xi$:
	\begin{equation}\label{minprob1}
		\,{\mathcal{I}(\varphi_\xi)}=\dd\int_{\Omega_\xi}W_{\rm{CG}}(F _\xi){\rm d}{\xi}
		- \Pi(\varphi_\xi)\quad 
		{\to}
		\textrm{\ \ min.} \quad  {\rm   w.r.t. }\quad \varphi_\xi\, ,
	\end{equation}
	where
	\begin{align}
		F_\xi:\,=\,&\nabla_\xi\varphi_\xi\in\mathbb{R}^{3\times3},  \notag\\
		\dd W_{\rm{CG}}(F _\xi):\,=\,&\dd\frac{\mu}{2}\left[|F_\xi|^2-2\,\log (\det F_\xi)-3\right]+\frac{\lambda}{4}\left[\,(\det F_\xi)^2-2\,\log (\det F_\xi)-1\right], \notag\\
		\Pi(\varphi_\xi):\,=\,
		&\Pi_f(\varphi_\xi)+ \Pi_t(\varphi_\xi)=\textrm{the external loading potential},\\
		\Pi_f(\varphi_\xi):\,=\,&\dd\int_{\Omega_\xi} \bigl\langle  f,  u \bigr\rangle   \, {\rm d}{\xi}\,= \textrm{potential of external applied body forces $ f $},  \notag\\
		\Pi_t(\varphi_\xi):\,=\,&\dd\int_{\partial\Omega_t} \bigl\langle  t,  u \bigr\rangle   \, {\rm d}{\xi}\,= \textrm{potential of external applied boundary forces $ t $}\, ,  \label{loadpot2}\notag
	\end{align}
	and ${\rm d}{\xi}$ denotes {either} the volume element or the area element, respectively in the reference configuration, $u= \varphi_\xi - \xi$ is the displacement vector, $ \partial\Omega_t $ is a subset of the boundary of $ \Omega_\xi $.
	
	The Ciarlet--Geymonat-type energy considered here is polyconvex and possesses a favorable variational structure. For the Ciarlet--Geymonat choice of the volumetric term, the energy keeps the deformation orientation-preserving by penalizing $\det F_\xi\to 0^{+}$ through the logarithmic contribution, while the quadratic term in $\det F_\xi$ supplies additional convexity and growth. This combination is one of the reasons why the model has become a classical compressible hyperelastic law with good analytical properties. Moreover, this energy is consistent with the classical linearized elasticity near the natural configuration\footnote{Here, and in the rest of the paper, the identity tensor in $\mathbb{R}^{n \times n}$ will be denoted by $\id_n$.} $F_\xi=\id_3$.  More precisely, when the deformation is close to the identity, the stored energy admits the same quadratic approximation as the {Saint Venant--Kirchhoff} model, namely a quadratic form in the Green-St.~Venant strain tensor
	$
	E_\xi=\frac{1}{2}(C_\xi-\id_3), \  C_\xi=F_\xi^T F_\xi,
	$
	with Lam\'e coefficients $\lambda$ and $\mu$. Thus, at small strain, the model reproduces the standard linear elastic response, while for finite deformations, it preserves the nonlinear hyperelastic structure and the polyconvexity properties built into the original three-dimensional formulation. This agreement with linear elasticity ``to first order'' is precisely one of the main motivations for the Ciarlet--Geymonat construction.

	Applying the change of variables formula and {$\nabla_\xi \varphi_\xi(\Theta(x))=[\nabla_x \varphi(x)]\,[\nabla_x \Theta(x) ]^{-1}$ }we obtain a new form of the energy functional $\mathcal{I}$ which suggests {seeking} the unknown function $\varphi$  as {solution} of the following minimization problem formulated on the plate-like domain $\Omega_h$
	\begin{equation}\label{minprob2}
		\mathcal{I}_h(\varphi)=\dd\int_{\Omega_h}W_{\rm{CG}}^h(F)\det \nabla_x \Theta(x){\rm d}{x}
		- \Pi_h(\varphi)\quad 
		{\to}
		\textrm{\ \ min.} \quad  {\rm   w.r.t. }\quad \varphi\, ,
	\end{equation}
	where
	\begin{align}
		F:\,=\,\nabla_x \varphi(x),  \quad 
		\dd W_{\rm{CG}}^h(F):=W_{\rm{CG}}(F\,[\nabla_x \Theta(x) ]^{-1}), \quad 
		\,{\Pi_h(\varphi)}:\,=\,
		\dd\int_{\Omega_h} \bigl\langle  \tilde{f}, \tilde{v} \bigr\rangle   \, {\rm d}{x}+ \int_{\Gamma_t} \bigl\langle \tilde t, \tilde{v} \bigr\rangle   \, {\rm d}x,
	\end{align}
	where $ \tilde{v} (x)= \varphi(x) - \Theta(x) $ is the displacement vector and the vector fields $ \tilde f $ and $ \tilde t $ can be determined in terms of $   f $ and $   t $, respectively, for instance
	{$ \;\tilde f(x)=f(\Theta(x))\det(\nabla_x \Theta(x))$} {and 
		$
		\widetilde t(x)
		=
		t(\Theta(x))
		\left|\Cof\big(\nabla_x\Theta(x)\big)\,N(x)\right|,
		$
		where \(N(x)\) denotes the outward unit normal to
		\(\Gamma_t\subset\partial\Omega_h\).} {In the above relations, $\mathrm{d}x$ denotes either the volume element or area element}, $ \Gamma_t $ and $ \Gamma_d $ are nonempty subsets of the boundary of $ \Omega_h $ such that $   \Gamma_t \cup \Gamma_d= \partial\Omega_h $ and $ \Gamma_t \cap \Gamma_d= \emptyset $\,. On $ \Gamma_t $ we consider traction boundary conditions, while on $ \Gamma_d $ we have Dirichlet-type boundary conditions (i.e., $ \varphi $ is prescribed on $ \Gamma_d $). We assume that  $ \Gamma_d  $ has the form $ \Gamma_d = \gamma_d\times (-\frac{h}{2} , \frac{h}{2})$, where the curve $\gamma_d $ is a subset of $ \partial \omega $ with length$(\gamma_d)>0$. Accordingly, the boundary subset $ \Gamma_t $ has the form  $ \Gamma_t = \Big(\gamma_t\times (-\frac{h}{2} , \frac{h}{2})\Big) \cup \Big(\omega\times \Big\{\frac{h}{2}\Big\} \Big) \cup \Big(\omega\times \Big\{-\frac{h}{2}\Big\} \Big) $
	and $ \Theta(\Gamma_t) = \partial\Omega_t\, $.

	\,{Observe} that
	\begin{align}
		\nabla_x \Theta(x',x_3)&\,=\,(\nabla {y_0}|n_{y_0})+x_3(\nabla n_{y_0}|0) \, \  \forall\, x_3\in \left(-\frac{h}{2},\frac{h}{2}\right),\qquad x=(x',x_3), \quad x'=(x_1,x_2),
		\\
		\nabla_x \Theta(x',0)&=\,(\nabla {y_0}|\,n_{y_0}),\ \ [\nabla_x \Theta(x',0)]^{-T}\, e_3\,=n_{y_0},\notag
	\end{align}
	and  $\det\nabla_x \Theta(x',0)=\det(\nabla {y_0}|n_{y_0})=\sqrt{\det[ (\nabla {y_0})^T\nabla {y_0}]}$ represents the surface element. {Throughout the paper, for a mapping  defined on $\omega$, we use
		$
		\nabla \cdot :=(\partial_{x_1} \cdot\,|\,\partial_{x_2} \cdot)
		$
		for its in-plane gradient, whereas $\nabla_x$ denotes the full gradient with respect to $(x_1,x_2,x_3)$ and $\nabla_\xi$ the gradient with respect to the physical coordinates $(\xi_1,\xi_2,\xi_3)$.}

	The matrix representation of {\it the first,  second, and  third  fundamental form}  and of  {\it the Weingarten map} (or shape operator)  on  ${y_0}(\omega)$ are given through
	\begin{align}
		\index{fundamental form ! first}\index{fundamental form ! first extended}
		\label{first_fundamental_form}
		{\rm I}_{y_0}&:\,=\,[{\nabla  {y_0}}]^T\,{\nabla  {y_0}}\,\in\mathbb{R}^{2\times 2},\qquad \quad 		{\rm II}_{y_0}:\,=\,-[{\nabla  {y_0}}]^T\,{\nabla  n_{y_0}}\in \mathbb{R}^{2\times 2}, \\
		{\rm III}_{y_0}&:\,=\,[{\nabla  n_{y_0}}]^T\,{\nabla  n_{y_0}}\in \mathbb{R}^{2\times 2}\, ,\qquad 
		\,{ {\rm L}_{y_0}}\,:=\, {\rm I}_{y_0}^{-1} {\rm II}_{y_0}\, , \quad \text{respectively}.\notag
	\end{align}
	Because ${\rm rank} (\nabla {y_0})\,=\,2$, the tensor $[\nabla {y_0}]^T\nabla {y_0}$ is positive definite, while since $n_{y_0}$ is orthogonal to the tangent space  of the surface ${y_0}$, we have that ${\rm II}_{y_0}\in {\rm Sym}(2)$. The third fundamental form is also symmetric.

	The Gaussian curvature $\, {\rm K}_{y_0}$ and the {\it mean curvature} $\,{\rm H}_{y_0}\,$  of the surface {${y_0}(\omega)$}  are determined by
	\begin{align}
		\index{Gauss curvature}
		\label{gauss_curvatureA}
		\, {\rm K}_ {{y_0}}:\,=\,{\rm det}{({\rm L}_{y_0})}\, ,\qquad 
		2\,{\rm H}_{y_0}\, :\,=\,{\rm tr}({{\rm L}_{y_0}}) \, , \quad  \text{respectively}.
	\end{align}
	The principal curvatures $\kappa_{1,y_0},\kappa_{2,y_0}$ \index{principal curvatures} are the solutions of $\kappa_{y_0}^2-2\,{\rm H}_{y_0}\,\kappa_{y_0}+\, {\rm K}_{y_0}\,=\,0,$ i.e., of the 
	characteristic equation of
	\,{${\rm L}_{y_0}$}.

	We note that \begin{align}\det\nabla_x\Theta(x',x_3)= {\rm det}(\nabla {y_0}|n_{y_0})[1-2\,\,{\rm H}_{y_0}\,x_3+\, {\rm K}_{y_0}\, x_3^2]={\rm det}(\nabla {y_0}|n_{y_0})(1-\kappa_{1,y_0}\,x_3)(1-\kappa_{2,y_0}\, x_3)>0, \end{align} $\forall \, x_3\in \left[-\frac{h}{2},\frac{h}{2}\right]$ if and only if \cite{anicic2018polyconvexity,anicic2019existence,GhibaNeffPartI}
	\begin{align}\label{ch5in}
		h\,\max \{\sup_{x'\in {\overline{\omega}}}|{\kappa_{1,y_0}}|,\sup_{x'\in {\overline{\omega}}}|{\kappa_{2,y_0}}|\}<2.\end{align}
	Clearly, in terms of the principal radii of curvature\footnote{\,{We use the convention that the $R_i=\infty $ if $\kappa_i=0$.}} $R_{1,y_0}=\frac{1}{{|\kappa_{1,y_0}|}}$, $R_{2,y_0}=\frac{1}{{|\kappa_{2,y_0}|}}$, the  condition \eqref{ch5in} is equivalent to 
	\begin{align}\label{relaxedthick}
		h <2\, R_{1,y_0},\quad  h <2\,R_{2,y_0}\quad \text{in}\quad \omega \quad 
		\Leftrightarrow\quad h<2\,\min \{\inf_{x'\in {\omega}}R_{1,y_0},\inf_{x'\in {\omega}}R_{2,y_0}\}.
	\end{align}
	
	\,{For a given matrix $M\in \mathbb{R}^{2\times 2}$ we introduce the {3D-lifted quantities}
		\begin{align}
			\widehat{M} =\begin{footnotesize}\begin{pmatrix}
					M_{11}& M_{12}&0 \\
					M_{21}&M_{22}&0 \\
					0&0&1
				\end{pmatrix}\in \mathbb{R}^{3\times 3}
				\quad \text{and} \quad
			\end{footnotesize}M^\flat =\begin{footnotesize}\begin{pmatrix}
					M_{11}& M_{12}&0 \\
					M_{21}&M_{22}&0 \\
					0&0&0
			\end{pmatrix}  \end{footnotesize}
			\in \mathbb{R}^{3\times 3}, \quad M^\flat =\widehat{M}\, \id_2^\flat, \quad \widehat{M}=M^\flat+e_3 \otimes e_3.
	\end{align}}
	We recall, see, e.g., \cite{GhibaNeffPartI}, that the diffeomorphism $\Theta$ has the following properties for all $x_3$:
	\begin{itemize}
		\item[i)] ${\rm det}(\nabla_x \Theta(x',x_3))\,=\,{\rm det}(\nabla {y_0}|n_{y_0})\Big[1-2\,x_3\,{\rm H}_{y_0}\,+x_3^2 \,\,{\rm K}_{y_0}\Big]$;
		\item[ii)] if $ h\,\max \{\sup_{x'\in {\omega}}|\kappa_{1,y_0}|,\sup_{x'\in {\omega}}|\kappa_{2,y_0}|\}<2$, then for all $x_3\in \left(-\frac{h}{2},\frac{h}{2}\right)$: \begin{align}\label{invf}
			[	\nabla_x \Theta(x',x_3)]^{-1}\,=\,\dd \frac{1}{1-2\,{\rm H}_{y_0}\, x_3+\, {\rm K}_{y_0}\, x_3^2}\left[\id_3+x_3({\rm L}^{\flat}_{y_0}-2\,{\rm H}_{y_0}\, \id_3)+x_3^2\,\,{\rm K}_{y_0}\,e_3 \otimes e_3
			\right] [	\nabla_x \Theta(x',0)]^{-1},
		\end{align}
		\break where  $ e_3 \otimes e_3	=\begin{footnotesize}\begin{pmatrix}
				0&0&0 \\
				0&0&0 \\
				0&0&1
		\end{pmatrix}\end{footnotesize}.$
		
	\end{itemize}

	\subsection{The ansatz and the approximations}
	Following the classical Kirchhoff--Love kinematics, we assume that the deformation
	admits the representation
	\begin{equation}\label{KLa}
		\varphi(x_1,x_2,x_3)\,=\,m(x_1,x_2)+x_3\ n_m(x_1,x_2), \ \ \ \ \ \ \ \ \quad  n_m\,=\,\dd\frac{\partial_{x_1}m\times \partial_{x_2}m}{\lVert \partial_{x_1}m\times \partial_{x_2}m\rVert}\,,
	\end{equation}
	where $m:\omega\to \mathbb{R}^3$ is a regular mapping which defines the midsurface of the deformed configuration $\Omega_c$.  {\,For the deformed midsurface $m(\omega)$ we use the same geometric
		conventions as for the reference midsurface. Namely,
		$
		{\rm I}_m:=(\nabla m)^T\nabla m,
		$ $
		{\rm II}_m:=-(\nabla m)^T\nabla n_m,
		$ $
		{\rm III}_m:=(\nabla n_m)^T\nabla n_m,
		$ $
		{\rm L}_m:={\rm I}_m^{-1}{\rm II}_m.
		$
		The mean and Gaussian curvatures are defined by
		$
		2\,{\rm H}_m:=\tr{\rm L}_m,
		$ $
		{\rm K}_m:=\det{\rm L}_m,
		$
		while $\kappa_{1,m},\kappa_{2,m}$ denote the eigenvalues of ${\rm L}_m.$
	}
	
	To this end, we start from the
	three-dimensional model introduced in the previous subsection and perform
	a dimensional reduction of the energy through the thickness. 
	
	The entire three-dimensional energy  depends on 
	\begin{align}\label{cr1}
		\,{\nabla_\xi \varphi_\xi(\Theta (x',x_3))=[\nabla_x \varphi(x',x_3)]\,[\nabla_x \Theta(x',x_3) ]^{-1}},
	\end{align}
	where
	\begin{align}\label{cr2}
		\nabla_x \varphi(x',x_3)\,&=\,(\nabla m|n_m)+x_3(\,{\nabla n_m}|0), \quad \  \forall\, x_3\in \left(-\frac{h}{2},\frac{h}{2}\right).
	\end{align}
	In contrast to other derivation approaches and models \cite{LeDret96,Raoult95c,ciarlet2018existence,giang2024existence,giang2025existence}, 
	we explicitly retain the dependence of $\nabla_x \Theta(x',x_3)$ on $x_3$ 
	throughout the construction of the model. The advantages of retaining the full $x_3$ dependence of $\nabla_x \Theta(x',x_3)$, and of not approximating $\nabla_x \Theta(x',x_3)$ too early by its value at $x_3=0$,  have been explored {in classical nonlinear elasticity} \cite{anicic1999formulation,anicic2018polyconvexity,anicic2019existence,SteigmannJE2012} and in the Cosserat nonlinear framework \cite{Neff_plate04_cmt,GhibaNeffPartI,Neff_Chelminski_ifb07}.
	
	\,{Using \eqref{cr1}, \eqref{cr2} and \eqref{invf}, we have
		\begin{align}[
			\nabla_\xi \varphi]^T[
			\nabla_\xi \varphi]=\frac{1}{b(x_3)^2}[	\nabla_x \Theta(x',0)]^{-T}&\left[\id_3+x_3({\rm L}^{\flat}_{y_0}-2\,{\rm H}_{y_0}\, \id_3)+x_3^2\,\,{\rm K}_{y_0}\,e_3 \otimes e_3
			\right]^T \notag\\&\left[\,(\nabla m|n_m)+x_3(\,{\nabla n_m}|0)\right]^T\left[\,(\nabla m|n_m)+x_3(\,{\nabla n_m}|0)\right]	\\&\left[\id_3+x_3({\rm L}^{\flat}_{y_0}-2\,{\rm H}_{y_0}\, \id_3)+x_3^2\,\,{\rm K}_{y_0}\,e_3 \otimes e_3
			\right]	[\nabla_x \Theta(x',0)]^{-1},\notag
		\end{align}
		where
		$
		b(x_3)=	1-2\,{\rm H}_{y_0}\, x_3+\, {\rm K}_{y_0}\, x_3^2.
		$
	}
	
	\,{After multiplications and by using the definitions of  $\widehat{\rm I}_m$, ${\rm II}_m^\flat$ and ${\rm III}_m^\flat$, we deduce}
	\begin{align}
		[
		\nabla_\xi \varphi]^T[
		\nabla_\xi \varphi]=	\frac{1}{b(x_3)^2}[	\nabla_x \Theta(x',0)]^{-T}&\left[\id_3+x_3{\rm B}_{y_0} +x_3^2\,\,{\rm K}_{y_0}e_3 \otimes e_3		\right]^T [\widehat{\rm I}_m\,{-}2\,x_3 {\rm II}_m^\flat+x_3^2 {\rm III}_m^\flat]	\\&\left[\id_3+x_3{\rm B}_{y_0} +x_3^2\,\,{\rm K}_{y_0}e_3 \otimes e_3
		\right][	\nabla_x \Theta(x',0)]^{-1},\notag
	\end{align}
	where
	$
	{\rm B}_{y_0} ={\rm L}^{\flat}_{y_0} - 2\,{\rm H}_{y_0} \id_3.
	$
	Using that for all $X\in \mathbb{R}^{2\times 2}$ we have 
	\begin{align}\widehat{X}=e_3 \otimes e_3+X^\flat,  \quad e_3 \otimes e_3 \widehat{X}=e_3 \otimes e_3= \widehat{X}e_3 \otimes e_3,  \quad e_3 \otimes e_3 X^\flat=0=X^\flat e_3 \otimes e_3,
	\end{align}
	after lengthy but straightforward calculations, we obtain
	\begin{align}
		|\nabla_\xi \varphi_{\,{\xi}}|^2=	\tr([\nabla_\xi \varphi_{\,{\xi}}]^T[\nabla_\xi \varphi_{\,{\xi}}])=\frac{1}{b(x_3)^2}\sum_{p=0}^4 x_3^p \tr[\widetilde P_p],
	\end{align}
	where
	\[
	\begin{aligned}
		\widetilde P_p&=[\nabla_x \Theta(x',0)]^{-T} P_p [\nabla_x \Theta(x',0)]^{-1},\quad p=\,{0},1,2,3,4, \notag\\[0.4em]
		P_{0}& = \widehat{\rm I}_m, \qquad 
		P_1 = {\rm B}_{y_0}^T \widehat{\rm I}_m + \widehat{\rm I}_m {\rm B}_{y_0} {-} 2 {\rm II}_m^\flat, \\[0.4em]
		P_2 &= 2\,{\rm K}_{y_0} e_3 \otimes e_3
		+ {\rm B}_{y_0}^T \widehat{\rm I}_m {\rm B}_{y_0}
		\,{-} 2\,{\rm B}_{y_0}^T {\rm II}_m^\flat
		\,{-} 2 {\rm II}_m^\flat {\rm B}_{y_0}
		+ {\rm III}_m^\flat , \\[0.4em]
		P_3 &=\,{\rm K}_{y_0} \bigl( {\rm B}_{y_0}^T e_3 \otimes e_3 + e_3 \otimes e_3 {\rm B}_{y_0} \bigr)
		\,{-} 2\,{\rm B}_{y_0}^T  {\rm II}_m^\flat {\rm B}_{y_0}
		+{\rm B}_{y_0}^T {\rm III}_m^\flat 
		+ {\rm III}_m^\flat  {\rm B}_{y_0}, \notag\\[0.4em] 
		P_4 &=\,{\rm K}_{y_0}^2 e_3 \otimes e_3 + {\rm B}_{y_0}^T {\rm III}_m^\flat  {\rm B}_{y_0} .
	\end{aligned}
	\]
	
	Hence, after  multiplication with   the Jacobian $\det \nabla_x \Theta(x',x_3)={\rm det}(\nabla {y_0}|n_{y_0})\,b(x_3)$ we will have 
	\begin{align}
		|\nabla_\xi \varphi_{\,{\xi}}|^2\det \nabla_x \Theta(x',x_3)&=\frac{1}{b(x_3)}\sum_{p=0}^4 x_3^p \tr[\widetilde P_p]{\rm det}(\nabla {y_0}|n_{y_0}).
	\end{align}
	Then, using  the {Taylor expansion about $x_3=0$}
	(since $x_3 \in \big(-\frac{h}{2}, \frac{h}{2}\,\big)$ and $\,h\,\max \{\sup_{x'\in {\omega}}|\kappa_{1,y_0}|,\sup_{x'\in {\omega}}|\kappa_{2,y_0}|\}<2$ {implies that  
		\(\det\nabla_x\Theta>0\) throughout the thickness and excludes the local
		degeneracy of the normal parametrization})
	\begin{align}\label{e76}
		\dfrac{1}{b(x_3)} =&1+ 2\,{\rm H}_{y_0}\,x_3+ (4\,{\rm H}^2_{y_0}\,-\, {\rm K}_{y_0})\,x_3^2+
		(8\,{\rm H}^3_{y_0}\,-4\,{\rm H}_{y_0}\,\,{\rm K}_{y_0})\,x_3^3 \\&    +(\, {\rm K}_{y_0}^2-12 {\rm H}_{y_0}^2 \,\,{\rm K}_{y_0}+16\,{\rm H}_{y_0}^4)\,x_3^4+O(x_3^5),\notag
	\end{align}
	we have
	\begin{align}
		|\nabla_\xi \varphi_{\,{\xi}}|^2\det \nabla_x \Theta(x',x_3) \,=\, {{\rm det}(\nabla {y_0}|n_{y_0})} [&1+ 2\,{\rm H}_{y_0}\,x_3+ (4\,{\rm H}^2_{y_0}\,-\, {\rm K}_{y_0})\,x_3^2+
		(8\,{\rm H}^3_{y_0}\,-4\,{\rm H}_{y_0}\,\,{\rm K}_{y_0})\,x_3^3    \\& +(\, {\rm K}_{y_0}^2-12 {\rm H}_{y_0}^2 \,\,{\rm K}_{y_0}+16\,{\rm H}_{y_0}^4)\,x_3^4+O(x_3^5)]\sum_{p=0}^4 x_3^p \tr[\widetilde P_p]. \notag
	\end{align}

	Using that ${\rm B}_{y_0}^T\,e_3 \otimes e_3=-2\, {\rm H}_{y_0} \,e_3 \otimes e_3= \,e_3 \otimes e_3\,{\rm B}_{y_0}$ and $\tr \,e_3 \otimes e_3=1$, and {because  the contribution of the fifth-order term vanishes}, direct but lengthy computations yield
	\begin{align}
		&\,{\frac{1}{{\rm det}(\nabla {y_0}|n_{y_0})}}
		\int_{-h/2}^{h/2}|\nabla_\xi \varphi_{\,{\xi}}|^2
		\det \nabla _x\Theta(x',x_3)\,{\rm d}{x}_3
		=\alpha_0\,\langle \widehat{\rm I}_m,\widehat{\rm I}_{y_0}^{-1}\rangle
		\notag\\
		&\quad+\alpha_1\Big(
		\langle \widehat{\rm I}_m,
		{\rm L}^{\flat}_{y_0}\widehat{\rm I}_{y_0}^{-1}\rangle
		+
		\langle \widehat{\rm I}_m,
		\,{
			\widehat{\rm I}_{y_0}^{-1}{\rm L}^{\flat,T}_{y_0}}
		\rangle
		-4{\rm H}_{y_0}
		\langle \widehat{\rm I}_m,\widehat{\rm I}_{y_0}^{-1}\rangle
		-2\langle {\rm II}_m^\flat,\widehat{\rm I}_{y_0}^{-1}\rangle
		\Big)\notag\\
		&\quad+\alpha_2\Bigg[
		2{\rm K}_{y_0}
		+
		\Big\langle \widehat{\rm I}_m,
		\,{
			{\rm L}^{\flat}_{y_0}
			\widehat{\rm I}_{y_0}^{-1}
			{\rm L}^{\flat,T}_{y_0}}
		\Big\rangle
		-2{\rm H}_{y_0}\Big(
		\langle \widehat{\rm I}_m,
		{\rm L}^{\flat}_{y_0}\widehat{\rm I}_{y_0}^{-1}\rangle
		+
		\langle \widehat{\rm I}_m,
		\,{
			\widehat{\rm I}_{y_0}^{-1}{\rm L}^{\flat,T}_{y_0}}
		\rangle
		\Big)
		+4{\rm H}_{y_0}^2
		\langle \widehat{\rm I}_m,\widehat{\rm I}_{y_0}^{-1}\rangle
		\\
		&\qquad\qquad
		-2\Big(
		\langle {\rm II}_m^\flat,
		{\rm L}^{\flat}_{y_0}\widehat{\rm I}_{y_0}^{-1}\rangle
		+
		\langle {\rm II}_m^\flat,
		\,{
			\widehat{\rm I}_{y_0}^{-1}{\rm L}^{\flat,T}_{y_0}}
		\rangle
		\Big)
		+8{\rm H}_{y_0}
		\langle {\rm II}_m^\flat,\widehat{\rm I}_{y_0}^{-1}\rangle
		+
		\langle {\rm III}_m^\flat,\widehat{\rm I}_{y_0}^{-1}\rangle
		\Bigg]\notag\\
		&\quad+\alpha_3\Bigg[
		-4{\rm H}_{y_0}{\rm K}_{y_0}
		-2\Big\langle {\rm II}_m^\flat,
		\,{
			{\rm L}^{\flat}_{y_0}
			\widehat{\rm I}_{y_0}^{-1}
			{\rm L}^{\flat,T}_{y_0}}
		\Big\rangle
		+4{\rm H}_{y_0}\Big(
		\langle {\rm II}_m^\flat,
		{\rm L}^{\flat}_{y_0}\widehat{\rm I}_{y_0}^{-1}\rangle
		+
		\langle {\rm II}_m^\flat,
		\,{
			\widehat{\rm I}_{y_0}^{-1}{\rm L}^{\flat,T}_{y_0}}
		\rangle
		\Big)
		\notag\\
		&\qquad\qquad 
		-8{\rm H}_{y_0}^2
		\langle {\rm II}_m^\flat,\widehat{\rm I}_{y_0}^{-1}\rangle
		+
		\Big(
		\langle {\rm III}_m^\flat,
		{\rm L}^{\flat}_{y_0}\widehat{\rm I}_{y_0}^{-1}\rangle
		+
		\langle {\rm III}_m^\flat,
		\,{
			\widehat{\rm I}_{y_0}^{-1}{\rm L}^{\flat,T}_{y_0}}
		\rangle
		\Big)
		-4{\rm H}_{y_0}
		\langle {\rm III}_m^\flat,\widehat{\rm I}_{y_0}^{-1}\rangle
		\Bigg]\notag\\
		&\quad+\alpha_4\Bigg[
		{\rm K}_{y_0}^2
		+
		\Big\langle {\rm III}_m^\flat,
		\,{
			{\rm L}^{\flat}_{y_0}
			\widehat{\rm I}_{y_0}^{-1}
			{\rm L}^{\flat,T}_{y_0}}
		\Big\rangle
		-2{\rm H}_{y_0}\Big(
		\langle {\rm III}_m^\flat,
		{\rm L}^{\flat}_{y_0}\widehat{\rm I}_{y_0}^{-1}\rangle
		+
		\langle {\rm III}_m^\flat,
		\,{
			\widehat{\rm I}_{y_0}^{-1}{\rm L}^{\flat,T}_{y_0}}
		\rangle
		\Big)
		+4{\rm H}_{y_0}^2
		\langle {\rm III}_m^\flat,\widehat{\rm I}_{y_0}^{-1}\rangle
		\Bigg]
		\,{+\,O(h^7)} .\notag
	\end{align}
	where
	\begin{align}
		\alpha_0&=h+\frac{h^3}{12}(4\, {\rm H}_{y_0}^2-\, {\rm K}_{y_0})+\frac{h^5}{80}(\, {\rm K}_{y_0}^2-12\, {\rm H}_{y_0}^2\, {\rm K}_{y_0}+16\, {\rm H}_{y_0}^4),
		\quad 
		\alpha_1=\frac{h^3}{12}(2\, {\rm H}_{y_0})+\frac{h^5}{80}(8\, {\rm H}_{y_0}^3-4\, {\rm H}_{y_0}\, {\rm K}_{y_0}),\notag\\
		\alpha_2&=\frac{h^3}{12}+\frac{h^5}{80}(4\, {\rm H}_{y_0}^2-\, {\rm K}_{y_0}),\qquad
		\alpha_3=\frac{h^5}{80}(2\, {\rm H}_{y_0}),\qquad
		\alpha_4=\frac{h^5}{80}.\
	\end{align}
	
	Let us introduce the following  three linear operators $\mathcal F_i:\mathbb{R}^{2\times 2}\to \mathbb{R}$, $i=0,1,2$
	\begin{align}
		\label{basic-contractions}
		\mathcal F_0(Q)&:=\langle Q,{{\rm I}}_{y_0}^{-1}\rangle,\notag\\
		\mathcal F_1(Q)&:=\langle Q,{\rm L}_{y_0}{{\rm I}}_{y_0}^{-1}+{{\rm I}}_{y_0}^{-1}\,{\rm L}_{y_0}^{\,{T}}\rangle,\\
		\mathcal F_2(Q)&:=\langle Q,{\rm L}_{y_0}{{\rm I}}_{y_0}^{-1}{\rm L}_{y_0}^{\,{T}}\rangle, \qquad  \qquad Q\in \mathbb{R}^{2\times 2}.\notag
	\end{align}
	With the help of these functions and using the identities $ \widehat X\widehat{Y}=\widehat{X\, Y}$, $ X^\flat {Y}^\flat=(X\, Y)^\flat$, $ X^\flat \widehat{Y}=(X\, Y)^\flat$, $\langle X^\flat, \widehat{Y}\rangle_{\mathbb{R}^{3\times 3}}=\langle X, {Y}\rangle_{\mathbb{R}^{2\times 2}}$, $\langle X^\flat, Y^\flat\rangle_{\mathbb{R}^{3\times 3}}=\langle X, {Y}\rangle_{\mathbb{R}^{2\times 2}}$, $\langle \widehat X, \widehat Y\rangle_{\mathbb{R}^{3\times 3}}=1+\langle X, {Y}\rangle_{\mathbb{R}^{2\times 2}}$ for all $X,Y\in \mathbb{R}^{2\times 2}$, we have
	\begin{align}
		\frac{1}{\,{{\rm det}(\nabla {y_0}|n_{y_0})}}\	&\int_{-h/2}^{h/2}|\nabla_\xi \varphi_{\,{\xi}}|^2\det \nabla_x \Theta(x',x_3){\rm d}{x}_3\notag\\
		=&\Bigl(
		h
		+\frac{h^3}{12}(-\, {\rm K}_{y_0})
		+\frac{h^5}{80}(\, {\rm K}_{y_0}^2)
		\Bigr)\,\mathcal F_0({\rm I}_m)\,{+}
		\frac{h^3}{3}\,{\rm\,{\rm H}_{y_0}}
		\,\mathcal F_0({\rm II}_m)
		\notag\\
		&+\Bigl(
		\frac{h^3}{12}
		-\frac{h^5}{80}\, {\rm K}_{y_0}
		\Bigr)\,\mathcal F_0({\rm III}_m)+\Bigl(
		-\frac{h^5}{40}\,{\rm H}_{y_0}\, {\rm K}_{y_0}
		\Bigr)\,\mathcal F_1({\rm I}_m) \\
		&\,{-}\Bigl(
		\frac{h^3}{6}
		-\frac{h^5}{40}\, {\rm K}_{y_0}
		\Bigr)\,\mathcal F_1({\rm II}_m)+\Bigl(
		\frac{h^3}{12}
		+\frac{h^5}{80}(4\, {\rm H}_{y_0}^2-\, {\rm K}_{y_0})
		\Bigr)\,\mathcal F_2({\rm I}_m) \notag\\
		&\,{-}\Bigl(
		\frac{h^5}{20}\,{\rm H}_{y_0}
		\Bigr)\,\mathcal F_2({\rm II}_m)
		+
		\frac{h^5}{80}
		\,\mathcal F_2({\rm III}_m)+\Bigl(h+
		\frac{h^3}{12}\, {\rm K}_{y_0}
		\Bigr)\,{+\,O(h^7)},\notag
	\end{align}
	which provides an asymptotic direct approximation of the energy term $|\nabla_\xi \varphi_\xi|^2$ appearing in the three-dimensional Ciarlet--Geymonat energy.
	
	We now turn to the approximation of the purely volumetric terms of the 
	three-dimensional Ciarlet--Geymonat energy. 
	To this end, we define
	\begin{align}
		{\rm a}_m:=\det(\nabla m|n_m),\qquad \qquad {\rm a}_{y_0}:=\det(\nabla {y_0}|n_{y_0})
	\end{align}
	and the differences between the curvature of the referential (undeformed) midsurface defined by ${y_0}$ and the current (deformed) midsurface defined by $m$
	\begin{align}
		\Delta{\rm H}_m :=\,{\rm H}_m -\,{\rm H}_{y_0}, \qquad \qquad \Delta{\rm K}_m :=\,{\rm K}_m -\,{\rm K}_{y_0}.
	\end{align}
	
	For all $x_3 \in \big(-\frac{h}{2}, \frac{h}{2}\,\big)$, we have
	$
	\det \nabla_\xi\varphi_\xi
	=\,\,{\frac{{\rm a}_m}{ {\rm a}_{y_0}}}
	{\Big[1-2\,x_3\,{\rm H}_{m}\,+x_3^2 \,\,{\rm K}_{m}\Big]}
	{\Big[1-2\,x_3\,{\rm H}_{y_0}\,+x_3^2 \,\,{\rm K}_{y_0}\Big]^{-1}},
	$
	\,{and, under \eqref{ch5in}, the following Taylor expansion {about $x_3=0$} through order
		$x_3^4$,}
	\begin{align}
		(\det \nabla_\xi \varphi_{\,{\xi}})^2
		=
		\left(\frac{{\rm a}_m}{{\rm a}_{y_0}}\right)^2 \Big(
		1
		&+ c_1 x_3
		+ c_2 x_3^2
		+ c_3 x_3^3
		+ c_4 x_3^4
		+ O(x_3^5)
		\Big),
	\end{align}
	where the coefficients are
	\begin{align}
		c_1 &= -4\,\Delta{\rm H}_m,
		\qquad 
		c_2 = -8\,{\rm H}_{y_0} \Delta{\rm H}_m
		+ 4 (\Delta{\rm H}_m)^2
		+ 2 \Delta{\rm K}_m,\notag
		\\[0.5em]
		c_3 &= -16\,{\rm H}_{y_0}^2 \Delta{\rm H}_m
		+ 16\,{\rm H}_{y_0} (\Delta{\rm H}_m)^2
		+ 4\,{\rm K}_{y_0}\Delta{\rm H}_m
		+ 4\,{\rm H}_{y_0} \Delta{\rm K}_m
		- 4 \Delta{\rm H}_m \Delta{\rm K}_m,\notag
		\\[0.5em]
		c_4 &=
		-32\,{\rm H}_{y_0}^3 \Delta{\rm H}_m
		+ 48\,{\rm H}_{y_0}^2 (\Delta{\rm H}_m)^2
		+ 8\,{\rm H}_{y_0}^2 \Delta{\rm K}_m
		+ 16\,{\rm H}_{y_0}\,{\rm K}_{y_0} \Delta{\rm H}_m
		-16\,{\rm H}_{y_0} \Delta{\rm H}_m \Delta{\rm K}_m \\[0.5em]
		& \qquad-8\,{\rm K}_{y_0} (\Delta{\rm H}_m)^2
		-2\,{\rm K}_{y_0}\Delta{\rm K}_m
		+ (\Delta{\rm K}_m)^2.\notag
	\end{align}
	Then, after  multiplication with  the Jacobian $\det \nabla_x \Theta(x',x_3)={\rm det}(\nabla {y_0}|n_{y_0})\,b(x_3)$  and integration through the thickness we obtain 
	\begin{align}
		\int_{-h/2}^{h/2} (\det \nabla_\xi \varphi_{\,{\xi}})^2 \det \nabla_x \Theta(x',x_3)\,{\rm d}{x}_3
		= &\frac{{\rm a}_m^2}{{\rm a}_{y_0}}\Bigg[
		h
		+\frac{h^3}{12}\Big(\, {\rm K}_{y_0}+4(\Delta{\rm H}_m)^2+2\Delta{\rm K}_m\Big)+\frac{h^5}{80}\Big(
		16\, {\rm H}_{y_0}^2(\Delta{\rm H}_m)^2 \notag\\
		&-8\, {\rm H}_{y_0}\Delta{\rm H}_m\,\Delta{\rm K}_m
		-4\, {\rm K}_{y_0}\,{(}\Delta{\rm H}_m)^2
		+(\Delta{\rm K}_m)^2
		\Big)
		\,{+O(h^7)}
		\Bigg],
	\end{align}
	which is the dimensional reduction of the pure volumetric  term 
	$(\det \nabla_\xi \varphi_\xi)^2$ appearing in the three-dimensional 
	Ciarlet--Geymonat energy. {
		Let us anticipate that this Taylor-expanded contribution, which will be
		used only in Model~III, is convex in the variables
		$({\rm a}_m,{\rm a}_m{\rm H}_m,{\rm a}_m{\rm K}_m)$, as shown below,
		under an additional explicit restriction on $h$. This restriction is
		not needed for Models~I and II, where the complete volumetric
		contribution is kept in the positive-weight quadrature structure \cite{Davis}:
		Boole's rule is used in Model~I and Simpson's rule in Model~II.}  {Moreover, if Model~III} is truncated at order ${\rm O}(h^3)$, the convexity is lost {since the term $(\Delta{\rm K}_m)^2$ is absent}, and, in this case, another approximation of the pure volumetric  term 
	$(\det \nabla_\xi \varphi_\xi)^2$ should be considered. {This loss of convexity is an algebraic consequence of the {truncation of order $h^3$} of this particular Taylor-expanded contribution and should not be interpreted as a separate physical instability of the 3D body, even though this loss of convexity is important for the existence of minimizers
		of the resulting two-dimensional shell model.}
	
	We now address the reduction of the volumetric term $-\log(\det F_\xi)$.
	In contrast to {$\det \nabla_\xi\varphi_\xi$}, the quantity $-\log(\det F_\xi)$
	does not have a polynomial-type structure. A direct Taylor expansion
	would therefore change the favorable structure of the energy and lead to
	a reduced model for which the lower semicontinuity of the internal shell
	energy becomes unclear. For this reason, following the approach and the results of Anicic
	\cite{anicic1999formulation,anicic2018polyconvexity,anicic2019existence},
	we adopt a different strategy for the reduction of this term.
	Although rigorous shell theories are often obtained through asymptotic
	methods or $\Gamma$-convergence techniques, we approximate the thickness
	integration of the three-dimensional energy term involving
	$-\log(\det F_\xi)$ by means of a quadrature rule, a procedure commonly
	employed in computational shell mechanics
	\cite{hokkanen2020quadrature,cirak2000subdivision}. {More precisely, we approximate the corresponding thickness integrals by
		positive-weight quadrature rules chosen according to the intended order
		of the reduced model.}
	
	Although the following direct Taylor approximation of the logarithmic
	volumetric term will not be used in any of the final models, we briefly
	present it in order to understand why a quadrature-based treatment is
	preferable for this contribution. This intermediate computation
	is useful for understanding the structure of the reduced energy and
	for assessing the lower semicontinuity properties of the resulting
	functional. More precisely, {after a Taylor expansion of $\log(\det F_\xi)$ about $x_3=0$,} we would have	\footnote{For the following Taylor expansion we additionally assume that the
		deformed Kirchhoff--Love parametrization remains regular throughout the
		thickness, namely
		$
		\frac h2
		\max\{|\kappa_{1,m}|,|\kappa_{2,m}|\}<1
		$ $\text{a.e. in }\omega.
		$
		This condition is incorporated later in the
		admissible class \(V^h\).}
	\begin{align}
		\int_{-h/2}^{h/2} \log(\det F_\xi)\det \nabla_x \Theta(x',x_3)\,{\rm d}{x}_3
		&=
		{\rm a}_{y_0}\Bigg[
		h\,\log\!\left(\frac{{\rm a}_m}{{\rm a}_{y_0}}\right)
		+\frac{h^3}{12}\Bigg(
		\, {\rm K}_{y_0}\log\!\left(\frac{{\rm a}_m}{{\rm a}_{y_0}}\right)
		+\Delta{\rm K}_m -2(\Delta{\rm H}_m )^2
		\Bigg)\\[0.4em]
		&\qquad\quad
		+\frac{h^5}{80}\Bigg(
		-4(\Delta{\rm H}_m )^4
		-\frac{32}{3}\, {\rm H}_{y_0}(\Delta{\rm H}_m )^3
		+(2\, {\rm K}_{y_0}-8\, {\rm H}_{y_0}^2)(\Delta{\rm H}_m )^2\notag\\
		&\quad\qquad\quad
		+4\, {\rm H}_{y_0}\Delta{\rm H}_m \,\Delta{\rm K}_m 
		+4(\Delta{\rm H}_m )^2\Delta{\rm K}_m 
		-\frac12(\Delta{\rm K}_m )^2\notag
		\Bigg)
		\Bigg]\,{+O(h^7)}.
	\end{align}
	We note that the obtained energy is expressed in terms of the
	differences $\, {\rm H}_m-\, {\rm H}_{y_0}$ and $\, {\rm K}_m-\, {\rm K}_{y_0}$. {However, the quantities
		\({\rm H}_m-{\rm H}_{y_0}\) and
		\({\rm K}_m-{\rm K}_{y_0}\)
		do not possess, in general, the separate weak compactness properties
		available for the area-weighted variables
		\({\rm a}_m\Delta{\rm H}_m\) and
		\({\rm a}_m\Delta{\rm K}_m\); see
		Remark~\ref{sec3-subsec1-rem1}. }
	
	{\,
		In contrast, these difficulties do not arise if the thickness integration
		is approximated by a suitable positive-weight quadrature rule. To treat the
		lower- and higher-order reductions in a unified manner, let
		\begin{equation}\label{general-quadrature}
			\mathcal Q_h[\theta]
			:=
			h\sum_{j\in J}\omega_j \theta(h\zeta_j),
			\qquad
			-\frac12\leq \zeta_j\leq\frac12,
			\qquad
			\omega_j>0,
			\qquad
			\sum_{j\in J}\omega_j=1,
		\end{equation}
		be a symmetric quadrature rule for a function $\theta\in C([-h/2,h/2])
		$, 
		with \(J=-J\), i.e.,
		$
		\zeta_{-j}=-\zeta_j,
		$ $
		\omega_{-j}=\omega_j.
		$
		For \(s\in[-1/2,1/2]\), introduce
		\begin{equation}\label{A-general}
			\mathcal A_m(s)
			:=
			1-2hs\,{\rm H}_m+h^2s^2\,{\rm K}_m,
			\qquad
			\mathcal A_{y_0}(s)
			:=
			1-2hs\,{\rm H}_{y_0}+h^2s^2\,{\rm K}_{y_0}.
		\end{equation}
		Thus
		$
		\det\nabla_x\Theta(x',hs)
		=
		{\rm a}_{y_0}\mathcal A_{y_0}(s),
		$ $
		\det F_\xi(x',hs)
		=
		\frac{{\rm a}_m\mathcal A_m(s)}
		{{\rm a}_{y_0}\mathcal A_{y_0}(s)},
		$
		and, for deformations for which the corresponding normal parametrization
		remains regular, 
		the logarithmic volumetric contribution and  the quadratic volumetric contribution, respectively, are approximated by
		\begin{align}
			&\int_{-h/2}^{h/2}
			\log(\det F_\xi)\,
			\det\nabla_x\Theta(x',x_3)\,{\rm d}x_3
			\approx
			{\rm a}_{y_0}h
			\sum_{j\in J}
			\omega_j\,\mathcal A_{y_0}(\zeta_j)
			\log\!\left(
			\frac{{\rm a}_m\mathcal A_m(\zeta_j)}
			{{\rm a}_{y_0}\mathcal A_{y_0}(\zeta_j)}
			\right)\notag
			\label{general-quadrature-log}
			\\
			&\int_{-h/2}^{h/2}
			(\det\nabla_\xi\varphi_\xi)^2
			\det\nabla_x\Theta(x',x_3)\,{\rm d}x_3
			\approx
			{\rm a}_{y_0}h
			\sum_{j\in J}
			\omega_j\,\mathcal A_{y_0}(\zeta_j)
			\left(
			\frac{{\rm a}_m\mathcal A_m(\zeta_j)}
			{{\rm a}_{y_0}\mathcal A_{y_0}(\zeta_j)}
			\right)^2.
		\end{align}
		
		Since the quadrature
		weights are positive, and since the two nonlinear functions entering the
		volumetric Ciarlet--Geymonat energy, namely
		\(x\mapsto-\log x\) and \(x\mapsto x^2\), are convex on
		\((0,\infty)\), the quadrature approximation preserves a convex dependence
		on these weakly compact variables. This is precisely the structure needed
		in the subsequent weak lower semicontinuity argument; see also Anicic
		\cite{anicic1999formulation,anicic2018polyconvexity,anicic2019existence}.
		
		Moreover, the order of the quadrature rule can be chosen consistently with
		the intended order of the reduced shell model. For the model retaining
		terms through order \(h^3\), we use Simpson's rule. More precisely, taking
		\[
		J_S=\{-1,0,1\},
		\qquad
		(\zeta_{-1},\zeta_0,\zeta_1)
		=
		\left(-\frac12,0,\frac12\right),
		\]
		and
		\begin{equation}\label{Simpson-general}
			(\omega_{-1},\omega_0,\omega_1)
			=
			\frac16(1,4,1),
		\end{equation}
		formula \eqref{general-quadrature} becomes the classical Simpson rule.
		It is exact for all polynomials of degree at most \(3\). Moreover, for
		$\theta\in C^4([-h/2,h/2])$, its local quadrature error is of order
		\(O(h^5)\). In the present application this regularity is satisfied,
		for every fixed admissible deformation, as long as the corresponding
		normal parametrizations remain uniformly away from degeneracy.  In this case, we have 
		\[
		\mathcal A_m\!\left(-\frac12\right)
		=
		1+h\,{\rm H}_m+\frac{h^2}{4}{\rm K}_m
		=
		{\rm A}_m^-,
		\qquad 
		\mathcal A_m(0)=1,
		\qquad
		\mathcal A_m\!\left(\frac12\right)
		=
		1-h\,{\rm H}_m+\frac{h^2}{4}{\rm K}_m
		=
		{\rm A}_m^+,
		\]
		and analogously for the reference quantities
		\({\rm A}_{y_0}^{\pm}\). Hence
		\eqref{general-quadrature-log} 
		reduce  to
		\begin{align}
			&\int_{-h/2}^{h/2}
			\log(\det F_\xi)\det\nabla_x\Theta(x',x_3)\,{\rm d}x_3
			\approx
			{\rm a}_{y_0}\frac{h}{6}
			\Bigg[
			{\rm A}_{y_0}^-
			\log\!\left(
			\frac{{\rm a}_m{\rm A}_m^-}
			{{\rm a}_{y_0}{\rm A}_{y_0}^-}
			\right)
			+
			4\log\!\left(\frac{{\rm a}_m}{{\rm a}_{y_0}}\right)
			+
			{\rm A}_{y_0}^+
			\log\!\left(
			\frac{{\rm a}_m{\rm A}_m^+}
			{{\rm a}_{y_0}{\rm A}_{y_0}^+}
			\right)
			\Bigg],
			\label{Simpson-log}\notag
			\\
			&\int_{-h/2}^{h/2}
			(\det\nabla_\xi\varphi_\xi)^2
			\det\nabla_x\Theta(x',x_3)\,{\rm d}x_3
			\approx
			{\rm a}_{y_0}\frac{h}{6}
			\Bigg[
			{\rm A}_{y_0}^-
			\left(
			\frac{{\rm a}_m{\rm A}_m^-}
			{{\rm a}_{y_0}{\rm A}_{y_0}^-}
			\right)^2
			+
			4\frac{{\rm a}_m^2}{{\rm a}_{y_0}^2}
			+
			{\rm A}_{y_0}^+
			\left(
			\frac{{\rm a}_m{\rm A}_m^+}
			{{\rm a}_{y_0}{\rm A}_{y_0}^+}
			\right)^2
			\Bigg].
		\end{align}
		
		For the model retaining terms through order \(h^5\), we instead employ
		Boole's rule. Taking
		\[
		J_B=\{-2,-1,0,1,2\},
		\qquad 
		(\zeta_{-2},\zeta_{-1},\zeta_0,\zeta_1,\zeta_2)
		=
		\left(
		-\frac12,-\frac14,0,\frac14,\frac12
		\right),
		\]
		and
		\begin{equation}\label{Boole-general}
			(\omega_{-2},\omega_{-1},\omega_0,\omega_1,\omega_2)
			=
			\frac1{90}(7,32,12,32,7),
		\end{equation}
		formula \eqref{general-quadrature} becomes Boole's rule. This rule is exact for all polynomials of degree at most \(5\). Moreover,
		for $\theta\in C^6([-h/2,h/2])$, its local quadrature error is of order
		\(O(h^7)\). In the present application this regularity is satisfied,
		for every fixed admissible deformation, as long as the corresponding
		normal parametrizations remain uniformly away from degeneracy. Consequently, when combined
		with the non-volumetric expansion retained through order \(h^5\), the
		quadrature error is consistent with an approximation through order
		\(h^5\). By contrast, Simpson's rule, with an \(O(h^5)\) quadrature
		error, is  matched to the model in which the non-volumetric
		contribution is retained through order \(h^3\).
	}
	
	The last step is the approximation of the external load potentials in the shell model. To this aim, we next perform direct integration over the thickness. Thus, from \eqref{defTheta} and \eqref{KLa} we find
	$
	\tilde{v}(x)  =\varphi(x) - \Theta(x)= (m-y_0) + x_3 (n_m -n_{y_0}).
	$
	We insert this into the external load potentials to obtain the simplified form
	\begin{equation}\label{e2o}
		\begin{array}{l}
			\dd\int_{\Omega_h} \bigl\langle  \tilde{f}, \tilde{v} \bigr\rangle   \, {\rm d}{x}\, \,=\, 
			\int_{\omega} \left( \bigl\langle \int_{-h/2}^{h/2} \tilde{f}\,{\rm d}{x}_3 ,   m-y_0 \bigr\rangle    + 
			\bigl\langle \int_{-h/2}^{h/2} x_3 \tilde{f}\,{\rm d}{x}_3 ,  n_m-n_{y_0} \bigr\rangle   \right) {\rm d}{x}'\,.
		\end{array}
	\end{equation}
	Denoting {by} $ \tilde{t}^{\pm}(x'):\,=\, \tilde{t}(x', \pm \frac{h}{2} ) $ and taking into account that $ \Gamma_t \,=\,  \Big(\omega\times \Big\{\frac{h}{2}\Big\} \Big) \cup \Big(\omega\times \Big\{-\frac{h}{2}\Big\} \Big) \cup \Big(\gamma_t\times (-\frac{h}{2} , \frac{h}{2})\Big) $, we obtain {similarly}
	\begin{align} \label{e3o}
		\dd\int_{\Gamma_t} \bigl\langle  \tilde{t},  \tilde{v} \bigr\rangle   \, {\rm d}{x}  \,=\, & \int_{\omega}  \bigl\langle  \tilde{t}^{+} + \tilde{t}^{-}  ,   m-y_0  \bigr\rangle    {\rm d}{x}'
		+ 
		\int_{\omega}  \bigl\langle  
		\frac{h}{2}  (\tilde{t}^{+} - \tilde{t}^{-}), n_m-n_{y_0}   \bigr\rangle   \,
		{\rm d}{x}'  \notag\\
		& + \int_{\gamma_t}   \bigl\langle  \int_{-h/2}^{h/2} \tilde{t}\,{\rm d}{x}_3 ,  m-y_0  \bigr\rangle   {\rm d}{x}' + 
		\int_{\gamma_t}   \bigl\langle  \int_{-h/2}^{h/2} x_3\tilde{t}\,{\rm d}{x}_3 ,   n_m- n_{y_0}  \bigr\rangle   {\rm d}x'.
	\end{align}
	\,{Here, $\mathrm{d}x'$ denotes either the area element or length element.} With \eqref{e2o} and \eqref{e3o}, the potential of external applied loads in the three-dimensional model can be written in the form
	\begin{align}\label{e4o}
		\mathcal{L}(m,n_m)\,=\,  \Pi_\omega(m,n_m) + \Pi_{\gamma_t}(m,n_m)\,,
	\end{align}
	with
	\begin{align}
		\Pi_\omega(m,n_m) \,=\, \dd\int_{\omega} \bigl\langle  \bar{f},  v \bigr\rangle   \, {\rm d}{x}' + {\Lambda_\omega(n_m)},\qquad \qquad 
		\Pi_{\gamma_t}(m,n_m)\,=\, \dd\int_{\gamma_t} \bigl\langle  \bar{t},  v \bigr\rangle   \, {\rm d}{x}'+ {\Lambda_{\gamma_t}(n_m)}\,,
	\end{align}
	where $ v(x_1,x_2) \,=\, m(x_1,x_2)-y_0(x_1,x_2) $ is the displacement vector of the midsurface and
	\begin{align} \label{e5o}
		\bar f & \,=\, \dd \int_{-h/2}^{h/2} \tilde{f}\,{\rm d}{x}_3 + (\tilde{t}^{+} + \tilde{t}^{-}), \qquad  \qquad  \qquad  \qquad  \qquad  \quad  \ 
		\bar t \,=\, \int_{-h/2}^{h/2} \tilde{t}\,{\rm d}{x}_3 \,,  \\
		\Lambda_\omega(n_m) & \,=\, \int_{\omega}  \bigl\langle  \dd \int_{-h/2}^{h/2} x_3\,\tilde{f}\,{\rm d}{x}_3 +
		\frac{h}{2}  (\tilde{t}^{+} - \tilde{t}^{-}), n_m-n_{y_0}  \bigr\rangle   \,
		{\rm d}{x}', \quad 
		\Lambda_{\gamma_t}(n_m)  \,=\, \int_{\gamma_t}  \bigl\langle  \dd \int_{-h/2}^{h/2} x_3\,\tilde{t}\,{\rm d}{x}_3 , n_m-n_{y_0} \bigr\rangle   \,
		{\rm d}x'.
		\notag
	\end{align}
	\,{We assume that the resulting load fields are such that \(\mathcal L\) extends to a continuous affine functional on \({\rm W}^{1,2}(\omega;\mathbb R^3)\times {\rm W}^{1,2}(\omega;\mathbb R^3)\).}

	\subsection{The proposed models}
	
	Before summarizing the resulting models, we mention that all constitutive
	coefficients are derived directly from the three-dimensional formulation,
	without any a posteriori fitting of two-dimensional constitutive parameters. Moreover, for simplicity, we will also impose the boundary conditions for the deformation of the midsurface $m$ {as follows:}
	$
	m\mid_{\gamma _{y_0}}\,=\,y_0$ \text{and} $n_m\mid_{\gamma_{y_0}} \,= n_{y_0},
	$
	where $\gamma_{y_0}$ is a nonempty relatively open subset of $\partial\omega$. The potential of the external applied loads $\mathcal{L}(m,n_m)$ is the same as that obtained in the previous subsection. All quantities appearing in the formulation of the model have been introduced in the dimensional reduction process described in the previous subsections.
	\begingroup
	\,
	
	For the variational formulation of the reduced models, it is useful
	to replace the scalar volumetric functions by the extended-valued functions
	on the whole real line, i.e., to consider 
	\begin{align}
		\overline{\mathcal H}_{\rm CG}(t)
		:=
		\begin{cases}
			\mathcal H_{\rm CG}(t), & t>0,\\[1mm]
			+\infty, & t\leq 0,
		\end{cases}
		\qquad 
		\text{and}
		\qquad 
		\overline{\mathcal G}_{\log}(t)
		:=
		\begin{cases}
			-\log t, & t>0,\\[1mm]
			+\infty, & t\leq 0.
		\end{cases}
	\end{align}
	These extensions do not modify the reduced energies on their
	orientation-preserving domain, where all their arguments are strictly
	positive.
	
	\endgroup
	
	{\,
		\subsubsection{Model I: a reduced model consistent through order $h^5$}
		
		Collecting the results of the previous subsections, we first consider a
		higher-order reduced model in which the quadratic shell contribution
		is retained through order $h^5$, while the complete pure volumetric contribution
		is approximated by Boole's rule. The corresponding two-dimensional
		variational problem consists in minimizing, with respect to
		$m:\omega\to\mathbb R^3$, the functional
		\begin{equation}
			\label{model1}
			\mathcal J_1(m)
			:=
			\int_\omega
			\Big[
			W_{\rm shell}^{(1)}
			({\rm I}_m,{\rm II}_m,{\rm III}_m)
			+
			W_{\rm vol}^{B}
			({\rm a}_m,{\rm a}_m{\rm H}_m,{\rm a}_m{\rm K}_m)
			\Big]\,
			{\rm det}(\nabla y_0|n_{y_0})\,{\rm d}x'
			-
			\mathcal L(m,n_m).
		\end{equation}
		Here
		\begin{align}
			\label{wshell1}
			W_{\rm shell}^{(1)}
			({\rm I}_m,{\rm II}_m,{\rm III}_m)
			:=&\,
			\frac{\mu}{2}\Bigg\{
			\left(
			h-\frac{h^3}{12}{\rm K}_{y_0}
			+\frac{h^5}{80}{\rm K}_{y_0}^2
			\right)\mathcal F_0({\rm I}_m)
			+
			\frac{h^3}{3}{\rm H}_{y_0}\mathcal F_0({\rm II}_m)
			\notag\\
			&+
			\left(
			\frac{h^3}{12}
			-\frac{h^5}{80}{\rm K}_{y_0}
			\right)\mathcal F_0({\rm III}_m)
			-
			\frac{h^5}{40}{\rm H}_{y_0}{\rm K}_{y_0}
			\mathcal F_1({\rm I}_m)
			\notag\\
			&-
			\left(
			\frac{h^3}{6}
			-\frac{h^5}{40}{\rm K}_{y_0}
			\right)\mathcal F_1({\rm II}_m)
			+
			\left[
			\frac{h^3}{12}
			+\frac{h^5}{80}
			\left(4{\rm H}_{y_0}^2-{\rm K}_{y_0}\right)
			\right]\mathcal F_2({\rm I}_m)
			\\
			&-
			\frac{h^5}{20}{\rm H}_{y_0}\mathcal F_2({\rm II}_m)
			+
			\frac{h^5}{80}\mathcal F_2({\rm III}_m)
			+
			\left(
			h+\frac{h^3}{12}{\rm K}_{y_0}
			\right)
			\Bigg\}-
			\frac{3\mu}{2}
			\left(
			h+\frac{h^3}{12}{\rm K}_{y_0}
			\right),\notag
		\end{align}
		and the pure volumetric contribution is
		\begin{align}
			\label{wvol1}
			&W_{\rm vol}^{B}
			({\rm a}_m,{\rm a}_m{\rm H}_m,{\rm a}_m{\rm K}_m)
			:=
			h\sum_{j\in J_B}
			\omega_j^{B}\,
			\mathcal A_{y_0}(\zeta_j^{B})\,
			\overline{\mathcal H}_{\rm CG}
			\left(
			\frac{
				{\rm a}_m\mathcal A_m(\zeta_j^{B})
			}{
				{\rm a}_{y_0}\mathcal A_{y_0}(\zeta_j^{B})
			}
			\right),
		\end{align}
		where
		$
		(\zeta_j^{B})_{j\in J_B}
		=
		\left(
		-\frac12,-\frac14,0,\frac14,\frac12
		\right),
		$ $
		(\omega_j^{B})_{j\in J_B}
		=
		\frac1{90}(7,32,12,32,7).
		$
		
		The two approximations entering Model~I have the same consistency level.
		Indeed, the quadratic contribution is retained through order $h^5$,
		with an integrated remainder of order $O(h^7)$, while Boole's rule is exact
		for polynomials of degree at most $5$ and has a local quadrature error of
		order $O(h^7)$. Hence, for every fixed sufficiently regular admissible
		deformation, Model~I is consistent through order $h^5$, with a local
		through-thickness remainder of order $O(h^7)$.

		\subsubsection{Model II: a reduced model consistent through order $h^3$}
		
		The second model is deliberately constructed as a lower-order
		approximation. We truncate the quadratic  shell contribution after the
		$h^3$-terms and approximate the complete pure volumetric contribution by
		Simpson's rule. The corresponding functional is
		\begin{equation}
			\label{model2}
			\mathcal J_2(m)
			:=
			\int_\omega
			\Big[
			W_{\rm shell}^{(2)}
			({\rm I}_m,{\rm II}_m,{\rm III}_m)
			+
			W_{\rm vol}^{S}
			({\rm a}_m,{\rm a}_m{\rm H}_m,{\rm a}_m{\rm K}_m)
			\Big]\,
			{\rm det}(\nabla y_0|n_{y_0})\,{\rm d}x'
			-
			\mathcal L(m,n_m).
		\end{equation}
		Here
		\begin{align}
			\label{wshell2}
			W_{\rm shell}^{(2)}
			({\rm I}_m,{\rm II}_m,{\rm III}_m)
			:=&\,
			\frac{\mu}{2}\Bigg\{
			\left(
			h-\frac{h^3}{12}{\rm K}_{y_0}
			\right)\mathcal F_0({\rm I}_m)
			+
			\frac{h^3}{3}{\rm H}_{y_0}\mathcal F_0({\rm II}_m)
			+
			\frac{h^3}{12}\mathcal F_0({\rm III}_m)
			\notag\\
			&-
			\frac{h^3}{6}\mathcal F_1({\rm II}_m)
			+
			\frac{h^3}{12}\mathcal F_2({\rm I}_m)
			+
			\left(
			h+\frac{h^3}{12}{\rm K}_{y_0}
			\right)
			\Bigg\}
			-
			\frac{3\mu}{2}
			\left(
			h+\frac{h^3}{12}{\rm K}_{y_0}
			\right),
		\end{align}
		and
		\begin{align}
			\label{wvol2}
			&W_{\rm vol}^{S}
			({\rm a}_m,{\rm a}_m{\rm H}_m,{\rm a}_m{\rm K}_m)
			:=
			h\sum_{j\in J_S}
			\omega_j^{S}\,
			\mathcal A_{y_0}(\zeta_j^{S})\,
			\overline{\mathcal H}_{\rm CG}
			\left(
			\frac{
				{\rm a}_m\mathcal A_m(\zeta_j^{S})
			}{
				{\rm a}_{y_0}\mathcal A_{y_0}(\zeta_j^{S})
			}
			\right),
		\end{align}
		where
		$
		(\zeta_j^{S})_{j\in J_S}
		=
		\left(-\frac12,0,\frac12\right),
		$ $
		(\omega_j^{S})_{j\in J_S}
		=
		\frac16(1,4,1).
		$
		
		Here the  expansion of the quadratic part of the  three-dimensional energy is truncated after the $h^3$-terms and
		therefore has an integrated remainder of order $O(h^5)$. Simpson's rule is
		exact for polynomials of degree at most $3$ and has a local quadrature
		error of order $O(h^5)$. Consequently, for every fixed sufficiently
		regular admissible deformation, Model~II is consistent through order
		$h^3$, with a local through-thickness remainder of order $O(h^5)$.

		\subsubsection{Model III: a mixed Taylor--Boole reduction consistent through order $h^5$}
		
		In the third model, the non-volumetric shell contribution is retained
		through order $h^5$, as in Model~I. The logarithmic volumetric contribution
		is again approximated by Boole's rule, while the quadratic volumetric
		contribution $(\det F_\xi)^2$ is replaced by its Taylor expansion retained
		through order $h^5$. Thus Model~III differs from Model~I only in the
		treatment of the quadratic volumetric term.
		
		The corresponding variational problem consists in minimizing
		\begin{align}
			\label{model3}
			\mathcal J_3(m)
			:=&\,
			\int_\omega
			\Big[
			W_{\rm shell}^{(1)}
			({\rm I}_m,{\rm II}_m,{\rm III}_m)
			+
			W_{\log}^{B}
			({\rm a}_m,{\rm a}_m{\rm H}_m,{\rm a}_m{\rm K}_m)
			\notag\\
			&\hspace{7em}
			+
			W_{\rm vol}^{T}
			({\rm a}_m,
			{\rm a}_m\Delta{\rm H}_m,
			{\rm a}_m\Delta{\rm K}_m)
			\Big]\,
			{\rm det}(\nabla y_0|n_{y_0})\,{\rm d}x'
			-
			\mathcal L(m,n_m),
		\end{align}
		where $W_{\rm shell}^{(1)}$ is defined in \eqref{wshell1},
		\begin{align}
			\label{wcurv3}
			&W_{\log}^{B}
			({\rm a}_m,{\rm a}_m{\rm H}_m,{\rm a}_m{\rm K}_m)
			:=
			\frac{\lambda+2\mu}{2}\,
			h\sum_{j\in J_B}
			\omega_j^{B}\,
			\mathcal A_{y_0}(\zeta_j^{B})
			\overline{\mathcal G}_{\log}\left(
			\frac{
				{\rm a}_m\mathcal A_m(\zeta_j^{B})
			}{
				{\rm a}_{y_0}\mathcal A_{y_0}(\zeta_j^{B})
			}
			\right),
		\end{align}
		and
		\begin{align}
			\label{wcurv4}
			&W_{\rm vol}^{T}
			({\rm a}_m,
			{\rm a}_m\Delta{\rm H}_m,
			{\rm a}_m\Delta{\rm K}_m)
			:=
			\frac{\lambda}{4}
			\Bigg\{
			\left(
			\frac{{\rm a}_m}{{\rm a}_{y_0}}
			\right)^2
			\Bigg[
			h
			+
			\frac{h^3}{12}
			\left(
			{\rm K}_{y_0}
			+
			4(\Delta{\rm H}_m)^2
			+
			2\Delta{\rm K}_m
			\right)
			\\
			&\hspace{8em}
			+
			\frac{h^5}{80}
			\Big(
			16{\rm H}_{y_0}^2(\Delta{\rm H}_m)^2
			-
			8{\rm H}_{y_0}\Delta{\rm H}_m\Delta{\rm K}_m
			-
			4{\rm K}_{y_0}(\Delta{\rm H}_m)^2
			+
			(\Delta{\rm K}_m)^2
			\Big)
			\Bigg]
			\notag\\
			&\hspace{8em}
			-
			\left(
			h+\frac{h^3}{12}{\rm K}_{y_0}
			\right)
			\Bigg\}.\notag
		\end{align}
		
		All approximation steps are consistent at the same order. The
		quadratic  part is retained through order $h^5$, with integrated
		remainder $O(h^7)$; the Taylor expansion of the quadratic volumetric
		contribution is retained through order $h^5$, with integrated remainder
		$O(h^7)$; and the Boole approximation of the logarithmic contribution has
		quadrature error $O(h^7)$. Therefore, for every fixed sufficiently regular
		admissible deformation, Model~III is also consistent through order $h^5$,
		with a local through-thickness remainder of order $O(h^7)$.
		
		The distinction between Models~I and III is thus not the formal consistency
		order, but the treatment of the quadratic volumetric contribution:
		Model~I retains the positive-weight Boole quadrature structure inherited
		directly from the three-dimensional energy, whereas Model~III replaces
		this contribution by its explicit Taylor approximation through order
		$h^5$. As shown below, this distinction is important for the convexity
		analysis and leads to an additional thickness restriction only for
		Model~III.

		All three reduced models share the  zero-energy
		character of the reference configuration. Indeed, for $m=y_0$ we have
		$
		{\rm I}_m={\rm I}_{y_0},
		$ $
		{\rm II}_m={\rm II}_{y_0},
		$ $
		{\rm III}_m={\rm III}_{y_0},
		$ $
		{\rm H}_m={\rm H}_{y_0},
		$ $
		{\rm K}_m={\rm K}_{y_0}.
		$
		Hence
		$
		\mathcal A_m(\zeta_j)
		=
		\mathcal A_{y_0}(\zeta_j),
		$ $
		\Delta{\rm H}_m
		=
		\Delta{\rm K}_m
		=
		0,
		$
		and therefore every quadrature ratio equals one. Since
		$
		\overline{\mathcal H}_{\rm CG}(1)=0,
		$ and $\overline{\mathcal G}_{\log}(1)=0$, 
		direct substitution gives
		$
		W_{\rm shell}^{(1)}
		=
		W_{\rm shell}^{(2)}
		=
		W_{\rm vol}^{B}
		=
		W_{\rm vol}^{S}
		=
		W_{\log}^{B}
		=
		W_{\rm vol}^{T}
		=
		0
		$
		in the reference configuration. Thus neither the dimensional reduction,
		the quadrature approximations, nor the Taylor approximation used in
		Model~III introduces a  residual stored energy in the reference
		configuration.
	}
	\section{The existence results } 
	\label{sec3}
	\subsection{A general existence theorem}
	\label{sec3-subsec1}
	In this subsection, we present a general existence theorem that will be used to establish the existence results for our proposed models. The result below is an adaptation, to the present notation and admissible class, of the existence theorem established by Anicic \cite{anicic2019existence, anicic2018polyconvexity}; related formulations were also used in \cite{giang2024existence,giang2025existence}. To begin with, we have the following lemma.
	\begin{lemma}
		\label{limit-identification}
		Let $\omega \subset \mathbb{R}^2$ be an open bounded domain with Lipschitz boundary. Let $y_0 \in C^2(\overline{\omega}, \mathbb{R}^3)$ be such that $\partial_{x_1} y_0(x')$ and $\partial_{x_2} y_0(x')$ are linearly independent for every $x' \in \overline{\omega}$ and let $\gamma _{y_0}$ be a nonempty relatively open subset of $\partial \omega$. For $h > 0$, $p \geq 2$ and $q > 1$, where $h$ satisfies \eqref{ch5in}, we define
		\begin{align}
			V^h := \bigg\{ m \in {\rm W}^{1,p}(\omega; \mathbb{R}^3)\,\big| &\,  {\rm a}_m \in {\rm L}^q(\omega), \,\  {\rm a}_m > 0 \text{ a.e. in } \omega, \notag\\
			& n_m \in {\rm W}^{1,p}(\omega; \mathbb{R}^3), \qquad
			\frac h2\max\{|\kappa_{1,m}|,|\kappa_{2,m}|\}<1 \text{ a.e. in } \omega, \\
			& m = y_0  \quad \text{and} \quad n_m = n_{y_0} \, \text{on} \quad \gamma_{y_0}\text{ in the trace sense} \bigg\}. \notag
		\end{align}

		Assume that $(m_k)$ is a sequence with $m_k \in V^h$ for all $k$, for which there exist $m \in {\rm W}^{1,p}(\omega; \mathbb{R}^3)$, $\varkappa \in {\rm W}^{1,p}(\omega; \mathbb{R}^3)$, $(\xi_1, \xi_2, \xi_3) \in ({\rm L}^q(\omega; \mathbb{R}^3))^3$ and $(\beta_1, \beta_2, \beta_3) \in ({\rm L}^q(\omega))^3$ such that
		\begin{align}
			\begin{array}{rcllrcll}
				m_k &\rightharpoonup& m &\text{ in } {\rm W}^{1,p}(\omega; \mathbb{R}^3), &n_{m_k} &\rightharpoonup& \varkappa &\text{ in } {\rm W}^{1,p}(\omega; \mathbb{R}^3),\vspace{2mm}
				\\
				\partial_1 m_k \times  \partial_2 m_k &\,{\rightharpoonup}& \xi_1 &\text{ in } {\rm L}^q(\omega; \mathbb{R}^3), & {\rm a}_{m_k} &\,{\rightharpoonup}& \beta_1 & \text{ in } {\rm L}^q(\omega),\vspace{2mm}
				\\
				{\rm H}_{m_k} \partial_1 m_k \times  \partial_2 m_k &\rightharpoonup& \xi_2 & \text{ in } {\rm L}^q(\omega; \mathbb{R}^3), & {\rm H}_{m_k} {\rm a}_{m_k} &\rightharpoonup& \beta_2 & \text{ in } {\rm L}^q(\omega),\vspace{2mm}
				\\
				{\rm K}_{m_k} \partial_1 m_k \times  \partial_2 m_k &\rightharpoonup& \xi_3 & \text{ in } {\rm L}^q(\omega; \mathbb{R}^3), & {\rm K}_{m_k} {\rm a}_{m_k} &\rightharpoonup& \beta_3 &\text{ in } {\rm L}^q(\omega).
			\end{array}
		\end{align}
		Assume further that
		$
		{\rm a}_m > 0 \  \textrm{ a.e. in } \  \omega.
		$
		Then, almost everywhere in $\omega$, we have   
		\begin{align}
			\varkappa &= n_m, \quad  \qquad \qquad \ \ \ \,\qquad\xi_1 = \partial_1 m \times  \partial_2 m,\notag
			\\
			\xi_2 &=\,{\rm H}_m \partial_1 m \times  \partial_2 m, \quad \quad \quad \, \xi_{3} =\,{\rm K}_m \partial_1 m \times  \partial_2 m,
			\\
			\beta_1 &= {\rm a}_m, \quad \quad\beta_2 =\,{\rm H}_m {\rm a}_m \quad\quad \text{and} \quad\quad \beta_3 =\,{\rm K}_m {\rm a}_m.\notag
		\end{align}
	\end{lemma}
	
	\begin{proof}
		The main ingredient of the proof is the compensated compactness argument, which is by now classical in nonlinear elasticity.    The proof follows from the same argument as in part (iii) of the proof of \cite[Theorem 1]{anicic2018polyconvexity}, and for this reason,  some details will be omitted. 
		\,{However, for the reader's convenience, we indicate the essential compactness argument. }
		\,{Since $p\geq2$ and $\omega\subset\mathbb R^2$, the compact Sobolev embeddings imply, after passing to a subsequence, that $m_k\to m$ and $n_{m_k}\to\varkappa$ strongly in ${\rm L}^r(\omega;\mathbb R^3)$ for every finite $r$. We first recall the elementary {compensated compactness} mechanism which will be used below. If $u_k\rightharpoonup u$ and $v_k\rightharpoonup v$ in ${\rm W}^{1,p}(\omega)$, then
			$
			\partial_1u_k\,\partial_2v_k-\partial_2u_k\,\partial_1v_k
			=
			\partial_1(u_k\partial_2v_k)-\partial_2(u_k\partial_1v_k)
			\to
			\partial_1u\,\partial_2v-\partial_2u\,\partial_1v
			$
			in $\mathcal D'(\omega)$, since $u_k\to u$ strongly in ${\rm L}^{p'}(\omega)$ and $\nabla v_k\rightharpoonup\nabla v$ in ${\rm L}^p(\omega)$. Applying this identity componentwise we obtain 
			$
			\partial_1m_k\times \partial_2m_k
			\to
			\partial_1m\times \partial_2m
			\ \text{in }\mathcal D'(\omega).
			$
			Since the left-hand side converges weakly in ${\rm L}^q$ to $\xi_1$, uniqueness of the distributional limit yields $\xi_1=\partial_1m\times \partial_2m$.}
		\,{Moreover, $|n_{m_k}|=1$ and $\langle n_{m_k},\partial_\alpha m_k\rangle=0$, $\alpha=1,2$. Passing to the limit by the strong convergence of $n_{m_k}$ and the weak convergence of $\nabla m_k$, we obtain $|\varkappa|=1$ and $\langle\varkappa,\partial_\alpha m\rangle=0$. Hence $\varkappa=\pm n_m$. On the other hand, ${\rm a}_{m_k}=\langle n_{m_k},\partial_1m_k\times \partial_2m_k\rangle$. Taking $q'=q/(q-1)$ and using the strong convergence $n_{m_k}\to\varkappa$ in ${\rm L}^{q'}$ we deduce
			$
			\beta_1=\langle\varkappa,\xi_1\rangle.
			$
			Since ${\rm a}_{m_k}>0$, its weak limit satisfies $\beta_1\geq0$, whereas ${\rm a}_m=|\partial_1m\times \partial_2m|>0$. Therefore $\varkappa=n_m$ and $\beta_1={\rm a}_m$.}
		
		\,{For the mean curvature we use the geometric identity
			$
			2\,{\rm H}_{m_k}\,
			(\partial_1m_k\times \partial_2m_k)
			=
			\partial_2(n_{m_k}\times \partial_1m_k)
			-\partial_1(n_{m_k}\times \partial_2m_k)
			$
			in the sense of distributions. The right-hand side passes to the limit because $n_{m_k}\to n_m$ strongly in ${\rm L}^{p'}$ and $\nabla m_k\rightharpoonup\nabla m$ in ${\rm L}^p$. Hence
			$
			\xi_2={\rm H}_m\,\partial_1m\times \partial_2m.
			$}
		\,{Similarly, the Weingarten relations give
			$
			{\rm K}_{m_k}\,
			(\partial_1m_k\times \partial_2m_k)
			=
			\partial_1n_{m_k}\times \partial_2n_{m_k}.
			$
			The right-hand side is again a vector of two-dimensional Jacobians, and the preceding {compensated compactness} argument yields
			$\xi_3={\rm K}_m\,\partial_1m\times \partial_2m$.}
		
		\,{Then,
			${\rm H}_{m_k}{\rm a}_{m_k}
			=\langle n_{m_k},{\rm H}_{m_k}\partial_1m_k\times \partial_2m_k\rangle$
			and analogously for ${\rm K}_{m_k}{\rm a}_{m_k}$. Using once more the strong convergence of $n_{m_k}$ in ${\rm L}^{q'}$ and the weak ${\rm L}^q$ convergences assumed in the statement, we obtain
			$\beta_2={\rm H}_m{\rm a}_m$ and $\beta_3={\rm K}_m{\rm a}_m$, which completes the proof.}
	\end{proof}

	\begin{remark}
		\label{sec3-subsec1-rem1}
		\begin{itemize}
			\item[]
			\item[(i)] In the previous lemma, we are able to show that if $m_k \rightharpoonup m$, then, {under the assumptions of Lemma \ref{limit-identification}, the weak limits of the area-weighted curvatures are identified as}
			$$
			{\rm H}_{m_k} {\rm a}_{m_k} \rightharpoonup\,{\rm H}_m {\rm a}_m \qquad\textrm{ and } \qquad{\rm K}_{m_k} {\rm a}_{m_k} \rightharpoonup\,{\rm K}_m {\rm a}_m.
			$$
			However, {the compactness argument used here does not imply, in general, the separate weak convergences }    $
			{\rm H}_{m_k} \rightharpoonup\,{\rm H}_m$ \textrm{ and }$ {\rm K}_{m_k} \rightharpoonup\,{\rm K}_m.
			$
			Therefore, to establish our existence results, we will need to check the convexity of energetic terms with respect to $\, {\rm H}_m {\rm a}_m$ and $\, {\rm K}_m {\rm a}_m$ instead of $\, {\rm H}_m$ and $\, {\rm K}_m$.
			\item[(ii)] The following is a direct consequence of Lemma \ref{limit-identification} under the same assumptions:
			$$
			{\rm a}_{m_k}\Delta{\rm H}_{m_k} \rightharpoonup {\rm a}_m \Delta{\rm H}_m\qquad \textrm{ and } \qquad {\rm a}_{m_k} \Delta{\rm K}_{m_k}  \rightharpoonup {\rm a}_m\,\,{\Delta}{\rm K}_m.
			$$
		\end{itemize}
		
	\end{remark}
	
	We are now in a position to state a general existence result.
	
	\begin{theorem}
		\label{general-existence}
		Let $\omega \subset \mathbb{R}^2$ be an open bounded domain with Lipschitz boundary. Let $y_0 \in C^2(\overline{\omega}, \mathbb{R}^3)$ be such that $\partial_{x_1} y_0(x')$ and $\partial_{x_2} y_0(x')$ are linearly independent for every $x' \in \overline{\omega}$ and let $\gamma _{y_0}$ be a nonempty relatively open subset of $\partial \omega$. For $h > 0$, $p \geq 2$ and $q > 1$, where $h$ satisfies \eqref{ch5in}, we define the set of admissible deformations $V^h$ as in Lemma \ref{limit-identification}, and for each $m \in V^h$,
		$$
		\mathcal{I}(m) := \int_\omega W(x', m) \, {\rm d}{x}' - \mathcal L(m, n_m),
		$$
		where $\mathcal L$ is  a continuous {affine functional on}
		the space ${\rm W}^{1,p}(\omega; \mathbb{R}^3) \times {\rm W}^{1,p}(\omega; \mathbb{R}^3)$ and $W : \omega \times V^h \to \mathbb{R}$ is a function with the following properties:
		\begin{enumerate}
			\item \textbf{Polyconvexity:} {\, 
				For almost every $x'\in\omega$, there exists a proper, convex and
				lower semicontinuous function
				$
				\mathbb W(x',\cdot):
				(\mathbb R^{3\times2})^2\times\mathbb R^3
				\to 
				\mathbb R\cup\{+\infty\},
				$
				whose effective domain contains
				\[
				\mathbf M
				:=
				\left\{
				(A,B,a,b,c)\in
				(\mathbb R^{3\times2})^2\times\mathbb R^3
				\,\middle|\,
				a-\frac h2|b|>0,\quad
				a-h|b|+\frac{h^2}{4}c>0
				\right\},
				\]
				and such that
				$
				W(x',m)
				=
				\mathbb W
				\bigl(
				x',
				\nabla m,
				\nabla n_m,
				{\rm a}_m,
				{\rm a}_m{\rm H}_m,
				{\rm a}_m{\rm K}_m
				\bigr)
				$
				for every $m\in V^h$ and almost every $x'\in\omega$.}
			\item  \textbf{Measurability:} For every
			\[
			(A,B,a,b,c)\in
			(\mathbb R^{3\times2})^2\times\mathbb R^3,
			\]
			the map
			\[
			x'\mapsto
			\mathbb W(x',A,B,a,b,c)
			\]
			is measurable.    \item \textbf{Coercivity:} There exist constants $C_1 > 0$ and $C_2$ such that
			$$
			W(x', m) \geq C_1 \{ |\nabla m|^p + |\nabla n_m|^{p} + {\rm a}_m^{q} \} + C_2
			$$
			for all $m \in V^h$ and almost all $x' \in \omega$.
			\,{\item \textbf{Orientation-preserving {condition}:}
				There exist {lower semicontinuous} functions
				$
				\Psi_0,\Psi_+,\Psi_-:(0,\infty)\to\mathbb R,
				$
				bounded from below and satisfying
				$
				\lim_{t\to0^+}\Psi_0(t)
				=
				\lim_{t\to0^+}\Psi_+(t)
				=
				\lim_{t\to0^+}\Psi_-(t)
				=+\infty,
				$
				and a constant \(C_3\in\mathbb R\) such that, for all \(m\in V^h\) and
				almost every \(x'\in\omega\),
				\[
				W(x',m)
				\ge
				\Psi_0({\rm a}_m)
				+
				\Psi_+({\rm a}_m{\rm A}_m^+)
				+
				\Psi_-({\rm a}_m{\rm A}_m^-)
				+
				C_3.
				\]}
		\end{enumerate}
		
		Assume that $\inf_{m \in V^h} \mathcal{I}(m) < +\infty$, then there exists at least one function $m^* \in V^h$ such that
		$
		\mathcal{I}(m^*) = \inf_{m \in V^h} \mathcal{I}(m).
		$
	\end{theorem}
	
	\begin{proof}
		
		\,{ The proof follows the {direct method approach} developed by Anicic
			\cite{anicic2018polyconvexity,anicic2019existence}. Since the present
			admissible class and the orientation-preserving condition are \, {analogous to those of \cite{anicic2018polyconvexity,anicic2019existence}}, we {only highlight} the main steps and indicate explicitly the compactness
			and closure arguments needed below. }
		
		\,{ Let \((m_k)\subset V^h\) be a minimizing sequence. By the continuity of
			\(\mathcal L\), the prescribed boundary conditions and Poincar\'e's
			inequality, there exists \(C>0\) such that
			\[
			|\mathcal L(m_k,n_{m_k})|
			\le
			C\bigl(
			1+\|\nabla m_k\|_{{\rm L}^p}
			+\|\nabla n_{m_k}\|_{{\rm L}^p}
			\bigr).
			\]
			Hence, by Young's inequality, for every sufficiently small
			\(\varepsilon>0\), we have
			\[
			|\mathcal L(m_k,n_{m_k})|
			\le
			\varepsilon\bigl(
			\|\nabla m_k\|_{{\rm L}^p}^p
			+
			\|\nabla n_{m_k}\|_{{\rm L}^p}^p
			\bigr)
			+
			C_\varepsilon.
			\]
			Combining this estimate with the coercivity assumption and the fact that
			\((m_k)\) is minimizing yields uniform bounds for \(m_k\) and
			\(n_{m_k}\) in \({\rm W}^{1,p}\) and for \({\rm a}_{m_k}\) in \({\rm L}^q\).
			Consequently, after passing to a subsequence,
			$$
			m_k\rightharpoonup m,\qquad n_{m_k}\rightharpoonup \varkappa
			\quad\text{in }{\rm W}^{1,p}(\omega;\mathbb R^3),\qquad
			{\rm a}_{m_k}\rightharpoonup\beta_1\quad\text{in }{\rm L}^q(\omega).
			$$}\,{In addition, from the definition of $V^h$ we have $|\kappa_i(m_k)|<2/h$, and therefore $|{\rm H}_{m_k}|<2/h$ and $|{\rm K}_{m_k}|<4/h^2$. Thus the sequences ${\rm H}_{m_k}{\rm a}_{m_k}$ and ${\rm K}_{m_k}{\rm a}_{m_k}$ are bounded in ${\rm L}^q(\omega)$. The same holds for the corresponding vector quantities since
			$\partial_1m_k\times \partial_2m_k={\rm a}_{m_k}n_{m_k}$. Hence, after considering  a
			subsequence (denote by the same notation), all the weak limits appearing in Lemma
			\ref{limit-identification} can be extracted.
			At this stage we first use only the compensated compactness argument
			from the proof of Lemma \ref{limit-identification}. It gives
			\[
			\xi_1=\partial_1m\times \partial_2m,\qquad
			|\varkappa|=1,\qquad
			\langle\varkappa,\partial_\alpha m\rangle=0,\quad \alpha=1,2.
			\]
			Moreover, since
			${\rm a}_{m_k}
			=\langle n_{m_k},\partial_1m_k\times \partial_2m_k\rangle$,
			we obtain
			$
			\beta_1=\langle\varkappa,\xi_1\rangle.
			$
			Since ${\rm a}_{m_k}\geq0$, we have $\beta_1\geq0$. On the other hand,
			$|\varkappa|=1$ and $\varkappa$ is orthogonal to
			$\partial_1m,\partial_2m$, hence
			$\langle\varkappa,\xi_1\rangle=\pm|\xi_1|$. Consequently,
			$
			\beta_1=|\xi_1|
			=|\partial_1m\times \partial_2m|
			={\rm a}_m\geq0.
			$
			The strict inequality ${\rm a}_m>0$ a.e. follows from the orientation-preserving condition. Indeed, if
			$E:=\{x'\in\omega;\ {\rm a}_m(x')=0\}$ had positive measure, then
			${\rm a}_{m_k}\rightharpoonup{\rm a}_m$ in ${\rm L}^q(\omega)$ and
			${\rm a}_{m_k}\geq0$ would imply
			$\int_E{\rm a}_{m_k}\,{\rm d}x'\to0$. Hence, after passing to a
			subsequence, ${\rm a}_{m_k}\to0$ a.e. in $E$. The {orientation-preserving} assumption, Fatou's lemma, and the fact that
			\(\Psi_+\) and \(\Psi_-\) are bounded from below then imply
			$$
			\liminf_{k\to\infty}\int_E W(x',m_k)\,{\rm d}x'=+\infty,
			$$
			contradicting the boundedness of the energy along the minimizing
			sequence. Thus ${\rm a}_m>0$ a.e., and Lemma
			\ref{limit-identification} can now be applied.
			We obtain
			\[
			\varkappa=n_m,\qquad
			({\rm a}_{m_k},{\rm H}_{m_k}{\rm a}_{m_k},
			{\rm K}_{m_k}{\rm a}_{m_k})
			\rightharpoonup
			({\rm a}_m,{\rm H}_m{\rm a}_m,{\rm K}_m{\rm a}_m)
			\quad\text{in }{\rm L}^q(\omega;\mathbb R^3).
			\]}\,{Let us also note why the set $\mathbf M$ used here is sufficient,
			although {it is analogous} to the set considered in
			\cite{anicic2018polyconvexity}. For any $m\in V^h$,  since $|{\rm H}_m|<2/h$, we have
			${\rm a}_m-\frac h2|{\rm H}_m{\rm a}_m|={\rm a}_m(1-\frac h2|{\rm H}_m|)>0$.
			Furthermore,  $|h\,\frac{\kappa_{1,m}}{2}|,|h\,\frac{\kappa_{2,m}}{2}|<1$, and 
			$$
			{\rm a}_m-h|{\rm H}_m{\rm a}_m|+\frac{h^2}{4}{\rm K}_m{\rm a}_m
			={\rm a}_m\,[1-|h\,\frac{\kappa_{1,m}}{2}+h\,\frac{\kappa_{2,m}}{2}|+h\,\frac{\kappa_{1,m}}{2}\,h\,\frac{\kappa_{2,m}}{2}]>0.
			$$
			Indeed, the last factor equals $(1-h\,\frac{\kappa_{1,m}}{2})(1-r)$ if $h\,\frac{\kappa_{1,m}}{2}+\frac{\kappa_{2,m}}{2}\geq0$ and $(1+h\,\frac{\kappa_{1,m}}{2})(1+h\,\frac{\kappa_{2,m}}{2})$ if $h\,\frac{\kappa_{1,m}}{2}+h\,\frac{\kappa_{2,m}}{2}h<0$. Hence all geometric states generated by $V^h$ belong to $\mathbf M$. }
		
		\,{The trace conditions pass to the weak limit. Moreover, since
			$m_k\in V^h$, we have
			$
			{\rm a}_{m_k}-\frac h2
			|{\rm a}_{m_k}{\rm H}_{m_k}|>0
			\ \text{a.e. in }\omega.
			$
			The set
			$
			\left\{(a,b,c)\in\mathbb R^3;\,
			a-\frac h2|b|\geq0\right\}
			$
			is closed and convex. Hence, using the weak convergence established
			above, we obtain
			$
			{\rm a}_m-\frac h2|{\rm a}_m{\rm H}_m|\geq0
			\ \text{a.e. in }\omega.
			$
			The same orientation-preserving blow-up argument gives
			${\rm a}_m{\rm A}_m^\pm>0$ a.e. Since ${\rm a}_m>0$, we have
			$
			|h\,\frac{\kappa_{1,m}}{2}+\frac{\kappa_{2,m}}{2}|\leq2, \ 
			(1-h\,\frac{\kappa_{1,m}}{2})(1-h\,\frac{\kappa_{2,m}}{2})>0,$ $
			(1+h\,\frac{\kappa_{1,m}}{2})(1+h\,\frac{\kappa_{2,m}}{2})>0.
			$
			These inequalities imply $|h\,\frac{\kappa_{1,m}}{2}|<1$ and $|h\,\frac{\kappa_{2,m}}{2}|<1$. Therefore
			$
			\frac h2\max\{|\kappa_{1,m}|,|\kappa_{2,m}|\}<1
			\ \text{a.e. in }\omega,
			$
			and hence $m\in V^h$.}
		
		\,{\,{Since the limit belongs to $\mathbf M$ almost everywhere, the standard weak lower semicontinuity theorem applies to the weakly converging variables
				$
				(\nabla m_k,\nabla n_{m_k},
				{\rm a}_{m_k},
				{\rm a}_{m_k}{\rm H}_{m_k},
				{\rm a}_{m_k}{\rm K}_{m_k}).
				$
				Therefore, we obtain}
			$$
			\int_\omega W(x',m)\,{\rm d}x'
			\leq
			\liminf_{k\to\infty}\int_\omega W(x',m_k)\,{\rm d}x',
			$$
			whereas the continuous affine functional $\mathcal L$ is weakly continuous. Consequently,
			$\mathcal I(m)\leq\liminf_{k\to\infty}\mathcal I(m_k)=\inf_{V^h}\mathcal I$, and therefore $m$ is a minimizer.}
	\end{proof}
	
	\,{The preceding proof shows more than the weak closedness needed for the
		direct method. Namely, the geometric boundary of the admissible set is
		energetically inaccessible at finite energy. If an admissible sequence
		has uniformly bounded energy and converges in the weak topology used
		above, then its limit satisfies
		\[
		{\rm a}_m>0,\qquad
		{\rm a}_m{\rm A}_m^\pm>0,
		\qquad
		\frac h2
		\max\{|\kappa_{1,m}|,|\kappa_{2,m}|\}<1
		\quad\text{a.e. in }\omega.
		\]
		\,{Thus the strict curvature condition is preserved under
			weak limits of sequences with uniformly bounded energy.}
	}
	
	{\,
		\subsection{The convexity of energetic terms}
		\label{sec3-subsec2}
		
		In order to establish existence theorems for our proposed models, let us remark that the
		construction of our models was done so that Theorem \ref{general-existence} could be applied.
		In
		this subsection, we first rewrite the shell energies as quadratic functions of
		$(\nabla m,\nabla n_m)$. We then express all the contractions involved in these quadratic forms
		through the same symmetric representation of the Weingarten map and diagonalize them along the
		principal directions of the reference midsurface. This provides a necessary and sufficient
		pointwise convexity criterion. As a consequence, we have that the shell parts of Models I and II are convex
		for every thickness satisfying the geometric condition \eqref{ch5in}. For Model III, the Taylor-expanded quadratic volumetric contribution leads to an additional explicit and sharp condition on the thickness.
		
		We begin with the energetic term
		$W^{(1)}_{\mathrm{shell}}({\rm I}_{m},{\rm II}_{m},{\rm III}_{m})$. Thanks to
		\eqref{basic-contractions}, we can rewrite this term as follows:
		\begin{align}
			\label{convex1-e1}
			W^{(1)}_{\mathrm{shell}}\big({\rm I}_{m},{\rm II}_{m},{\rm III}_{m}\big)
			=\mathcal F(\nabla m,\nabla n_m),
		\end{align}
		where the function $\mathcal F:\mathbb R^{3\times2}\times\mathbb R^{3\times2}\to\mathbb R$
		is defined by
		\begin{align}
			\label{convex1-e2}
			\mathcal F(E,G):=&\frac{\mu}{2}\Bigg\{
			\Big(h-\frac{h^3}{12}{\rm K}_{y_0}+\frac{h^5}{80}{\rm K}_{y_0}^2\Big)
			\widetilde{\mathcal F}_0(E,E)
			-\frac{h^3}{3}{\rm H}_{y_0}\widetilde{\mathcal F}_0(E,G)\notag\\
			&+\Big(\frac{h^3}{12}-\frac{h^5}{80}{\rm K}_{y_0}\Big)
			\widetilde{\mathcal F}_0(G,G)
			-\frac{h^5}{40}{\rm H}_{y_0}{\rm K}_{y_0}\widetilde{\mathcal F}_1(E,E)\notag\\
			&+\Big(\frac{h^3}{6}-\frac{h^5}{40}{\rm K}_{y_0}\Big)
			\widetilde{\mathcal F}_1(E,G)
			+\Big(\frac{h^3}{12}+\frac{h^5}{80}(4{\rm H}_{y_0}^2-{\rm K}_{y_0})\Big)
			\widetilde{\mathcal F}_2(E,E)\notag\\
			&+\frac{h^5}{20}{\rm H}_{y_0}\widetilde{\mathcal F}_2(E,G)
			+\frac{h^5}{80}\widetilde{\mathcal F}_2(G,G)
			+\Big(h+\frac{h^3}{12}{\rm K}_{y_0}\Big)\Bigg\}
			-\frac{3\mu}{2}\Big(h+\frac{h^3}{12}{\rm K}_{y_0}\Big),
		\end{align}
		for all $E,G\in\mathbb R^{3\times2}$, with 
		\begin{align}
			\widetilde{\mathcal F}_0(E,G)
			&:=\big\langle {\rm I}_{y_0}^{-1/2}E^TG{\rm I}_{y_0}^{-1/2},\id_2\big\rangle
			=\big\langle E{\rm I}_{y_0}^{-1/2},G{\rm I}_{y_0}^{-1/2}\big\rangle,\notag
			\\
			\label{convex1-e4}
			\widetilde{\mathcal F}_1(E,G)
			&:=\big\langle E^TG,{\rm L}_{y_0}{\rm I}_{y_0}^{-1}
			+{\rm I}_{y_0}^{-1}{\rm L}_{y_0}^{T}\big\rangle,
			\\
			\widetilde{\mathcal F}_2(E,G)
			&:=\big\langle {\rm I}_{y_0}^{-1/2}{\rm L}_{y_0}^{T}E^TG
			{\rm L}_{y_0}{\rm I}_{y_0}^{-1/2},\id_2\big\rangle=\big\langle E{\rm L}_{y_0}{\rm I}_{y_0}^{-1/2},
			G{\rm L}_{y_0}{\rm I}_{y_0}^{-1/2}\big\rangle.\notag
		\end{align}

		Similarly,
		\begin{align}
			\label{convex2-e1}
			W^{(2)}_{\mathrm{shell}}\big({\rm I}_{m},{\rm II}_{m},{\rm III}_{m}\big)
			=\mathcal G(\nabla m,\nabla n_m),
		\end{align}
		where $\mathcal G:\mathbb R^{3\times2}\times\mathbb R^{3\times2}\to\mathbb R$ is given by
		\begin{align}
			\label{convex2-e2}
			\mathcal G(E,G):=&\frac{\mu}{2}\Bigg\{
			\Big(h-\frac{h^3}{12}{\rm K}_{y_0}\Big)\widetilde{\mathcal F}_0(E,E)
			-\frac{h^3}{3}{\rm H}_{y_0}\widetilde{\mathcal F}_0(E,G)
			+\frac{h^3}{12}\widetilde{\mathcal F}_0(G,G)\notag\\
			&+\frac{h^3}{6}\widetilde{\mathcal F}_1(E,G)
			+\frac{h^3}{12}\widetilde{\mathcal F}_2(E,E)
			+\Big(h+\frac{h^3}{12}{\rm K}_{y_0}\Big)\Bigg\}
			-\frac{3\mu}{2}\Big(h+\frac{h^3}{12}{\rm K}_{y_0}\Big).
		\end{align}
		
		\subsubsection{A symmetric representation of the Weingarten map}
		\label{sec3-subsec2-frame}
		
		Throughout this subsection $x'\in\bar\omega$ is fixed and, for the sake of brevity, we omit it
		from the notation. Since ${\rm I}_{y_0}\in{\rm Sym}(2)$ is positive definite, its square root
		${\rm I}_{y_0}^{1/2}$ and its inverse square root ${\rm I}_{y_0}^{-1/2}$ are well-defined
		symmetric matrices. Moreover, they depend continuously on $x'\in\bar\omega$. We define
		\begin{align}
			\label{shape-orthonormal}
			{\rm S}_{y_0}:={\rm I}_{y_0}^{-1/2}{\rm II}_{y_0}{\rm I}_{y_0}^{-1/2}
			={\rm I}_{y_0}^{1/2}{\rm L}_{y_0}{\rm I}_{y_0}^{-1/2}.
		\end{align}
		
		\begin{lemma}
			\label{lem-frame}
			The matrix ${\rm S}_{y_0}$ defined by \eqref{shape-orthonormal} is symmetric and similar to
			${\rm L}_{y_0}$. Hence, its eigenvalues are the principal curvatures
			$\kappa_{1,y_0},\kappa_{2,y_0}$ and
			$
			\tr {\rm S}_{y_0}=2{\rm H}_{y_0}, $ $
			\det {\rm S}_{y_0}={\rm K}_{y_0}, $ $
			|{\rm S}_{y_0}|_{\rm op}=\max_{i=1,2}|\kappa_{i,y_0}|.
			$
			There exists an orthonormal basis $\{w_1,w_2\}$ of $\mathbb R^2$ such that
			\begin{align}
				\label{spectral-S}
				{\rm S}_{y_0}=\kappa_{1,y_0}w_1\otimes w_1+\kappa_{2,y_0}w_2\otimes w_2.
			\end{align}
			The vectors $w_i$ are the coordinate representatives, in the orthonormalized tangent frame,
			of the principal directions of the reference midsurface. Finally, under a change of
			parametrization, ${\rm S}_{y_0}$ changes only by an orthogonal conjugation. In particular,
			$|{\rm S}_{y_0}|_{\rm op}$ and $|{\rm S}_{y_0}|$ are independent of the chosen chart.
		\end{lemma}
		
		\begin{proof}
			The symmetry follows immediately from the first expression in \eqref{shape-orthonormal}. The
			second expression gives
			${\rm S}_{y_0}={\rm I}_{y_0}^{1/2}{\rm L}_{y_0}{\rm I}_{y_0}^{-1/2}$, so that ${\rm S}_{y_0}$ is similar
			to ${\rm L}_{y_0}$. Therefore, both matrices have the same characteristic polynomial,
			$\kappa^2-2{\rm H}_{y_0}\kappa+{\rm K}_{y_0}$, and hence the same eigenvalues. Since
			${\rm S}_{y_0}$ is symmetric, \eqref{spectral-S} follows from the spectral theorem and its operator
			norm is the largest modulus of its eigenvalues. Under an orientation-preserving change of parameters with Jacobian
			$P\in{\rm GL}^+(2)$, one has
			${\rm I}_{y_0}\mapsto P^T{\rm I}_{y_0}P$ and
			${\rm II}_{y_0}\mapsto P^T{\rm II}_{y_0}P$. The corresponding symmetric matrix obtained from
			\eqref{shape-orthonormal} is similar to the transformed Weingarten map, and therefore has the
			same eigenvalues as ${\rm S}_{y_0}$. Two symmetric $2\times2$ matrices with the same eigenvalues are
			orthogonally similar, which proves the last assertion.
		\end{proof}
		
		The next lemma rewrites the three bilinear forms given by \eqref{convex1-e4} with
		respect to the same Euclidean scalar product.
		
		\begin{lemma}
			\label{lem-dictionary}
			For $E, G\in\mathbb R^{3\times2}$, {we denote $\bar E:=E{\rm I}_{y_0}^{-1/2}$ and $\bar G:=G{\rm I}_{y_0}^{-1/2}$.} Then, for all
			$E,G\in\mathbb R^{3\times2}$,
			\begin{align}
				\label{dictionary}
				\widetilde{\mathcal F}_0(E,G)&=\langle\bar E,\bar G\rangle,\qquad
				\widetilde{\mathcal F}_1(E,G)=2\langle\bar G,\bar E{\rm S}_{y_0}\rangle,\qquad
				\widetilde{\mathcal F}_2(E,G)=\langle\bar E{\rm S}_{y_0},\bar G{\rm S}_{y_0}\rangle.
			\end{align}
		\end{lemma}
		
		\begin{proof}
			Since
			$
			{\rm L}_{y_0}{\rm I}_{y_0}^{-1}
			={\rm I}_{y_0}^{-1}{\rm II}_{y_0}{\rm I}_{y_0}^{-1}
			={\rm I}_{y_0}^{-1}{\rm L}_{y_0}^T\in{\rm Sym}(2),
			$
			we obtain from \eqref{convex1-e4}
			\begin{align}
				\widetilde{\mathcal F}_1(E,G)
				&=2\tr\big(G^TE{\rm L}_{y_0}{\rm I}_{y_0}^{-1}\big)=2\tr\big((G{\rm I}_{y_0}^{-1/2})^T(E{\rm I}_{y_0}^{-1/2}){\rm S}_{y_0}\big)
				=2\langle\bar G,\bar E{\rm S}_{y_0}\rangle.
			\end{align}
			Moreover, by \eqref{convex1-e4},
			$\widetilde{\mathcal F}_2(E,G)=\langle E{\rm L}_{y_0}{\rm I}_{y_0}^{-1/2},
			G{\rm L}_{y_0}{\rm I}_{y_0}^{-1/2}\rangle$, while
			$E{\rm L}_{y_0}{\rm I}_{y_0}^{-1/2}=\bar E{\rm S}_{y_0}$, and the third identity follows.
		\end{proof}

		Since ${\rm S}_{y_0}\in{\rm Sym}(2)$, the three forms in \eqref{dictionary} are symmetric, i.e.
		$\widetilde{\mathcal F}_i(E,G)=\widetilde{\mathcal F}_i(G,E)$, $i=0,1,2$. Moreover,
		\begin{align}
			|\widetilde{\mathcal F}_1(E,G)|
			\le 2|{\rm S}_{y_0}|_{\rm op}|\bar E|\,|\bar G|
			=2\max_i|\kappa_{i,y_0}|\,|\bar E|\,|\bar G|.
		\end{align}
		The constant $2\max_i|\kappa_{i,y_0}|$ is optimal. Indeed, equality is obtained by taking
		$\bar E$ of rank one with right direction corresponding to a principal direction associated
		with $|S_{y_0}|_{\rm op}$ and choosing $\bar G$ proportional to $\bar E S_{y_0}$. We will not
		use this estimate below, since the principal-direction decomposition avoids the loss produced
		by a Cauchy--Schwarz estimate.
		
		\begin{remark}\label{rem-nonunique-F}
			Since the Weingarten map satisfies the Cayley--Hamilton identity
			$
			{\rm L}_{y_0}^2-2{\rm H}_{y_0}{\rm L}_{y_0}
			+{\rm K}_{y_0}\id_2=0,
			$
			the three linear operators introduced above are not linearly independent.
			More precisely,
			$
			\mathcal F_2(Q)
			=
			{\rm H}_{y_0}\mathcal F_1(Q)
			-
			{\rm K}_{y_0}\mathcal F_0(Q),
			$
			and, equivalently,
			$
			\widetilde{\mathcal F}_2(E,G)
			=
			{\rm H}_{y_0}\widetilde{\mathcal F}_1(E,G)
			-
			{\rm K}_{y_0}\widetilde{\mathcal F}_0(E,G).
			$
			Consequently, the coefficients in a representation of a quadratic form
			in terms of
			\(\widetilde{\mathcal F}_0,\widetilde{\mathcal F}_1,
			\widetilde{\mathcal F}_2\)
			are not unique. This does not affect the principal-direction
			coefficients in Lemma~\ref{lem-decoupling}. Indeed, replacing a triple
			\((a,b,c)\) by
			$
			(a+\tau{\rm K}_{y_0},\,b-\tau{\rm H}_{y_0},\,c+\tau)
			$
			leaves
			$
			a+2b\kappa_i+c\kappa_i^2
			$
			unchanged, since
			$
			\kappa_i^2-2{\rm H}_{y_0}\kappa_i+{\rm K}_{y_0}=0.
			$
			Thus the quantities \(A_i,B_i,C_i\), and hence the convexity criterion,
			are intrinsic although the intermediate coefficients
			\(a_0,\ldots,c_1\) are {dependent on the representation}.
		\end{remark}
		
		\subsubsection{Decoupling along the principal directions}
		\label{sec3-subsec2-decoupling}
		
		All the quadratic forms occurring in $W^{(1)}_{\mathrm{shell}}$ and
		$W^{(2)}_{\mathrm{shell}}$ are of the following type. Let
		$a_0,a_{01},a_1,b_0,b_{01},b_1,c_0,c_{01},c_1\in\mathbb R$ and let
		\begin{align}
			\label{general-Q}
			\mathcal W(E,G):=&a_0\widetilde{\mathcal F}_0(E,E)
			+a_{01}\widetilde{\mathcal F}_0(E,G)+a_1\widetilde{\mathcal F}_0(G,G)\notag\\
			&+b_0\widetilde{\mathcal F}_1(E,E)
			+b_{01}\widetilde{\mathcal F}_1(E,G)+b_1\widetilde{\mathcal F}_1(G,G)\\
			&+c_0\widetilde{\mathcal F}_2(E,E)
			+c_{01}\widetilde{\mathcal F}_2(E,G)+c_1\widetilde{\mathcal F}_2(G,G).\notag
		\end{align}
		
		\begin{lemma}
			\label{lem-decoupling}
			Let $\{w_1,w_2\}$ be as in \eqref{spectral-S} and  let us define
			$
			\bar e_i:=E\,{\rm I}_{y_0}^{-1/2}w_i,$ $
			\bar g_i:=G\,{\rm I}_{y_0}^{-1/2}w_i,$ $ i=1,2.
			$
			Then $(E,G)\mapsto(\bar e_1,\bar g_1,\bar e_2,\bar g_2)$ is a linear bijection from
			$\mathbb R^{3\times2}\times\mathbb R^{3\times2}$ onto $(\mathbb R^3)^4$, and
			\begin{align}
				\label{decoupling}
				\mathcal W(E,G)=\sum_{i=1}^2\Big[A_i|\bar e_i|^2
				+B_i\langle\bar e_i,\bar g_i\rangle+C_i|\bar g_i|^2\Big],
			\end{align}
			where
			\begin{align}
				\label{ABC-general}
				A_i:=a_0+2b_0\kappa_{i,y_0}+c_0\kappa_{i,y_0}^2,\qquad
				B_i:=a_{01}+2b_{01}\kappa_{i,y_0}+c_{01}\kappa_{i,y_0}^2,\qquad
				C_i:=a_1+2b_1\kappa_{i,y_0}+c_1\kappa_{i,y_0}^2.
			\end{align}
		\end{lemma}
		
		\begin{proof}
			The first assertion follows from the fact that $E\mapsto E\,{\rm I}_{y_0}^{-1/2}$ is a linear
			bijection and $\{w_1,w_2\}$ is an orthonormal basis of $\mathbb R^2$. Moreover,
			$\bar E\,{\rm S}_{y_0}w_i=\kappa_{i,y_0}\bar e_i$, {$\bar G\,{\rm S}_{y_0}w_i=\kappa_{i,y_0}\bar g_i$, where $\bar E$ and $\bar G$ are given as in Lemma \ref{lem-dictionary}}, and therefore
			\begin{align}
				\langle\bar E,\bar G\rangle&=\sum_{i=1}^2\langle\bar e_i,\bar g_i\rangle,\qquad
				\langle\bar G,\bar E\,{\rm S}_{y_0}\rangle
				=\sum_{i=1}^2\kappa_{i,y_0}\langle\bar e_i,\bar g_i\rangle,\qquad 
				\langle\bar E\,{\rm S}_{y_0},\bar G\,{\rm S}_{y_0}\rangle
				=\sum_{i=1}^2\kappa_{i,y_0}^2\langle\bar e_i,\bar g_i\rangle.
			\end{align}
			Inserting these identities into \eqref{general-Q} yields \eqref{decoupling}.
		\end{proof}
		
		\begin{proposition}
			\label{cor-criterion}
			With the notation of Lemma \ref{lem-decoupling},
			\begin{align}
				\label{criterion}
				\mathcal W(E,G)\ge0\quad\forall(E,G)
				\qquad\Longleftrightarrow\qquad
				A_i\ge0,\quad C_i\ge0,\quad B_i^2\le4\,A_i\,C_i,\qquad i=1,2.
			\end{align}
			If the inequalities are strict, then
			\begin{align}
				\label{coercivity-Q}
				\mathcal W(E,G)\ge\Lambda\Big(|E\,{\rm I}_{y_0}^{-1/2}|^2
				+|G\,{\rm I}_{y_0}^{-1/2}|^2\Big),\qquad
				\Lambda:=\min_{i=1,2}\frac{4\,A_i\,C_i-B_i^2}{4(A_i+C_i)}>0.
			\end{align}
		\end{proposition}
		
		\begin{proof}
			By \eqref{decoupling}, the vectors $\bar e_i,\bar g_i$ can be chosen independently. Thus,
			$\mathcal W\ge0$ is equivalent to the nonnegativity of the quadratic forms
			$A_i|e_i|^2+B_i\langle e_i,g_i\rangle+C_i|g_i|^2$ for every $i$. This is equivalent to
			$A_i\ge0$, $C_i\ge0$ and $B_i^2\le4\,A_i\,C_i$. If the inequalities are strict, the smallest
			eigenvalue of the matrix
			$\left(\begin{smallmatrix}A_i&B_i/2\\B_i/2&C_i\end{smallmatrix}\right)$ is larger than 
			$
			\frac{4\,A_i\,C_i-B_i^2}{4(A_i+C_i)},
			$
			and summing over $i$ yields \eqref{coercivity-Q}.
		\end{proof}

		The previous argument is pointwise in $x'\in\bar\omega$ and does not require a measurable choice
		of the principal directions. In particular, umbilical points, where the principal frame is not
		unique, cause no difficulty.

		\subsubsection{The shell energies of Models I and II}
		\label{sec3-subsec2-shell}
		
		We now apply \eqref{ABC-general} to the functions $\mathcal F$ and $\mathcal G$. By
		\eqref{convex1-e2} and \eqref{convex2-e2}, we can write
		\begin{align}
			\mathcal F(E,G)=\frac{\mu}{2}\mathcal W^{(1)}(E,G)-\mu\Big(h+\frac{h^3}{12}{\rm K}_{y_0}\Big),
			\qquad
			\mathcal G(E,G)=\frac{\mu}{2}\mathcal W^{(2)}(E,G)-\mu\Big(h+\frac{h^3}{12}{\rm K}_{y_0}\Big),
		\end{align}
		where  $\mathcal W^{(1)}$ is the quadratic form defined by \eqref{general-Q} by considering the coefficients 
		\begin{align}
			\label{coef-model1}
			a_0&=h-\frac{h^3}{12}{\rm K}_{y_0}+\frac{h^5}{80}{\rm K}_{y_0}^2,
			& a_{01}&=-\frac{h^3}{3}{\rm H}_{y_0},
			& a_1&=\frac{h^3}{12}-\frac{h^5}{80}{\rm K}_{y_0},\notag\\
			b_0&=-\frac{h^5}{40}{\rm H}_{y_0}{\rm K}_{y_0},
			& b_{01}&=\frac{h^3}{6}-\frac{h^5}{40}{\rm K}_{y_0},
			& b_1&=0,\\
			c_0&=\frac{h^3}{12}+\frac{h^5}{80}(4{\rm H}_{y_0}^2-{\rm K}_{y_0}),
			& c_{01}&=\frac{h^5}{20}{\rm H}_{y_0},
			& c_1&=\frac{h^5}{80},\notag
		\end{align}
		while $\mathcal W^{(2)}$ is the quadratic form defined by \eqref{general-Q} by considering the coefficients 		
		\begin{align}
			\label{coef-model2}
			a_0=h-\frac{h^3}{12}{\rm K}_{y_0},\qquad
			a_{01}=-\frac{h^3}{3}{\rm H}_{y_0},\qquad
			a_1=\frac{h^3}{12},\qquad
			b_{01}=\frac{h^3}{6},\qquad
			c_0=\frac{h^3}{12},
		\end{align}
		\,{with} all remaining coefficients taken to be zero.
		
		The coefficients above correspond to the particular representations
		chosen in \eqref{convex1-e2} and \eqref{convex2-e2}; by
		Remark~\ref{rem-nonunique-F}, the principal-direction coefficients
		\(A_i,B_i,C_i\) are independent of these choices.
		
		\begin{lemma}
			\label{lem-ABC-explicit}
			For Model I, the coefficients \eqref{ABC-general} are
			\begin{align}
				\label{ABC-model1}
				A_i&=h+\kappa_{i,y_0}(\kappa_{i,y_0}-\kappa_{j,y_0})\theta_i,
				& B_i&=2(\kappa_{i,y_0}-\kappa_{j,y_0})\theta_i,
				& C_i&=\frac{h^3}{12}+\frac{h^5}{80}\kappa_{i,y_0}(\kappa_{i,y_0}-\kappa_{j,y_0}),
			\end{align}
			where $i,j=1,2$ $i\neq j$, 
			$
			\theta_i:=\frac{h^3}{12}+\frac{h^5}{80}\kappa_{i,y_0}^2>0,
			$ whereas for Model II they are
			\begin{align}
				\label{ABC-model2}
				A_i=h+\frac{h^3}{12}\kappa_{i,y_0}(\kappa_{i,y_0}-\kappa_{j,y_0}),\qquad
				B_i=\frac{h^3}{6}(\kappa_{i,y_0}-\kappa_{j,y_0}),\qquad
				C_i=\frac{h^3}{12}.
			\end{align}
		\end{lemma}
		
		\begin{proof}
			The formulas follow by inserting $2\,{\rm H}_{y_0}=\kappa_{1,y_0}+\kappa_{2,y_0}$ and
			${\rm K}_{y_0}=\kappa_{1,y_0}\kappa_{2,y_0}$ into \eqref{ABC-general}. For example, for Model I,
			\begin{align}
				A_i
				&=h+\frac{h^3}{12}\kappa_{i,y_0}(\kappa_{i,y_0}-\kappa_{j,y_0})
				+\frac{h^5}{80}\kappa_{i,y_0}^3(\kappa_{i,y_0}-\kappa_{j,y_0})
				=h+\kappa_{i,y_0}(\kappa_{i,y_0}-\kappa_{j,y_0})\theta_i,\notag\\
				B_i
				&=\frac{h^3}{6}(\kappa_{i,y_0}-\kappa_{j,y_0})
				+\frac{h^5}{40}\kappa_{i,y_0}^2(\kappa_{i,y_0}-\kappa_{j,y_0})
				=2(\kappa_{i,y_0}-\kappa_{j,y_0})\theta_i,
			\end{align}
			and the formula for $C_i$ follows in the same way. The computation for Model II is immediate.
		\end{proof}
		
		\begin{proposition}
			\label{prop-no-restriction}
			Let $h>0$ satisfy \eqref{ch5in}. Then, for every $x'\in\bar\omega$ and $i=1,2$, the coefficients of Model I satisfy
			\begin{align}
				\label{bounds-model1}
				A_i\ge\frac{13}{15}h,\qquad
				C_i\ge\frac{17}{240}h^3,\qquad
				4A_iC_i-B_i^2\ge\frac{130-32\sqrt2}{405}h^4>0,
			\end{align}
			and those of Model II satisfy
			\begin{align}
				\label{bounds-model2}
				A_i\ge\frac{11}{12}h,\qquad
				C_i=\frac{h^3}{12},\qquad
				4A_iC_i-B_i^2\ge\frac{h^4}{9}>0.
			\end{align}
			Consequently, $\mathcal F$ and $\mathcal G$ are convex on
			$\mathbb R^{3\times2}\times\mathbb R^{3\times2}$ for every $h$ satisfying \eqref{ch5in}.
		\end{proposition}
		
		\begin{proof} In order to simplify the expression, let us introduce the quantities 
			$
			u:=h\kappa_{i,y_0},$ $v:=h\kappa_{j,y_0}$, where  $ i,j=1,2$, $j\neq i, 
			$ and $t:=\frac1{12}+\frac{u^2}{80}$.
			With this notation, we have $|u|<2$ and $|v|<2$, and we study the required inequalities on the closed square $[-2,2]^2$.
			For Model I, we also use the rescaled quantities given by 
			$
			A_i=h\mathfrak a,  B_i=h^2\mathfrak b,  C_i=h^3\mathfrak c,
			$
			where
			$
			\mathfrak  a=1+u(u-v)t,$ $
			\mathfrak b=2(u-v)t,$ $
			\mathfrak c=\frac1{12}+\frac{u(u-v)}{80}.
			$
			Since $u(u-v)=u^2-uv\ge u^2-2|u|\ge-1$, we have
			$\mathfrak c\ge\frac{17}{240}$. If $u(u-v)\ge0$, then $\mathfrak a\ge1$; otherwise,
			$u(u-v)\in[-1,0)$ and $t\le\frac2{15}$, hence $\mathfrak a\ge\frac{13}{15}$.
			
			It remains to estimate the determinant. {Note that
				$\frac{u^2}{80}=t-\frac1{12}$, a direct computation gives}
			$
			4\mathfrak a\mathfrak c-\mathfrak b^2=\frac13 P(u,v),
			$
			where
			\begin{align}
				\label{Ppoly}
				P(u,v)=1+\frac{3u^2}{20}-\frac{uv}{15}-\frac{v^2}{12}
				+\frac{u^3v}{80}-\frac{u^2v^2}{80}.
			\end{align}
			For fixed $u$, the function $v\mapsto P(u,v)$ is strictly concave, since
			$
			\partial_{vv}^2P(u,v)
			=
			-\frac16-\frac{u^2}{40}<0.
			$
			Hence its minimum on \(v\in[-2,2]\) is attained at one of the endpoints.
			Since \(P(-u,-v)=P(u,v)\), it is enough to consider
			\begin{align}
				\phi(u):=P(u,2)
				=
				\frac23-\frac{2u}{15}+\frac{u^2}{10}+\frac{u^3}{40}.
			\end{align}
			The critical points of $\phi$ satisfy
			$
			u=\frac{4(-1\pm\sqrt2)}3.
			$
			Only
			$
			u_+=\frac{4(\sqrt2-1)}3
			$
			belongs to \([-2,2]\). Since
			$
			\phi(-2)=\frac{17}{15},
			\,
			\phi(2)=1,
			$
			we obtain
			$
			\min_{[-2,2]^2}P
			=
			\phi(u_+)
			=
			\frac{130-32\sqrt2}{135}>0.
			$
			This proves \eqref{bounds-model1}.
			
			For Model II, the rescaled quantities are
			$
			\mathfrak a=1+\frac{u(u-v)}{12},$ $
			\mathfrak b=\frac{u-v}{6},$ $
			\mathfrak c=\frac1{12},
			$
			and therefore
			$
			4\mathfrak a\mathfrak c-\mathfrak b^2
			=\frac13+\frac{v(u-v)}{36}\ge\frac19,
			$
			since $v(u-v)=uv-v^2\ge-8$ on $[-2,2]^2$. Moreover
			$\mathfrak a\ge1-\frac1{12}=\frac{11}{12}$. Hence \eqref{bounds-model2} follows. The
			conclusion is now a direct consequence of Proposition  \ref{cor-criterion}.
		\end{proof}

		The preceding result shows that, for Models I and II, convexity of the shell energy produces no
		additional restriction on the thickness beyond \eqref{ch5in}. 
		
		\subsubsection{The volumetric and curvature energies}
		\label{sec3-subsec2-curv}
		
		We first note that, by \eqref{ch5in}, for every
		$s\in[-1/2,1/2]$,
		\begin{align}
			\label{Aquad-positive}
			\mathcal A_{y_0}(s)
			&=
			1-2hs\,{\rm H}_{y_0}+h^2s^2\,{\rm K}_{y_0}
			=
			\bigl(1-hs\,\kappa_{1,y_0}\bigr)
			\bigl(1-hs\,\kappa_{2,y_0}\bigr)>0
			\qquad\text{on }\bar\omega .
		\end{align}
		In particular,
		\begin{align}
			\label{Apm-positive}
			{\rm A}_{y_0}^{\pm}
			=
			\mathcal A_{y_0}(\pm1/2)>0
			\qquad\text{on }\bar\omega .
		\end{align}

		\begin{lemma}
			\label{lem-Mh-geometric}
			Let $\mathcal M_h$ be defined by \begin{align}
				\label{Mcurv}
				\mathcal M_h
				:=
				\Big\{
				(a,b,c)\in\mathbb R^3
				\ \big|\
				a-\frac h2|b|>0,\quad
				a-h|b|+\frac{h^2}{4}c>0
				\Big\}.
			\end{align}
			Then, for every $(a,b,c)\in\mathcal M_h$,
			\begin{align}
				\label{all-nodes-positive}
				a-2hs\,b+h^2s^2c>0
				\qquad
				\text{for all }s\in[-1/2,1/2].
			\end{align}
			Moreover
			\[
			({\rm a}_m,{\rm a}_m{\rm H}_m,{\rm a}_m{\rm K}_m)
			\in\mathcal M_h
			\quad\Longleftrightarrow\quad
			\frac h2\max\{|\kappa_{1,m}|,|\kappa_{2,m}|\}<1.
			\]
		\end{lemma}
		
		\begin{proof}
			Let $(a,b,c)\in\mathcal M_h$. The first defining inequality implies
			$a>0$, and we may define 
			$
			B:=\frac{hb}{a}
			$ and $
			C:=\frac{h^2c}{a},
			$
			which satisfy 
			$
			|B|<2,
			$ $
			1-|B|+\frac C4>0.
			$
			For $t=2s\in[-1,1]$ we have
			$
			\frac{1}{a}\bigl(a-2hsb+h^2s^2c\bigr)
			=
			q(t):=1-Bt+\frac C4t^2.
			$
			
			If $C\le0$, the function $q$ is concave, and hence its minimum on
			$[-1,1]$ is attained at an endpoint. Since
			$
			q(\pm1)=1\mp B+\frac C4
			\ge
			1-|B|+\frac C4>0,
			$
			we obtain $q>0$ on $[-1,1]$.
			
			Assume now that $C>0$. If the vertex of $q$ lies outside $[-1,1]$,
			the minimum is again attained at an endpoint. If the vertex lies in
			$[-1,1]$, then $C\ge2|B|$ and
			$
			q_{\min}
			=
			1-\frac{B^2}{C}.
			$
			Since $|B|<2$, we have $B^2<2|B|\le C$, and therefore $q_{\min}>0$.
			Thus \eqref{all-nodes-positive} holds for every
			$(a,b,c)\in\mathcal M_h$.
			
			For a geometric triple
			$(a,b,c)=({\rm a}_m,{\rm a}_m{\rm H}_m,{\rm a}_m{\rm K}_m)$,
			consider 
			$
			\xi=\frac h2\kappa_{1,m}
			\textrm{ and }
			\eta=\frac h2\kappa_{2,m}.
			$
			Since
			$
			{\rm H}_m=\frac{\kappa_{1,m}+\kappa_{2,m}}2,
			$ $
			{\rm K}_m=\kappa_{1,m}\kappa_{2,m},
			$
			we have
			$
			a-\frac h2|b|
			=
			a\left(1-\frac12|\xi+\eta|\right)
			$
			\textrm{ and }
			$
			a-h|b|+\frac{h^2}{4}c
			=
			a\bigl(1-|\xi+\eta|+\xi\eta\bigr).
			$
			If $|\xi|,|\eta|<1$, both quantities are positive. Conversely,
			$(a,b,c)\in\mathcal M_h$ implies $|\xi+\eta|<2$ and
			$1-|\xi+\eta|+\xi\eta>0$. If $\xi+\eta\ge0$, the latter inequality
			reads $(1-\xi)(1-\eta)>0$; together with $\xi+\eta<2$, this gives
			$|\xi|,|\eta|<1$. The case $\xi+\eta<0$ follows analogously from
			$(1+\xi)(1+\eta)>0$.
		\end{proof}
		
		For $s\in[-1/2,1/2]$, for later use, let us define 
		\begin{align}
			\label{ell-s}
			\ell_s(a,b,c):=a-2hs\,b+h^2s^2c.
		\end{align}
		Thus, for
		$(a,b,c)=({\rm a}_m,{\rm a}_m{\rm H}_m,{\rm a}_m{\rm K}_m)$,
		$
		\ell_s(a,b,c)={\rm a}_m\mathcal A_m(s).
		$

		Let $\mathcal Q$ be a symmetric positive-weight quadrature rule of the
		form \eqref{general-quadrature}, whose nodes belong to
		$[-1/2,1/2]$. We write
		\begin{align}
			\label{WvolQ-curv}
			W_{\rm vol}^{\mathcal Q}(x',a,b,c)
			:=
			h\sum_{j\in J}
			\omega_j\mathcal A_{y_0}(\zeta_j)
			\overline{\mathcal H}_{\rm CG}
			\left(
			\frac{\ell_{\zeta_j}(a,b,c)}
			{{\rm a}_{y_0}\mathcal A_{y_0}(\zeta_j)}
			\right)
		\end{align}
		and
		\begin{align}
			\label{WlogQ-curv}
			W_{\log}^{\mathcal Q}(x',a,b,c)
			:=
			\frac{\lambda+2\mu}{2}\,
			h\sum_{j\in J}
			\omega_j\mathcal A_{y_0}(\zeta_j)
			\overline{\mathcal G}_{\log }
			\left(
			\frac{\ell_{\zeta_j}(a,b,c)}
			{{\rm a}_{y_0}\mathcal A_{y_0}(\zeta_j)}
			\right).
		\end{align}
		In particular, $W_{\rm vol}^{B}$ and $W_{\rm vol}^{S}$ are obtained
		from \eqref{WvolQ-curv} by choosing the Boole and Simpson nodes and
		weights, respectively, whereas $W_{\log}^{B}$ is obtained from
		\eqref{WlogQ-curv} by choosing the Boole nodes and weights.
		
		\begin{lemma} \label{convex3-l3} Assume \eqref{ch5in}. For every symmetric positive-weight quadrature rule $\mathcal Q$ with nodes in $[-1/2,1/2]$, the functions $W_{\rm vol}^{\mathcal Q}(x',\cdot)$ \text{and}  $W_{\log}^{\mathcal Q}(x',\cdot)$ are proper, convex and lower semicontinuous on $\mathbb R^3$ for every $x'\in\overline{\omega}$. Moreover, their effective domains contain $\mathcal M_h$. 
		\end{lemma} 
		\begin{proof} By \eqref{Aquad-positive}, ${\rm a}_{y_0}\mathcal A_{y_0}(\zeta_j)>0$  for every quadrature node $\zeta_j$. 
			Hence, for every fixed $x'\in\overline{\omega}$ and every $j\in J$, the map $(a,b,c)\mapsto \frac{\ell_{\zeta_j}(a,b,c)} {{\rm a}_{y_0}\mathcal A_{y_0}(\zeta_j)} $ is affine. By construction, $\overline{\mathcal H}_{\rm CG}$ and $\overline{\mathcal G}_{\log}$ are proper, convex and lower semicontinuous on $\mathbb R$. Therefore, their compositions with the above affine maps are convex and lower semicontinuous on $\mathbb R^3$. Since all quadrature weights $\omega_j$ and all reference factors $\mathcal A_{y_0}(\zeta_j)$ are strictly positive, it follows that $W_{\rm vol}^{\mathcal Q}(x',\cdot)$ and $W_{\log}^{\mathcal Q}(x',\cdot)$ are convex and lower semicontinuous on $\mathbb R^3$. By Lemma \ref{lem-Mh-geometric}, for every $(a,b,c)\in\mathcal M_h$ and every $\zeta_j\in[-1/2,1/2]$, $\ell_{\zeta_j}(a,b,c)>0.$ Hence all the arguments entering $W_{\rm vol}^{\mathcal Q}$ and $W_{\log}^{\mathcal Q}$ are strictly positive on $\mathcal M_h$, and therefore both energies are finite there. Consequently, both functions are proper and their effective domains contain $\mathcal M_h$. 
		\end{proof}
		
		Thus Lemma \ref{convex3-l3} applies directly to Model~I with the Boole
		nodes and weights and to Model~II with the Simpson nodes and weights.
		It also yields the properness, convexity and lower semicontinuity of
		the logarithmic Boole contribution in Model~III.
		
		We now consider the Taylor-expanded quadratic volumetric contribution
		which occurs only in Model~III. We retain the notation from the
		derivation and write
		\begin{align}
			\label{convex4-e1}
			W_{\rm vol}^{T}
			({\rm a}_m,
			{\rm a}_m\Delta{\rm H}_m,
			{\rm a}_m\Delta{\rm K}_m)
			=&\,
			\Xi\left(
			\frac{{\rm a}_m}{{\rm a}_{y_0}},
			\frac{{\rm a}_m}{{\rm a}_{y_0}}\Delta{\rm H}_m,
			\frac{{\rm a}_m}{{\rm a}_{y_0}}\Delta{\rm K}_m
			\right)
			-
			\frac{\lambda}{4}
			\left(
			h+\frac{h^3}{12}{\rm K}_{y_0}
			\right),
		\end{align}
		where $\Xi:\mathbb R^3\to\mathbb R$ is the quadratic function
		\begin{align}
			\label{convex4-e2}
			\Xi(r,X,Y)
			:=
			\frac{\lambda}{4}\Bigg[
			&h r^2
			+
			\frac{h^3}{12}
			\Big(
			{\rm K}_{y_0}r^2+4X^2+2rY
			\Big)
			+
			\frac{h^5}{80}
			\Big(
			(16{\rm H}_{y_0}^2-4{\rm K}_{y_0})X^2
			-8{\rm H}_{y_0}XY+Y^2
			\Big)
			\Bigg].
		\end{align}
		
		\begin{lemma}
			\label{convex4-l4}
			Assume \eqref{ch5in}. Then $\Xi$ is convex if and only if
			\begin{align}\label{det-condition} 
				240h^2{\rm H}_{y_0}^2
				+
				9h^4{\rm K}_{y_0}^2
				-
				12h^2{\rm K}_{y_0}
				\le320.
			\end{align}
			Moreover, $\Xi$ is strictly convex whenever the inequality in
			\eqref{det-condition} is strict. The condition is satisfied for every
			$0<h<h_0^{(3)}$, where
			\begin{align}
				\label{h0-model3}
				h_0^{(3)}
				:=
				\inf_{x'\in\bar\omega}\sqrt{\tau_*(x')},
			\end{align}
			with
			\begin{align}
				\label{tau-star}
				\tau_*
				:=
				\begin{cases}
					\dfrac{
						2\big(
						\sqrt{\Sigma_{y_0}^2+80{\rm K}_{y_0}^2}
						-\Sigma_{y_0}
						\big)
					}{
						3{\rm K}_{y_0}^2
					},
					& {\rm K}_{y_0}\neq0,
					\\[1.2em]
					\dfrac{4}{3{\rm H}_{y_0}^2},
					& {\rm K}_{y_0}=0,\ {\rm H}_{y_0}\neq0,
					\\[0.8em]
					+\infty,
					& {\rm H}_{y_0}={\rm K}_{y_0}=0,
				\end{cases}
				\qquad
				\Sigma_{y_0}
				:=
				20{\rm H}_{y_0}^2-{\rm K}_{y_0}\ge0.
			\end{align}
			In particular, $h_0^{(3)}>0$.
		\end{lemma}
		
		\begin{proof} 
			
			A direct computation from \eqref{convex4-e2} gives
			\[
			\frac{4}{\lambda}\nabla^2\Xi
			=
			\begin{pmatrix}
				h(2+\varrho/6)
				&
				0
				&
				h^3/6
				\\[4pt]
				0
				&
				h^3(2/3+2\vartheta^2/5-\varrho/10)
				&
				-h^4\vartheta/10
				\\[4pt]
				h^3/6
				&
				-h^4\vartheta/10
				&
				h^5/40
			\end{pmatrix},
			\]
			where 
			$
			\vartheta=h{\rm H}_{y_0}
			$ and $
			\varrho=h^2{\rm K}_{y_0}.
			$
			Since \eqref{ch5in} implies $|\varrho|<4$, all diagonal entries are
			strictly positive, in particular,
			\[
			h^3
			\left(
			\frac23+\frac{2\vartheta^2}{5}-\frac{\varrho}{10}
			\right)
			\ge
			\frac{4h^3}{15}>0.
			\]
			The principal minors of order two satisfy
			\begin{align}
				\label{minors2-automatic}
				h^6
				\left(
				\frac1{45}+\frac{\varrho}{240}
				\right)
				&\ge
				\frac{h^6}{180},
				\qquad 
				h^8
				\left(
				\frac1{60}-\frac{\varrho}{400}
				\right)
				\ge
				\frac{h^8}{150},
			\end{align}
			while the remaining one is the product of two positive diagonal
			entries. Moreover, we have 
			\begin{align}
				\label{det-hessian}
				\left(\frac4\lambda\right)^3
				\det\nabla^2\Xi
				=
				\frac{h^9}{21600}
				\left(
				320+12\varrho-240\vartheta^2-9\varrho^2
				\right).
			\end{align}
			Thus all principal minors of orders one and two are strictly positive.
			Hence positive semidefiniteness is equivalent to
			$\det\nabla^2\Xi\ge0$, while strict positivity of the determinant
			yields positive definiteness.
			
			By considering  $\tau=h^2$, the condition \eqref{det-condition} is equivalent to
			\[
			q(\tau)
			:=
			9{\rm K}_{y_0}^2\tau^2
			+
			12(20{\rm H}_{y_0}^2-{\rm K}_{y_0})\tau
			-
			320
			\le0.
			\]
			Since
			$
			20{\rm H}_{y_0}^2-{\rm K}_{y_0}
			=
			5\kappa_{1,y_0}^2
			+
			9\kappa_{1,y_0}\kappa_{2,y_0}
			+
			5\kappa_{2,y_0}^2
			\ge 0,
			$
			the function $q$ is nondecreasing on $[0,+\infty)$.
			If
			$({\rm H}_{y_0},{\rm K}_{y_0})\neq(0,0)$,
			it has a unique positive root and $q(\tau)\le0$ precisely for
			$0\le\tau\le\tau_*$. Solving the quadratic equation gives
			\eqref{tau-star}. In the flat case $q\equiv-320$, and we set
			$\tau_*=+\infty$.
			
			Finally, $h_0^{(3)}>0$. Indeed, since ${\rm H}_{y_0}$ and
			${\rm K}_{y_0}$ are continuous on the compact set $\bar\omega$,
			the coefficients of $q$ are uniformly bounded. Hence there exists
			$\tau_0>0$ such that
			$
			q(\tau,x')<0
			$
			for all $x'\in\bar\omega$ and all $0\le\tau\le\tau_0$.
			Consequently,
			$
			\tau_*(x')\ge\tau_0
			\ \text{for every }x'\in\bar\omega,
			$
			and therefore
			$
			h_0^{(3)}\ge\sqrt{\tau_0}>0.
			$
		\end{proof}
		
		\begin{proposition}
			\label{cor-model3-vol}
			Assume \eqref{ch5in} and $0<h<h_0^{(3)}$. Then the function
			\begin{align}
				&(a,b,c)
				\mapsto
				W_{\log}^{B}(x',a,b,c)
				+
				W_{\rm vol}^{T}
				\bigl(
				a,\,
				b-{\rm H}_{y_0}a,\,
				c-{\rm K}_{y_0}a
				\bigr)
				\label{model3-total-vol}
			\end{align}
			is proper, convex and lower semicontinuous on $\mathbb R^3$
			for every $x'\in\bar\omega$. Moreover, its effective domain
			contains $\mathcal M_h$.
		\end{proposition}
		
		\begin{proof}
			By Lemma \ref{convex3-l3},
			$W_{\log}^{B}(x',\cdot)$ is proper, convex and lower
			semicontinuous on $\mathbb R^3$, and its effective domain
			contains $\mathcal M_h$. By Lemma \ref{convex4-l4},
			$W_{\rm vol}^{T}$ is a continuous convex quadratic function in
			$
			\left(
			\frac{a}{{\rm a}_{y_0}},
			\frac{b-{\rm H}_{y_0}a}{{\rm a}_{y_0}},
			\frac{c-{\rm K}_{y_0}a}{{\rm a}_{y_0}}
			\right).
			$
			For fixed $x'$, this is a linear transformation of $(a,b,c)$.
			Hence the composition of $W_{\rm vol}^{T}$ with this
			transformation is continuous and convex on $\mathbb R^3$.
			Therefore, its sum with $W_{\log}^{B}$ is convex and lower
			semicontinuous on $\mathbb R^3$. Since both terms are finite on
			$\mathcal M_h$, the sum is proper and its effective domain
			contains $\mathcal M_h$.
		\end{proof}

		Condition \eqref{det-condition} is a restriction for the
		convexity of the Taylor-expanded quadratic volumetric contribution.
		For a spherical reference midsurface,
		$\kappa_{1,y_0}=\kappa_{2,y_0}=\kappa$, it reduces to
		$
		h|\kappa|\le\frac{2}{\sqrt3}.
		$
		For a cylindrical reference midsurface,
		$\{\kappa_{1,y_0},\kappa_{2,y_0}\}=\{\kappa,0\}$, it gives
		$
		h|\kappa|\le\frac{4}{\sqrt3},
		$
		which is automatically satisfied under \eqref{ch5in}. In the flat
		case ${\rm H}_{y_0}={\rm K}_{y_0}=0$, one has
		$h_0^{(3)}=+\infty$.

		\begin{remark}
			\label{rem-model3-h3}
			If the terms of order $h^5$ are discarded from
			$W_{\rm vol}^{T}$, then, for the scaled Hessian
			$(4/\lambda)\nabla^2\Xi$, the $YY$-entry is zero whereas the
			$rY$-entry equals $h^3/6$. Thus the corresponding
			$2\times2$ principal minor is
			$
			-\frac{h^6}{36}<0,
			$
			and the truncated function is not convex for any $h>0$.
		\end{remark}

		\subsection{The existence results for the proposed models}
		\label{sec3-subsec3}
		
		We are now able to establish existence results for the proposed models.
		We first verify explicitly the coercivity  and the orientation-preserving
		behaviour required in Theorem \ref{general-existence}.
		
		\begin{lemma}
			\label{lem-coercivity}
			Assume \eqref{ch5in}.
			\begin{itemize}
				\item[(i)] For $k=1,2$, there exist constants $c_{1,k}>0$ and
				$C_{1,k}\in\mathbb R$ such that
				\begin{align}
					\label{coercivity-shell}
					{\rm W}^{(k)}_{\rm shell}({\rm I}_m,{\rm II}_m,{\rm III}_m)
					\ge
					c_{1,k}\bigl(|\nabla m|^2+|\nabla n_m|^2\bigr)+C_{1,k}
				\end{align}
				for every $m\in V^h$ and almost every $x'\in\omega$.
				
				\item[(ii)] There exist constants $c_{2,B},c_{2,S}>0$ and
				$C_{2,B},C_{2,S}\in\mathbb R$ such that
				\begin{align}
					\label{coercivity-volB}
					W_{\rm vol}^{B}
					\ge
					c_{2,B}\Big[{\rm a}_m^2+({\rm a}_m{\rm H}_m)^2
					+({\rm a}_m{\rm K}_m)^2\Big]+C_{2,B},
				\end{align}
				and
				\begin{align}
					\label{coercivity-volS}
					W_{\rm vol}^{S}
					\ge
					c_{2,S}\Big[{\rm a}_m^2+({\rm a}_m{\rm H}_m)^2
					+({\rm a}_m{\rm K}_m)^2\Big]+C_{2,S}.
				\end{align}
				If, in addition, $0<h<h_0^{(3)}$, there exist constants
				$c_3>0$ and $C_3\in\mathbb R$ such that
				\begin{align}
					\label{coercivity-vol3}
					W_{\log}^{B}
					+
					W_{\rm vol}^{T}
					\ge
					c_3\Big[{\rm a}_m^2+({\rm a}_m{\rm H}_m)^2
					+({\rm a}_m{\rm K}_m)^2\Big]+C_3,
				\end{align}
				where the two terms on the left-hand side are evaluated at the
				arguments occurring in Model~III.
				
				\item[(iii)] For ${\mathcal Q}\in\{B,S\}$ there exist {lower semicontinuous} functions
				$
				\Psi_{0,{\mathcal Q}},\Psi_{+,{\mathcal Q}},\Psi_{-,{\mathcal Q}}:
				(0,\infty)\to\mathbb R,
				$
				bounded from below and satisfying
				$
				\Psi_{0,{\mathcal Q}}(t),\ \Psi_{+,{\mathcal Q}}(t),\
				\Psi_{-,{\mathcal Q}}(t) \to +\infty
				\ \text{as }t\to0^+,
				$
				and a constant $C_{\mathcal Q}\in\mathbb R$ such that
				\begin{align}
					\label{barrier-volBS}
					W_{\rm vol}^{{\mathcal Q}}
					\ge&
					\Psi_{0,{\mathcal Q}}({\rm a}_m)
					+
					\Psi_{-,{\mathcal Q}}({\rm a}_m{\rm A}_m^-)
					+
					\Psi_{+,{\mathcal Q}}({\rm a}_m{\rm A}_m^+)
					+C_{\mathcal Q}.
				\end{align}
				If $0<h<h_0^{(3)}$, there exist functions
				$\Psi_{0,3},\Psi_{+,3},\Psi_{-,3}$ with the same properties and a
				constant $C_3'\in\mathbb R$ such that
				\begin{align}
					\label{barrier-vol3}
					W_{\log}^{B}+W_{\rm vol}^{T}
					\ge&
					\Psi_{0,3}({\rm a}_m)
					+
					\Psi_{-,3}({\rm a}_m{\rm A}_m^-)
					+
					\Psi_{+,3}({\rm a}_m{\rm A}_m^+)
					+C_3'.
				\end{align}
			\end{itemize}
		\end{lemma}
		
		\begin{proof}
			(i) By Proposition \ref{prop-no-restriction}, the inequalities in
			\eqref{criterion} are strict. We next make the corresponding coercivity
			constants uniform in $x'$, without requiring a global labeling of the
			principal directions. For Model~I, using \eqref{bounds-model1} and the
			bounds $\mathfrak a\le31/15$ and $\mathfrak c\le11/60$ for the scaled constants used in the proof of Proposition \ref{prop-no-restriction} on $[-2,2]^2$,
			we obtain
			\begin{align}
				A_i+C_i
				\le
				h\left(\frac{31}{15}+\frac{11}{60}h^2\right)
			\end{align}
			and therefore, by \eqref{coercivity-Q},
			\begin{align}
				\label{Lambda1-uniform}
				\min_{i=1,2}\frac{4\,A_i\,C_i-B_i^2}{4(A_i+C_i)} \geq
				\frac{(130-32\sqrt2)h^3}
				{1620\left(\frac{31}{15}+\frac{11}{60}h^2\right)}>0.
			\end{align}
			For Model~II,
			$A_i\le\frac53h$, $C_i=h^3/12$, and
			$4A_iC_i-B_i^2\ge h^4/9$, hence
			\begin{align}
				\label{Lambda2-uniform}
				\min_{i=1,2}\frac{4\,A_i\,C_i-B_i^2}{4(A_i+C_i)}\geq\frac{h^3}{3(20+h^2)}>0.
			\end{align}
			Consequently, for $k=1,2$, there exist positive constants $\Lambda_k$ such that 
			\begin{align}
				\mathcal W^{(k)}(E,G)
				\ge
				\Lambda_k
				\left(
				|E{\rm I}_{y_0}^{-1/2}|^2
				+
				|G{\rm I}_{y_0}^{-1/2}|^2
				\right).
			\end{align}
			Since
			$
			|E\, {\rm I}_{y_0}^{-1/2}|^2
			\ge
			\min_{\bar\omega}\lambda_{\min}({\rm I}_{y_0}^{-1})\,|E|^2 
			$ and $
			|G\, {\rm I}_{y_0}^{-1/2}|^2
			\ge
			\min_{\bar\omega}\lambda_{\min}({\rm I}_{y_0}^{-1})\,|G|^2 
			$
			we obtain \eqref{coercivity-shell}. 	
			
			(ii) In order to better see the structure, let us consider
			$
			(a,b,c)
			:=
			({\rm a}_m,{\rm a}_m{\rm H}_m,{\rm a}_m{\rm K}_m).
			$
			We define the linear map
			\begin{align}
				\label{Th-map}
				T_h(a,b,c)
				:=
				\left(
				a,\,
				a+hb+\frac{h^2}{4}c,\,
				a-hb+\frac{h^2}{4}c
				\right).
			\end{align}
			Since $T_h$ is invertible, there
			exist constants $c_h,C_h>0$ such that
			\begin{align}
				\label{Th-equivalence}
				c_h(a^2+b^2+c^2)
				\le
				|T_h(a,b,c)|^2
				\le
				C_h(a^2+b^2+c^2).
			\end{align}
			
			We first consider Models~I and II. Both Boole's and Simpson's rules
			contain the nodes $-1/2$, $0$ and $1/2$ with strictly positive
			weights. Moreover, by \eqref{Aquad-positive}, the corresponding
			reference factors are uniformly positive on $\bar\omega$. Since
			\begin{align}
				\label{HCG-lower}
				\mathcal H_{\rm CG}(z)
				\ge
				\alpha z^2-C_\alpha,
				\qquad z>0,
			\end{align}
			for suitable $\alpha>0$ and $C_\alpha\in\mathbb R$, the three nodal
			contributions at $-1/2$, $0$ and $1/2$ yield
			\begin{align}
				W_{\rm vol}^{B}
				&\ge
				c_B\Big[
				a^2+
				\left(a+hb+\frac{h^2}{4}c\right)^2+
				\left(a-hb+\frac{h^2}{4}c\right)^2
				\Big]-C_B,
			\end{align}
			and, analogously,
			\begin{align}
				W_{\rm vol}^{S}
				&\ge
				c_S\Big[
				a^2+
				\left(a+hb+\frac{h^2}{4}c\right)^2+
				\left(a-hb+\frac{h^2}{4}c\right)^2
				\Big]-C_S.
			\end{align}
			The two additional Boole contributions are bounded from below and are
			absorbed into the constant. Combining these estimates with
			\eqref{Th-equivalence} gives \eqref{coercivity-volB} and
			\eqref{coercivity-volS}.
			
			We next consider Model~III and assume $h<h_0^{(3)}$. By Lemma
			\ref{convex4-l4}, the Hessian of $\Xi$ is uniformly positive definite
			on $\bar\omega$. Hence there exists $\varepsilon_0>0$ such that
			$
			\Xi(r,X,Y)
			\ge
			\varepsilon_0(r^2+X^2+Y^2)
			$
			for all $(r,X,Y)\in\mathbb R^3$ and all $x'\in\bar\omega$.
			The linear map
			$
			(a,b,c)
			\mapsto
			\left(
			\frac{a}{{\rm a}_{y_0}},
			\frac{b-{\rm H}_{y_0}a}{{\rm a}_{y_0}},
			\frac{c-{\rm K}_{y_0}a}{{\rm a}_{y_0}}
			\right)
			$
			is invertible, with coefficients and inverse coefficients uniformly
			bounded on $\bar\omega$. Therefore, for suitable constants
			$\widehat c_h>0$ and $\widehat C_h\in\mathbb R$,
			\begin{align}
				\label{coercivity-Taylor-existence}
				W_{\rm vol}^{T}
				\ge
				\widehat c_h(a^2+b^2+c^2)-\widehat C_h,
			\end{align}
			where $W_{\rm vol}^{T}$ is evaluated at
			$(a,b-{\rm H}_{y_0}a,c-{\rm K}_{y_0}a)$.
			
			The five arguments of the logarithms in $W_{\log}^{B}$ are affine
			functions $\ell_{\zeta_j}(a,b,c)$, with
			$
			\zeta_j\in
			\left\{-\frac12,-\frac14,0,\frac14,\frac12\right\}.
			$
			For every $\varepsilon>0$ there exists $C_\varepsilon>0$ such that
			$
			-\log z\ge-\varepsilon z^2-C_\varepsilon,
			\  z>0.
			$
			Since the reference denominators are uniformly bounded above and below
			and there are only finitely many Boole nodes, there exists $C_B>0$ such
			that the sum of the corresponding quadratic negative parts is bounded
			below by
			$
			-\varepsilon\, C_B(a^2+b^2+c^2)-C_\varepsilon.
			$
			Choosing $\varepsilon>0$ sufficiently small, these terms are absorbed
			by \eqref{coercivity-Taylor-existence}, which proves
			\eqref{coercivity-vol3}.
			
			(iii) For Models~I and II, the nodes $0$ and $\pm1/2$ occur with
			strictly positive weights. Their deformed nodal quantities are
			respectively
			$
			{\rm a}_m,
			$ $
			{\rm a}_m{\rm A}_m^-,
			$ $
			{\rm a}_m{\rm A}_m^+.
			$
			Since $\mathcal H_{\rm CG}$ is bounded from below on $(0,\infty)$ and
			$
			\mathcal H_{\rm CG}(z) \to +\infty
			\ \text{as }z\to0^+,
			$
			while the corresponding reference factors are uniformly bounded above
			and below by positive constants, these three contributions provide the
			functions in \eqref{barrier-volBS}. All remaining quadrature terms are
			bounded from below and can be absorbed into the constant.
			
			For Model~III, by \eqref{Th-equivalence} and the positive quadratic estimate
			\eqref{coercivity-Taylor-existence}, a fixed part of the Taylor
			contribution controls ${\rm a}_m^2+({\rm a}_m{\rm A}_m^+)^2+({\rm a}_m{\rm A}_m^-)^2$. We keep unchanged the three
			logarithmic Boole contributions corresponding to the nodes
			$0,-1/2,1/2$ and estimate only the two interior contributions at
			$\pm1/4$ by
			$
			-\log z\ge-\varepsilon z^2-C_\varepsilon.
			$
			For $\varepsilon>0$ sufficiently small, the latter are absorbed by the
			remaining part of the quadratic Taylor contribution. Hence, for some
			$d>0$ and $C\in\mathbb R$,
			\begin{align}
				W_{\log}^{B}+W_{\rm vol}^{T}
				\ge
				d({\rm a}_m^2+({\rm a}_m{\rm A}_m^+)^2+({\rm a}_m{\rm A}_m^-)^2)
				-
				c_1(x')\log {\rm a}_m-
				c_2(x')\log {\rm a}_m{\rm A}_m^+-c_3(x')\log {\rm a}_m{\rm A}_m^--C,
			\end{align}
			where
			$
			0<\underline c_i\le c_i(x')\le\overline c_i<\infty.
			$
			Consequently, one may take
			\begin{align}
				\Psi_{j,3}(t)
				=
				\begin{cases}
					d\,t^2-\underline c_j\log t,&0<t\le1,\\[1mm]
					d\,t^2-\overline c_j\log t,&t\ge1,
				\end{cases}
				\qquad j\in\{0,-,+\}.
			\end{align}
			These functions are bounded from below and diverge to $+\infty$ as
			$t\to0^+$, which proves \eqref{barrier-vol3}.
		\end{proof}
		
		\begin{theorem}
			\label{existence1}
			Let $\omega\subset\mathbb R^2$ be an open bounded domain with Lipschitz
			boundary, and let $\gamma_{y_0}\subset\partial\omega$ be a nonempty
			relatively open subset. Let $y_0\in C^2(\bar\omega;\mathbb R^3)$ be
			such that $\partial_{x_1}y_0(x')$ and $\partial_{x_2}y_0(x')$ are
			linearly independent for every $x'\in\bar\omega$. Then, for every
			$h>0$ satisfying \eqref{ch5in}, the functional $\mathcal J_1$ defined
			in \eqref{model1} has a minimizer in $V^h$.
		\end{theorem}
		
		\begin{proof}
			We apply Theorem \ref{general-existence} with $p=q=2$. By
			\eqref{convex1-e1} and Proposition \ref{prop-no-restriction}, the
			function
			$
			(A,B)\mapsto W_{\rm shell}^{(1)}(x',A,B)
			$
			is continuous and convex. By Lemma \ref{convex3-l3}, using the
			Boole nodes and weights, the function
			$
			(a,b,c)\mapsto W_{\rm vol}^{B}(x',a,b,c)
			$
			is proper, convex and lower semicontinuous on $\mathbb R^3$,
			and its effective domain contains $\mathcal M_h$.
			
			The density entering $\mathcal J_1$ contains in addition the positive
			referential area factor
			${\rm a}_{y_0}=\det(\nabla y_0|n_{y_0})$. We therefore define
			\begin{align}
				\label{W-hat-model1}
				\widehat{\mathbb W}_1(x',A,B,a,b,c)
				:=
				{\rm a}_{y_0}(x')
				\Big[
				W_{\rm shell}^{(1)}(x',A,B)
				+
				W_{\rm vol}^{B}(x',a,b,c)
				\Big].
			\end{align}
			Since ${\rm a}_{y_0}$ is continuous and strictly positive on
			$\bar\omega$, $\widehat{\mathbb W}_1(x',\cdot)$ is proper,
			convex and lower semicontinuous on
			$
			(\mathbb R^{3\times2})^2\times\mathbb R^3
			$
			for almost every $x'$, and its effective domain contains
			$\mathbf M$. Moreover,
			$
			0<\underline a_0
			\le
			{\rm a}_{y_0}(x')
			\le
			\overline a_0<\infty
			$ on  $\bar\omega.
			$
			Hence \eqref{coercivity-shell} and \eqref{coercivity-volB} give the
			coercivity required in Theorem \ref{general-existence}, while
			\eqref{barrier-volBS} with ${\mathcal Q}=B$ gives the
			orientation-preserving condition.
			
			Measurability follows from the continuity of the reference
			coefficients in $x'$ and the Borel measurability of the
			extended-valued scalar functions. The loading functional
			$\mathcal L$ is continuous and affine on
			$
			{\rm W}^{1,2}(\omega;\mathbb R^3)
			\times
			{\rm W}^{1,2}(\omega;\mathbb R^3).
			$
			Finally, $m=y_0\in V^h$ and $\mathcal J_1(y_0)<+\infty$, so that the
			infimum is finite. All assumptions of Theorem
			\ref{general-existence} are satisfied, and the conclusion follows.
		\end{proof}

		\begin{theorem}
			\label{existence2}
			Under the assumptions of Theorem \ref{existence1}, for every $h>0$
			satisfying \eqref{ch5in}, the functional $\mathcal J_2$ defined in
			\eqref{model2} has a minimizer in $V^h$.
		\end{theorem}
		
		\begin{proof}
			We again apply Theorem \ref{general-existence} with $p=q=2$.
			By Proposition \ref{prop-no-restriction},
			$
			(A,B)\mapsto W_{\rm shell}^{(2)}(x',A,B)
			$
			is continuous and convex. By Lemma \ref{convex3-l3}, using the
			Simpson nodes and weights,
			$
			(a,b,c)\mapsto W_{\rm vol}^{S}(x',a,b,c)
			$
			is proper, convex and lower semicontinuous on $\mathbb R^3$,
			and its effective domain contains $\mathcal M_h$.
			
			We define
			\begin{align}
				\label{W-hat-model2}
				\widehat{\mathbb W}_2(x',A,B,a,b,c)
				:=
				{\rm a}_{y_0}(x')
				\Big[
				W_{\rm shell}^{(2)}(x',A,B)
				+
				W_{\rm vol}^{S}(x',a,b,c)
				\Big].
			\end{align}
			Since ${\rm a}_{y_0}$ is continuous and strictly positive on
			$\bar\omega$, $\widehat{\mathbb W}_2(x',\cdot)$ is proper,
			convex and lower semicontinuous on
			$
			(\mathbb R^{3\times2})^2\times\mathbb R^3
			$
			for almost every $x'$, and its effective domain contains
			$\mathbf M$.
			
			The estimates \eqref{coercivity-shell} and
			\eqref{coercivity-volS} provide the required coercivity, while
			\eqref{barrier-volBS} with ${\mathcal Q}=S$ provides the
			orientation-preserving condition. Measurability, continuity of the
			loading functional and finiteness of the infimum are treated as
			in Theorem \ref{existence1}. Hence Theorem
			\ref{general-existence} applies and $\mathcal J_2$ admits a
			minimizer in $V^h$.
		\end{proof}
		
		\begin{theorem}
			\label{existence3}
			Under the assumptions of Theorem \ref{existence1}, let $h>0$ satisfy
			\eqref{ch5in} and
			$
			h<h_0^{(3)},
			$
			where $h_0^{(3)}$ is defined in \eqref{h0-model3}. Then the functional
			$\mathcal J_3$ defined in \eqref{model3} has a minimizer in $V^h$.
		\end{theorem}
		
		\begin{proof}
			We apply Theorem \ref{general-existence} with $p=q=2$.
			By Proposition \ref{prop-no-restriction},
			$
			(A,B)\mapsto W_{\rm shell}^{(1)}(x',A,B)
			$
			is continuous and convex. By Proposition \ref{cor-model3-vol}, the
			function
			$
			(a,b,c)\mapsto
			W_{\log}^{B}(x',a,b,c)
			+
			W_{\rm vol}^{T}
			\bigl(
			a,b-{\rm H}_{y_0}a,c-{\rm K}_{y_0}a
			\bigr)
			$
			is proper, convex and lower semicontinuous on $\mathbb R^3$,
			and its effective domain contains $\mathcal M_h$.
			
			Consequently, we set
			\begin{align}
				\label{W-hat-model3}
				\widehat{\mathbb W}_3(x',A,B,a,b,c)
				:=
				{\rm a}_{y_0}(x')
				\Big[
				&W_{\rm shell}^{(1)}(x',A,B)
				+
				W_{\log}^{B}(x',a,b,c)
				\notag\\
				&+
				W_{\rm vol}^{T}
				\bigl(
				a,b-{\rm H}_{y_0}a,c-{\rm K}_{y_0}a
				\bigr)
				\Big].
			\end{align}
			For the geometric state
			$
			(a,b,c)
			=
			({\rm a}_m,{\rm a}_m{\rm H}_m,{\rm a}_m{\rm K}_m),
			$
			we have
			$
			b-{\rm H}_{y_0}a
			=
			{\rm a}_m\Delta{\rm H}_m,
			$ $
			c-{\rm K}_{y_0}a
			=
			{\rm a}_m\Delta{\rm K}_m.
			$
			Hence $\widehat{\mathbb W}_3(x',\cdot)$ is proper, convex and
			lower semicontinuous on
			$
			(\mathbb R^{3\times2})^2\times\mathbb R^3
			$
			for almost every $x'$, and its effective domain contains
			$\mathbf M$.
			
			Since ${\rm a}_{y_0}$ is positive and bounded above and below on
			$\bar\omega$, \eqref{coercivity-shell} and
			\eqref{coercivity-vol3} yield the required coercivity, while
			\eqref{barrier-vol3} gives the orientation-preserving condition.
			Measurability follows from the continuity of the reference
			coefficients in $x'$ and the Borel measurability of the
			extended-valued scalar functions. The loading functional and the
			finiteness of the infimum are handled as in Theorem
			\ref{existence1}. Thus all assumptions of Theorem
			\ref{general-existence} are satisfied, and the conclusion follows.
		\end{proof}

		\begin{remark}
			\label{finalremark}
			\begin{itemize}
				\item[(i)] For Models~I and II, the existence theory imposes no
				additional upper bound on the thickness beyond the common local
				geometric regularity condition \eqref{ch5in}. Model~I uses Boole's rule
				for the complete volumetric contribution, whereas Model~II uses
				Simpson's rule. In both cases the positive-weight quadrature structure
				provides the convexity, coercivity and orientation-preserving {condition} required by the direct method.
				
				\item[(ii)] Only Model~III, in which the quadratic volumetric
				contribution $(\det F_\xi)^2$ is treated by its Taylor expansion through
				order $h^5$ instead of by Boole's rule, requires the additional
				condition \eqref{det-condition}. Convexity holds under
				\eqref{det-condition}; for the uniform coercivity required in the
				existence theorem we impose $h<h_0^{(3)}$, under which the Hessian of
				$\Xi$ is uniformly positive definite on $\bar\omega$.
				
				\item[(iii)] If the terms of order $h^5$ are deleted from
				$W_{\rm vol}^{T}$, the resulting quadratic function is not convex
				for any $h>0$, as shown in Remark \ref{rem-model3-h3}.
				
				\item[(iv)] 
				The results remain valid for more general Dirichlet boundary
				conditions
				$
				m|_{\gamma_{y_0}}=\widetilde m,
				$ $
				n_m|_{\gamma_{y_0}}=n_{\widetilde m},
				$
				provided $\widetilde m$ is sufficiently regular and the admissible
				class is nonempty; the proof requires only the corresponding
				modification of the boundary part of Theorem \ref{general-existence}.
			\end{itemize}
		\end{remark}

		\section{The {case of flat plates}}
		\label{sec-flat-plates}
		
		It is useful to outline the particular form of the constructed shell models when
		the reference midsurface is flat. After a rigid motion and a choice of Cartesian coordinates on the
		reference plane, we may take
		\[
		y_0(x_1,x_2)=(x_1,x_2,0).
		\]
		Then
		$
		n_{y_0}=e_3, $ $
		{\rm I}_{y_0}=\id_2,$ $
		{\rm II}_{y_0}={\rm III}_{y_0}=0,$ $
		{\rm L}_{y_0}=0,
		$
		and consequently
		$
		{\rm H}_{y_0}={\rm K}_{y_0}
		=\kappa_{1,y_0}=\kappa_{2,y_0}=0,
		$ $
		{\rm a}_{y_0}=1,
		$ $
		\mathcal A_{y_0}(s)=1.
		$
		In particular, the contractions introduced in
		\eqref{basic-contractions} reduce to
		$
		\mathcal F_0(Q)=\tr Q,
		$ $
		\mathcal F_1(Q)=\mathcal F_2(Q)=0.
		$
		Hence the non-volumetric contributions of Models~I and II coincide and
		take the particularly simple form
		\begin{align}
			\label{plate-shell}
			W_{\rm shell}^{\rm plate}
			&=
			W_{\rm shell}^{(1)}
			=
			W_{\rm shell}^{(2)}
			=
			\frac{\mu}{2}
			\left[
			h\,\tr{\rm I}_m
			+\frac{h^3}{12}\tr{\rm III}_m
			-2h
			\right]
			=
			\frac{\mu h}{2}
			\bigl(|\nabla m|^2-2\bigr)
			+
			\frac{\mu h^3}{24}|\nabla n_m|^2.
		\end{align}
		In fact, in the flat reference case this expression is not an
		asymptotic truncation. Since
		$
		\Theta(x',x_3)=(x_1,x_2,x_3),
		$
		the Kirchhoff--Love ansatz gives
		$
		|\nabla_\xi\varphi_\xi|^2
		=
		|\nabla m+x_3\nabla n_m|^2+1,
		$
		and the term which is odd in $x_3$ vanishes after integration over the
		symmetric thickness interval. Thus \eqref{plate-shell} is the 
		through-thickness reduction of the quadratic term
		$\frac{\mu}{2}(|F_\xi|^2-3)$.
		
		The volumetric part {is also simplified} considerably. We have
		\[
		\det F_\xi(x',hs)
		=
		{\rm a}_m-2hs\,{\rm a}_m{\rm H}_m+h^2s^2{\rm a}_m{\rm K}_m
		=:\ell_s({\rm a}_m,{\rm a}_m{\rm H}_m,{\rm a}_m{\rm K}_m).
		\]
		Accordingly, Models~I and II contain, respectively,
		\[
		W_{\rm vol}^{B}
		=
		h\sum_{j\in J_B}\omega_j^B\,
		\overline{\mathcal H}_{\rm CG}
		\bigl(\ell_{\zeta_j^B}({\rm a}_m,{\rm a}_m{\rm H}_m,{\rm a}_m{\rm K}_m)\bigr)
		\]
		and
		\[
		W_{\rm vol}^{S}
		=
		h\sum_{j\in J_S}\omega_j^S\,
		\overline{\mathcal H}_{\rm CG}
		\bigl(\ell_{\zeta_j^S}({\rm a}_m,{\rm a}_m{\rm H}_m,{\rm a}_m{\rm K}_m)\bigr).
		\]
		
		For Model~III, the Taylor-expanded quadratic volumetric contribution
		becomes
		\begin{align}
			\label{plate-volT}
			W_{\rm vol}^{T}({\rm a}_m,{\rm a}_m{\rm H}_m,{\rm a}_m{\rm K}_m)
			=
			\frac{\lambda}{4}
			\left[
			h {\rm a}_m^2
			+\frac{h^3}{12}\bigl(4({\rm a}_m{\rm H}_m)^2+2{\rm a}_m({\rm a}_m{\rm K}_m)\bigr)
			+\frac{h^5}{80}({\rm a}_m{\rm K}_m)^2
			-h
			\right].
		\end{align}
		The quadratic form in \eqref{plate-volT} admits the representation
		\begin{align}
			\label{plate-volT-square}
			W_{\rm vol}^{T}({\rm a}_m,{\rm a}_m{\rm H}_m,{\rm a}_m{\rm K}_m)
			=
			\frac{\lambda}{4}
			\left[
			h\left({\rm a}_m+\frac{h^2}{12}{\rm a}_m{\rm K}_m\right)^2
			+\frac{h^3}{3}({\rm a}_m{\rm H}_m)^2
			+\frac{h^5}{180}({\rm a}_m{\rm K}_m)^2
			-h
			\right],
		\end{align}
		which shows directly that its quadratic part is positive definite for
		every $h>0$.
		
		There is, moreover, a special coincidence between Models~I and III.
		Indeed,
		\[
		\mathcal H_{\rm CG}(t)
		=
		\frac{\lambda}{4}(t^2-1)
		-\frac{\lambda+2\mu}{2}\log t,
		\]
		whereas $s\mapsto\ell_s(a,b,c)^2$ is a polynomial of degree four.
		Since Boole's rule is exact for polynomials of degree at most five,
		\[
		h\sum_{j\in J_B}\omega_j^B
		\ell_{\zeta_j^B}({\rm a}_m,{\rm a}_m{\rm H}_m,{\rm a}_m{\rm K}_m)^2
		=
		\int_{-h/2}^{h/2}
		\bigl({\rm a}_m-2x_3{\rm a}_m{\rm H}_m+x_3^2{\rm a}_m{\rm K}_m\bigr)^2\,{\rm d}x_3.
		\]
		The right-hand side is precisely
		$
		h\, {\rm a}_m^2
		+\frac{h^3}{12}(4({\rm a}_m{\rm H}_m)^2+2{\rm a}_m({\rm a}_m{\rm K}_m))
		+\frac{h^5}{80}({\rm a}_m{\rm K}_m)^2.
		$
		Consequently,
		$
		W_{\rm vol}^{B}
		=
		W_{\log}^{B}+W_{\rm vol}^{T},
		$
		and therefore
		\[
		\mathcal J_1=\mathcal J_3
		\qquad\text{for a flat reference midsurface}.
		\]
		Thus, in the plate case the three constructions reduce effectively to
		two models: the higher-order Boole model, represented equivalently by
		Models~I and III, and the lower-order Simpson model, Model~II.

		The existence theory takes a particularly simple form for a flat
		reference configuration. Indeed, since
		$
		\kappa_{1,y_0}=\kappa_{2,y_0}=0,
		$
		the geometric thickness condition \eqref{ch5in} is automatically
		satisfied for every $h>0$. Moreover, $h_0^{(3)}=+\infty$, so the
		additional thickness restriction required for Model~III in the curved
		case disappears completely. We may therefore state the following
		existence results.
		
		\begin{theorem}[Existence for the higher-order flat plate model]
			\label{thm-existence-flat-13}
			Let $\omega\subset\mathbb R^2$ be an open bounded domain with Lipschitz
			boundary and let $\gamma_{y_0}\subset\partial\omega$ be a nonempty
			relatively open subset. Assume that, after a rigid motion and a choice
			of Cartesian coordinates on the reference plane,
			$
			y_0(x_1,x_2)=(x_1,x_2,0).
			$
			Then, for every $h>0$, the functional $\mathcal J_1$ defined in
			\eqref{model1} admits at least one minimizer in $V^h$.
		\end{theorem}
		
		\begin{proof}
			Since the principal curvatures of the reference midsurface vanish,
			condition \eqref{ch5in} holds for every $h>0$. Hence
			Theorem~\ref{existence1} applies and yields a minimizer of
			$\mathcal J_1$ in $V^h$. Moreover, in the flat case
			$h_0^{(3)}=+\infty$, so that Theorem~\ref{existence3} also applies for
			every $h>0$. 
		\end{proof}
		
		\begin{theorem}[Existence for the Simpson flat plate model]
			\label{thm-existence-flat-2}
			Under the assumptions of Theorem~\ref{thm-existence-flat-13}, for every
			$h>0$, the functional $\mathcal J_2$ defined in \eqref{model2} admits
			at least one minimizer in $V^h$.
		\end{theorem}
		
		\begin{proof}
			Again, the flatness of the reference midsurface implies that
			\eqref{ch5in} is satisfied for every $h>0$. The assertion therefore
			follows directly from Theorem~\ref{existence2}.
		\end{proof}
		
		Thus all three reduced plate models admit minimizers for every
		$h>0$. Let us remark, however, that this absence of a thickness restriction
		concerns the reference configuration and the convexity and coercivity
		of the reduced energies. The admissible deformed midsurfaces still
		satisfy
		$
		{\rm a}_m>0,
		$ $
		\frac h2
		\max\{|\kappa_{1,m}|,|\kappa_{2,m}|\}<1
		\ \text{a.e. in }\omega,
		$
		as required in the definition of $V^h$.
	}

	{\,
		\section{Linearization and comparison with classical shell strain measures}
		\label{sec-linearization}
		
		The linearization of the models presented in this work lies beyond the main scope of the present paper. Nevertheless, we would like to highlight several particularities of the corresponding linearized models, as this may help the readers to understand the consequences of discarding too early certain quantities involving the coordinate $x_3$. In our derivation, these terms are retained until the final stage, where transparent approximation rules are applied. Discussing the structural form of the linearized versions of our models is also useful for relating our nonlinear formulation to the model proposed in \cite{anicic1999formulation}. Moreover, it contributes to the discussion concerning which strain measures are appropriate for describing bending effects and variations of curvature \cite{GhibaNeffPartVI,vsilhavycurvature,anicic1999formulation,anicic2002mesure}. Besides these aspects related to elasticity theory, the strain measures used here and the class of models considered in this paper may also be of interest from a purely geometric point of view.
		
		In the linearization process, we represent the total deformation of the midsurface as
		\begin{align}
			m(x_1,x_2)=y_0(x_1,x_2)+v(x_1,x_2),
		\end{align}
		where $v:\omega\to\mathbb R^3$ denotes the infinitesimal displacement of the shell midsurface. The linearized changes of the first and second fundamental forms, see \cite{Ciarlet00}, are given by
		\begin{equation}
			\label{equ11}
			\mathcal G_{\rm Koiter}^{\rm lin}
			:=
			\frac12\big[{\rm I}_m-{\rm I}_{y_0}\big]^{\rm lin}
			=
			\sym\big[(\nabla y_0)^T(\nabla v)\big]
			\in{\rm Sym}(2),
		\end{equation}
		and
		\begin{align}
			\label{equ12}
			\mathcal R_{\rm Koiter}^{\rm lin}
			&:=
			\big[{\rm II}_m-{\rm II}_{y_0}\big]^{\rm lin}
			=
			\left(
			\left\langle
			n_{y_0},
			\partial_{x_\alpha x_\beta}v
			-
			\sum_{\gamma=1,2}
			\Gamma^\gamma_{\alpha\beta}\partial_{x_\gamma}v
			\right\rangle
			\right)_{\alpha\beta}
			\in{\rm Sym}(2).
		\end{align}
		The expression of $\mathcal R_{\rm Koiter}^{\rm lin}$ involves the Christoffel symbols
		$
		\Gamma^\gamma_{\alpha\beta}
		=
		\bigl\langle a^\gamma,\partial_{x_\alpha}a_\beta\bigr\rangle
		=
		-\bigl\langle\partial_{x_\alpha}a^\gamma,a_\beta\bigr\rangle
		=
		\Gamma^\gamma_{\beta\alpha},
		$
		associated with the surface parametrized by $y_0$, while $a_1,a_2,a_3$ denote the columns of $\nabla_x\Theta(x',0)$, and $a^1,a^2,a^3$ denote the rows of $[\nabla_x\Theta(x',0)]^{-1}$. We let ${\rm Sym}(2)$ denote the set of all symmetric $2\times2$ matrices and, for all $X\in\mathbb R^{2\times2}$, we set
		$
		\sym X=\frac12(X^T+X)\in{\rm Sym}(2).
		$
		
		There is a general consensus regarding the choice of the strain tensor used to measure metric changes in linearized shell models. In contrast, when bending effects and curvature variations are considered, several different measures appear in the literature. In \cite{GhibaNeffPartVI}, it is observed that there are many reasons for considering modified versions of the classical Koiter strain measures. It is also worth mentioning that Naghdi and Nordgren \cite{naghdi1963nonlinear} regarded the direct use of differences of the first and second fundamental forms between two shell configurations as mainly based on a heuristic argument.
		
		We emphasize that the linearization of our models  involves the variations of all three fundamental forms associated with the considered family of deformations. Consider the family
		\begin{align}
			\left\{
			m=y_0+\eta v
			\ \middle|\ 
			\eta\in\mathbb R,\;
			v\in C^2(\omega;\mathbb R^3),\;
			y_0+\eta v\ \text{defines a regular surface}
			\right\}.
		\end{align}
		For such a family of deformations of the midsurface, the classical strain measures appearing in Koiter-type shell models correspond to the change of metric and the flexural strain. In this respect, we recall the relations established by Blouza and Le Dret \cite{blouza1994lemme}, see also \cite{anicic2002mesure,GhibaNeffPartVI},
		\begin{align}
			\label{fromdret}
			\frac12\frac{d\,{\rm I}_{y_0+\eta v}}{d\eta}\Big|_{\eta=0}
			&=
			\mathcal G_{\rm Koiter}^{\rm lin},
			\qquad
			\frac{d\,{\rm II}_{y_0+\eta v}}{d\eta}\Big|_{\eta=0}
			=
			\mathcal R_{\rm Koiter}^{\rm lin},
		\end{align}
		and
		\begin{align}
			\label{deriv3ffc}
			\frac{d\,{\rm III}_{y_0+\eta v}}{d\eta}\Big|_{\eta=0}
			=
			2\,\sym\big(
			{\rm L}_{y_0}^T[
			\mathcal R_{\rm Koiter}^{\rm lin}
			-
			\mathcal G_{\rm Koiter}^{\rm lin}{\rm L}_{y_0}]
			\big).
		\end{align}
		Hence, the variation of the third fundamental form  occurs in the linearization of the present nonlinear models. It is important, however, that it does not constitute an additional independent infinitesimal strain measure, since by \eqref{deriv3ffc} it is completely determined by $\mathcal G_{\rm Koiter}^{\rm lin}$ and $\mathcal R_{\rm Koiter}^{\rm lin}$.
		
		The Koiter--Sanders--Budiansky measure
		\[
		\mathcal R_{\rm KSB}^{\rm lin}
		:=
		\mathcal R_{\rm Koiter}^{\rm lin}
		-
		\sym\big[
		\mathcal G_{\rm Koiter}^{\rm lin}{\rm L}_{y_0}
		\big]
		\in{\rm Sym}(2)
		\]
		is  related to the variation of the third fundamental form. This measure was emphasized by Budiansky and Sanders \cite{budiansky1962best} as the measure preferred by Koiter himself. Although this tensor does not represent directly the change of curvature, Koiter referred to it as a ``modified tensor of change of curvature'', which has possibly led to some confusion in the literature. The third fundamental form has recently been interpreted as a pure bending measure in the plate setting by Virga \cite{virga2023pure}. Related recent developments concerning pure bending measures and bending--drilling decompositions can be found in \cite{Sonnet25,Tiwari26}.
		
		Hence, for our models, linearization leads to formulations that are able to capture simultaneously the change of metric and bending effects, the classical flexural strain tensor being present as well.
		
		There is another particular feature of our models. After linearization, the energy terms involving the mean curvature ${\rm H}_m$ and the Gaussian curvature ${\rm K}_m$ lead to expressions depending on the corresponding infinitesimal curvature variations. A minimal requirement for a curvature change tensor is that its linearization should characterize the local variations of both the mean curvature and the Gaussian curvature. Anicic and L\'eger \cite{anicic1999formulation} showed that, for the considered family of deformations, if the incremental curvature tensor
		\[
		\mathcal R_{\rm AL}^{\rm lin}(v)
		:=
		\mathcal R_{\rm Koiter}^{\rm lin}
		-
		2\,\sym[
		\mathcal G_{\rm Koiter}^{\rm lin}{\rm L}_{y_0}]
		\in{\rm Sym}(2)
		\]
		vanishes, then the variations of the curvatures vanish as well, i.e.,
		\begin{align}
			\frac{d{\rm H}_{y_0+\eta v}}{d\eta}\Big|_{\eta=0}=0,
			\qquad
			\frac{d{\rm K}_{y_0+\eta v}}{d\eta}\Big|_{\eta=0}=0.
		\end{align}
		In fact, these variations can be expressed in terms of $\mathcal R_{\rm AL}^{\rm lin}(v)$. It is also worth noting that, based on different arguments, {\v{S}}ilhav\'y \cite{vsilhavycurvature} concluded that the curvature tensor introduced by Anicic and L\'eger is more appropriate as a measure of curvature variations in a linear Kirchhoff--Love shell theory than the Koiter--Sanders--Budiansky bending measure. His approach consists of determining the three-dimensional strain tensor associated with a shear deformation of a shell-like body and subsequently linearizing it with respect to the displacement field and to the distance from the midsurface.
		
		Whereas the existing literature usually supplements the {change of metric} measure with only one additional strain measure, chosen to represent flexure, bending, or curvature change, the present nonlinear models contain these geometric effects simultaneously. At the infinitesimal level, however, the corresponding quantities are ultimately determined by the classical linearized Koiter measures introduced in \eqref{equ11} and \eqref{equ12}.
		
		\subsection{Direct linearization of Models I--III}
		
		We now linearize directly the three two-dimensional models \eqref{model1}, \eqref{model2}, and \eqref{model3}. In particular, we do not first linearize the three-dimensional parent problem and subsequently perform a dimensional reduction. Instead, the already obtained two-dimensional energies are evaluated at
		$
		m=y_0+\eta v
		$
		and expanded with respect to $\eta$ around $\eta=0$. In what follows we restrict the discussion to the quadratic
		linearization of the internal energies. A complete infinitesimal
		loading problem would additionally require a consistent {small load}
		scaling and linearization of $\mathcal L$, which is not pursued here.
		
		The natural reference configuration is stationary for the internal energy of each of the three reduced models. In the quadrature contributions, this follows, in particular, from
		$
		\mathcal H_{\rm CG}'(1)=-\mu,
		$ $
		\mathcal H_{\rm CG}''(1)=\lambda+\mu,
		$
		together with the corresponding first variations of the shell contribution. We denote, {respectively} by $W_{\rm lin}^{(k)}$, $k=1,2,3$, the coefficient of $\eta^2$ in the expansion of the internal energy densities of {Models~I, II, and III}.
		
		In order to give the resulting quadratic energies explicitly, we use the orthonormal principal frame introduced in \eqref{shape-orthonormal}--\eqref{spectral-S}. Let us denote the components of the already introduced Koiter measures in this frame by
		\begin{align}
			\label{lin-principal-components}
			g_{ij}
			&:=
			\left\langle
			{\rm I}_{y_0}^{-1/2}
			\mathcal G_{\rm Koiter}^{\rm lin}
			{\rm I}_{y_0}^{-1/2}w_j,w_i
			\right\rangle,
			\qquad \qquad 
			r_{ij}
			:=
			\left\langle
			{\rm I}_{y_0}^{-1/2}
			\mathcal R_{\rm Koiter}^{\rm lin}
			{\rm I}_{y_0}^{-1/2}w_j,w_i
			\right\rangle,
			\qquad \qquad i,j=1,2.
		\end{align}
		
		As in the convexity analysis, this representation is pointwise and does not require a global choice of principal directions; at umbilical points the resulting quadratic expressions are independent of the chosen orthonormal principal frame.
		Since both tensors are symmetric,
		$
		g_{12}=g_{21},
		$ $
		r_{12}=r_{21}.
		$
		Thus $g_{ij}$ and $r_{ij}$ are only their components in the principal orthonormal frame and do not represent additional strain measures.
		
		The linearized variation of the nodal quantity occurring in all volumetric contributions takes the form
		\begin{align}
			\label{chi-def}
			\chi(s)
			:={}&
			\left.
			\frac{d}{d\eta}
			\left[
			\frac{
				{\rm a}_{y_0+\eta v}\mathcal A_{y_0+\eta v}(s)
			}{
				{\rm a}_{y_0}\mathcal A_{y_0}(s)
			}
			\right]
			\right|_{\eta=0}
			=
			\frac{(1+hs\kappa_{1,y_0})g_{11}-hs\,r_{11}}
			{1-hs\kappa_{1,y_0}}
			+
			\frac{(1+hs\kappa_{2,y_0})g_{22}-hs\,r_{22}}
			{1-hs\kappa_{2,y_0}}.
		\end{align}
		In particular, only the diagonal components of the linearized Koiter measures enter the infinitesimal variation of the nodal volume ratio.
		
		The part produced by the direct linearization of the non-volumetric energy together with the term generated by $\mathcal H_{\rm CG}'(1)=-\mu$ can be written in terms of
		\begin{align}
			\label{Phi0}
			\Phi_0
			=&
			\frac12(g_{11}-g_{22})^2+2g_{12}^2,
			\notag\\
			\Phi_2
			=&
			\frac12\Big[
			4\kappa_{1,y_0}^2g_{11}^2
			-3\kappa_{1,y_0}\kappa_{2,y_0}g_{11}^2
			-2\kappa_{1,y_0}\kappa_{2,y_0}g_{11}g_{22}
			-4\kappa_{1,y_0}g_{11}r_{11}
			+2\kappa_{1,y_0}g_{11}r_{22}
			+2\kappa_{2,y_0}g_{11}r_{11}\notag
			\\
			&\ -3\kappa_{1,y_0}\kappa_{2,y_0}g_{22}^2
			+4\kappa_{2,y_0}^2g_{22}^2
			+2\kappa_{1,y_0}g_{22}r_{22}
			+2\kappa_{2,y_0}g_{22}r_{11}
			-4\kappa_{2,y_0}g_{22}r_{22}
			+(r_{11}-r_{22})^2
			\Big]
			\\
			&\ +
			2\Big[
			(\kappa_{1,y_0}^2-\kappa_{1,y_0}\kappa_{2,y_0}+\kappa_{2,y_0}^2)g_{12}^2
			-(\kappa_{1,y_0}+\kappa_{2,y_0})g_{12}r_{12}
			+r_{12}^2
			\Big],\notag
			\\
			\Phi_4
			=&
			\frac{\kappa_{1,y_0}-\kappa_{2,y_0}}{2}
			\Big[
			4\kappa_{1,y_0}^3g_{11}^2
			-4\kappa_{1,y_0}^2g_{11}r_{11}
			-4\kappa_{2,y_0}^3g_{22}^2
			+4\kappa_{2,y_0}^2g_{22}r_{22}
			+\kappa_{1,y_0}r_{11}^2-\kappa_{2,y_0}r_{22}^2
			\Big]
			\notag\\&\ +
			\frac{(\kappa_{1,y_0}-\kappa_{2,y_0})^2}{2}
			\Big[
			4(\kappa_{1,y_0}^2+\kappa_{1,y_0}\kappa_{2,y_0}+\kappa_{2,y_0}^2)g_{12}^2
			-4(\kappa_{1,y_0}+\kappa_{2,y_0})g_{12}r_{12}
			+r_{12}^2
			\Big].\notag
		\end{align}
		
		For the Taylor-expanded quadratic volumetric term in Model~III, we also use
		\begin{align}
			\label{Psi0}
			\Upsilon_0
			=&
			(g_{11}+g_{22})^2,\notag
			\\
			\Upsilon_2
			=&
			4\kappa_{1,y_0}^2g_{11}^2
			-3\kappa_{1,y_0}\kappa_{2,y_0}g_{11}^2
			+2\kappa_{1,y_0}\kappa_{2,y_0}g_{11}g_{22}
			-4\kappa_{1,y_0}g_{11}r_{11}
			-
			2\kappa_{1,y_0}g_{11}r_{22}
			+2\kappa_{2,y_0}g_{11}r_{11}
			\notag\\
			&\ -3\kappa_{1,y_0}\kappa_{2,y_0}g_{22}^2
			+4\kappa_{2,y_0}^2g_{22}^2
			+
			2\kappa_{1,y_0}g_{22}r_{22}
			-2\kappa_{2,y_0}g_{22}r_{11}
			-4\kappa_{2,y_0}g_{22}r_{22}
			+(r_{11}+r_{22})^2,
			\\
			\Upsilon_4
			={}&
			(\kappa_{1,y_0}-\kappa_{2,y_0})
			\Big[
			4\kappa_{1,y_0}^3g_{11}^2
			-4\kappa_{1,y_0}^2g_{11}r_{11}
			-4\kappa_{2,y_0}^3g_{22}^2
			+4\kappa_{2,y_0}^2g_{22}r_{22}
			+\kappa_{1,y_0}r_{11}^2-\kappa_{2,y_0}r_{22}^2
			\Big].\notag
		\end{align}
		
		We can now give the three linearized models explicitly.
		
		For Model~I, direct linearization of \eqref{wshell1} and \eqref{wvol1} gives
		\begin{align}
			\label{linear-model-I}
			W_{\rm lin}^{(1)}
			={}&
			\mu\left(
			h\Phi_0+\frac{h^3}{12}\Phi_2+\frac{h^5}{80}\Phi_4
			\right)
			+
			\frac{\lambda+\mu}{2}\,
			h\sum_{j\in J_B}
			\omega_j^B\,
			\mathcal A_{y_0}(\zeta_j^B)\,
			\chi(\zeta_j^B)^2.
		\end{align}
		
		For Model~II, direct linearization of \eqref{wshell2} and \eqref{wvol2} yields
		\begin{align}
			\label{linear-model-II}
			W_{\rm lin}^{(2)}
			={}&
			\mu\left(
			h\Phi_0+\frac{h^3}{12}\Phi_2
			\right)
			+
			\frac{\lambda+\mu}{2}\,
			h\sum_{j\in J_S}
			\omega_j^S\,
			\mathcal A_{y_0}(\zeta_j^S)\,
			\chi(\zeta_j^S)^2.
		\end{align}
		Thus, the lower order of Model~II remains visible after linearization: the contribution involving $\Phi_4$ is absent and the volumetric part contains the three Simpson nodes instead of the five Boole nodes.
		
		Finally, direct linearization of \eqref{wshell1}, \eqref{wcurv3} and \eqref{wcurv4} gives for Model~III
		\begin{align}
			\label{linear-model-III}
			W_{\rm lin}^{(3)}
			={}&
			\mu\left(
			h\Phi_0+\frac{h^3}{12}\Phi_2+\frac{h^5}{80}\Phi_4
			\right)
			+
			\frac{\lambda+2\mu}{4}\,
			h\sum_{j\in J_B}
			\omega_j^B\,
			\mathcal A_{y_0}(\zeta_j^B)\,
			\chi(\zeta_j^B)^2
			+
			\frac{\lambda}{4}
			\left(
			h\Upsilon_0+\frac{h^3}{12}\Upsilon_2+\frac{h^5}{80}\Upsilon_4
			\right).
		\end{align}
		The second line in \eqref{linear-model-III} originates from the logarithmic Boole contribution $W_{\log}^B$, whereas the last line is the contribution produced by the direct linearization of the Taylor-expanded quadratic volumetric energy $W_{\rm vol}^{T}$.
		
		The difference between Models~I and III is therefore transparent at the infinitesimal level. From \eqref{linear-model-I} and \eqref{linear-model-III},
		\begin{align}
			\label{difference-linear-13}
			W_{\rm lin}^{(1)}-W_{\rm lin}^{(3)}
			=
			\frac{\lambda}{4}
			\Bigg[{}&
			h\sum_{j\in J_B}
			\omega_j^B
			\mathcal A_{y_0}(\zeta_j^B)
			\chi(\zeta_j^B)^2
			-
			h\Upsilon_0-\frac{h^3}{12}\Upsilon_2-\frac{h^5}{80}\Upsilon_4
			\Bigg].
		\end{align}
		Thus the difference between the quadratic linearizations of Models~I and III is precisely the difference between the Boole treatment and the Taylor treatment of the quadratic volumetric contribution. In particular, for every fixed sufficiently regular infinitesimal displacement,
		\[
		W_{\rm lin}^{(1)}-W_{\rm lin}^{(3)}=O(h^7).
		\]
		Similarly, Model~II retains the common expansion only through order $h^3$. This is  the hierarchy of the three nonlinear reduced models after their direct linearization.
		
		\subsection{The flat reference configuration}
		
		For a flat reference midsurface, using the Cartesian parametrization introduced in Section~\ref{sec-flat-plates}, the expressions above reduce to
		\[
		\Phi_0
		=
		|\mathcal G_{\rm Koiter}^{\rm lin}|^2
		-\frac12
		\big(\tr\mathcal G_{\rm Koiter}^{\rm lin}\big)^2,
		\qquad
		\Phi_2
		=
		|\mathcal R_{\rm Koiter}^{\rm lin}|^2
		-\frac12
		\big(\tr\mathcal R_{\rm Koiter}^{\rm lin}\big)^2,
		\qquad
		\Phi_4=0,
		\]
		while
		\[
		\Upsilon_0
		=
		\big(\tr\mathcal G_{\rm Koiter}^{\rm lin}\big)^2,
		\qquad
		\Upsilon_2
		=
		\big(\tr\mathcal R_{\rm Koiter}^{\rm lin}\big)^2,
		\qquad
		\Upsilon_4=0.
		\]
		Moreover,
		\begin{align}
			\label{chi-flat}
			\chi(s)
			=
			\tr\mathcal G_{\rm Koiter}^{\rm lin}
			-hs\,\tr\mathcal R_{\rm Koiter}^{\rm lin}.
		\end{align}
		Hence $\chi(s)^2$ is a polynomial of degree two in $s$. Both Simpson's and Boole's rules are exact for this polynomial. It follows directly from \eqref{linear-model-I}--\eqref{linear-model-III} that
		\begin{align}
			\label{same-linear-flat}
			W_{\rm lin}^{(1)}
			=
			W_{\rm lin}^{(2)}
			=
			W_{\rm lin}^{(3)}.
		\end{align}
		Therefore, although the three nonlinear models are generally different, their quadratic linearizations coincide for a flat reference midsurface.
		
		The common linearized energy density is
		\begin{align}
			\label{flat-common-linearization}
			{\rm W}^{\rm lin}
			={}&
			h\left[
			\mu|\mathcal G_{\rm Koiter}^{\rm lin}|^2
			+
			\frac{\lambda}{2}
			\big(\tr\mathcal G_{\rm Koiter}^{\rm lin}\big)^2
			\right]
			+
			\frac{h^3}{12}
			\left[
			\mu|\mathcal R_{\rm Koiter}^{\rm lin}|^2
			+
			\frac{\lambda}{2}
			\big(\tr\mathcal R_{\rm Koiter}^{\rm lin}\big)^2
			\right].
		\end{align}
		Thus the differences between Models~I--III disappear  at the infinitesimal level for plates.
		
		The comparison with the classical Koiter model concerns here  the geometric strain measures. Indeed, the two-dimensional dilatational coefficient in \eqref{flat-common-linearization} remains the three-dimensional Lam\'e coefficient $\lambda$. This is a consequence of the thickness-inextensible Kirchhoff--Love ansatz adopted here, which imposes
		that the linearized normal strain in the thickness direction vanishes. 
		If the transverse normal strain is instead relaxed by imposing the plane-stress condition, minimization with respect {to the linearized normal strain in the thickness direction} replaces $\lambda$ by
		$
		\lambda^*
		=
		\frac{2\lambda\mu}{\lambda+2\mu}.
		$
		Hence the present models retain Koiter-type membrane and bending strain measures, but do not reproduce the classical plane-stress Koiter constitutive tensor.
		
		In conclusion, direct linearization of the already reduced models shows that the effects generated in the nonlinear theory by the first, second and third fundamental forms and by the mean and Gaussian curvatures remain encoded in the resulting quadratic theories. For initially curved shells, Models~I--III lead in general to different linearized energies, with Models~I and III agreeing through order $h^5$ and Model~II through order $h^3$. For a flat reference configuration, all these differences disappear and the three models possess the same quadratic linearization.

	}
	
	\section{Conclusions and final remarks} 
	
	\,{In this work, we have derived nonlinear shell models from the
		three-dimensional nonlinear Neo--Hooke--Ciarlet--Geymonat model. The
		derivation is carried out in a variational framework by combining the
		Kirchhoff--Love kinematical ansatz with an asymptotic dimensional
		reduction through the thickness and suitable positive-weight quadrature
		rules for the volumetric contribution. Boole's rule is used in the
		reductions consistent through order $h^5$, while Simpson's rule is used
		in the reduction consistent through order $h^3$; in Model~III the
		quadratic volumetric contribution is instead Taylor-expanded through
		order $h^5$.
		The resulting two-dimensional energies are expressed in terms of
		geometric quantities of the deformed midsurface. The shell energies
		depend on the first, second, and third fundamental forms, as well as on
		the mean and Gaussian curvatures. These quantities arise  in
		the dimensional reduction process and are not introduced a posteriori.
		In particular, the quadrature of the volumetric Ciarlet--Geymonat
		contribution produces area-weighted nodal quantities including ${\rm a}_m$ and ${\rm a}_m{\rm A}_m^\pm$ at the midpoint and
		the two boundary nodes. In Model~III, the Taylor-expanded quadratic
		volumetric contribution additionally involves
		${\rm a}_m\Delta{\rm H}_m$ and
		${\rm a}_m\Delta{\rm K}_m$.}
	
	A characteristic feature of the proposed models is that the constitutive coefficients are inherited directly from the three-dimensional constitutive parameters, while the influence of the geometry of the reference configuration {also} appears in the expressions of the coefficients of the two-dimensional shell model explicitly through the curvatures of the initial midsurface. This provides a consistent link between three-dimensional nonlinear elasticity and geometrically nonlinear shell theories as it shows that {\,the effective coefficients and the mechanical response of the reduced
		shell model depend not only on the three-dimensional material parameters
		and the thickness, but also explicitly on the geometry of the reference
		midsurface.}
	
	From the analytical point of view, the structure of the resulting functionals allows us to establish the existence of minimizers for the associated variational problems. We consider a polyconvexity concept in the shell theories, and we show that the energies of the proposed models inherit this convex structure compatible with the shell-polyconvexity framework from the three-dimensional {parent} energy. It is also important that the variational problems are written in terms of some quantities for which a compensated compactness argument identifies their weak limits.
	
	\,{A further consequence of the {principal-direction analysis} is that the non-volumetric shell energies of Models I and II remain
		strictly convex in $(\nabla m,\nabla n_m)$ throughout the whole
		geometrically admissible thickness range \eqref{ch5in}. Moreover,
		the volumetric contributions are convex in the area-weighted
		variables identified in Lemma \ref{limit-identification}, which
		yields weak lower semicontinuity of the corresponding integral
		functionals. In particular, no additional upper bound on $h$ is
		required for Models I and II. An additional thickness restriction arises only in Model III and is entirely due to the Taylor-expanded quadratic volumetric contribution. The corresponding condition is explicit and sharp for the convexity of this contribution.}
	
	The models obtained in this paper therefore provide a mathematically consistent class of nonlinear shell theories in which all terms arise directly from the dimensional reduction of a polyconvex three-dimensional energy, without the need for additional ad hoc corrections to enforce coercivity, convexity, or orientation preservation.
	
	{\, In the particular case of a flat reference midsurface, Models~I and III coincide, while all three reduced plate models admit minimizers for every thickness $h>0$. Indeed, the geometric thickness condition on the reference configuration is then automatically satisfied and, for Model~III, the additional convexity threshold becomes $h_0^{(3)}=+\infty$; the minimization is still performed over the admissible class $V^h$, which retains the regularity and curvature conditions on the deformed midsurface. Moreover, the quadratic linearizations of all three models coincide  in the flat case, whereas for initially curved shells they are generally different, with Models~I and III agreeing through order $h^5$ and Model~II retaining the common expansion through order $h^3$.}
	
	\begin{footnotesize}
		\paragraph*{Acknowledgments.}
		
		The work of I.D. Ghiba and C. Ureche was supported by a grant of the Ministry of Education and Research, CNCS--UEFISCDI,
		project number PN-IV-P1-PCE2026-0163, within PNCDI IV.
		The second author is partially supported by the Ministry of Education, Youth, and Sports of the Czech Republic via the project Ferroic Multifunctionalities (FerrMion), no. CZ.02.01.01/00/22\_008/0004591, co-funded by the European Union, and by the grant PRIMUS/24/SCI/020 of Charles University. Part of this work was performed when the second author was visiting the Faculty of Mathematics, Alexandru Ioan Cuza University of Ia\c si, Romania, and he would like to thank its members for their hospitality and support. 
		
		\paragraph*{Conflict of interest.}
		The authors declare that they have no conflict of interest.
		
		\paragraph*{Data availability.}
		No datasets were generated or analyzed during the current study.

		\clearpage
		\phantomsection
		\addcontentsline{toc}{section}{References}
		\bibliographystyle{plain} 

	\end{footnotesize}
\end{document}